%% file: arxiv_v3.tex
\documentclass{article}
\usepackage{Arxiv}

\usepackage{hyperref}       % hyperlinks
\usepackage{lipsum}
\usepackage{amsfonts}
\usepackage{algorithm}
\usepackage{algorithmic}
\usepackage{listings}
\usepackage{tikz}
\usepackage{tcolorbox}
\usepackage{bbding}
\usepackage{amsopn}
\usepackage{subcaption}
\usepackage{multirow}
\usepackage{mathtools}
\usepackage{breqn}
\usepackage{todonotes}
\usepackage{etoolbox}
\makeatletter
\patchcmd{\@addmarginpar}{\ifodd\c@page}{\ifodd\c@page\@tempcnta\m@ne}{}{}
\makeatother
\reversemarginpar
\input{macro}
\ifpdf
  \DeclareGraphicsExtensions{.eps,.pdf,.png,.jpg}
\else
  \DeclareGraphicsExtensions{.eps}
\fi

\newcommand{\iprod}[2]{\left \langle #1, #2 \right \rangle }
\def\xe{x_\epsilon}
\def\ye{y_\epsilon}
\def\cG{\mathcal{G}}

\crefname{hypothesis}{Hypothesis}{Hypotheses}
\crefname{lemma}{Lemma}{Lemmas}
\crefname{thm}{Theorem}{Theorems}
\crefname{assumption}{Assumption}{Assumptions}
\crefname{corollary}{Corollary}{Corollaries}
\crefname{remark}{Remark}{Remarks}
\newcommand{\pcal}{\mathcal{P}}
\usepackage[normalem]{ulem}
\newcommand\redsout{\bgroup\markoverwith{\textcolor{red}{\rule[0.5ex]{2pt}{1pt}}}\ULon}

\DeclareMathOperator{\diag}{diag}

\newcommand{\rmnum}[1]{\uppercase\expandafter{\romannumeral #1}}

\def\sa#1{\textcolor{black}{#1}}
\def\hj#1{\textcolor{black}{#1}}
\def\rv#1{\textcolor{black}{#1}}
\def\mod#1{\textcolor{black}{#1}}

\usepackage[utf8]{inputenc} % allow utf-8 input
\usepackage[T1]{fontenc}    % use 8-bit T1 fonts
\usepackage{url}            % simple URL typesetting
\usepackage{booktabs}       % professional-quality tables
\usepackage{amsfonts}       % blackboard math symbols
\usepackage{nicefrac}       % compact symbols for 1/2, etc.
\usepackage{microtype}      % microtypography
\usepackage{xcolor}         % colors

\usepackage[sort]{natbib}
\makeatletter
\newcommand{\thickhline}{%
    \noalign {\ifnum 0=`}\fi \hrule height 1pt
    \futurelet \reserved@a \@xhline
}
\newcolumntype{"}{@{\hskip\tabcolsep\vrule width 1pt\hskip\tabcolsep}}
\makeatother
\allowdisplaybreaks

\title{A Retraction-Free Method for Nonsmooth Minimax Optimization over a Compact Manifold}
\shorttitle{A Retraction-Free Nonsmooth Minimax Optimization}

\author{%
  Necdet Serhat Aybat\thanks{These authors contributed equally to this work.} \\
  Department of Industrial and Manufacturing Engineering\\
  Penn State University\\
  University Park, PA, USA\\
  \texttt{nsa10@psu.edu} \\
  \And
  Jiang Hu\footnotemark[1] \\
  Yau Mathematical Sciences Center\\
  Tsinghua University \\
  Beijing 100084, China \\
  \texttt{hujiangopt@mail.com} \\
  \And
  Zhanwang Deng\thanks{The student coauthor contributed to the numerical experiments.}  \\
  Academy for Advanced Interdisciplinary Studies\\
  Peking University\\
  Beijing 100871, China \\
  \texttt{dzw\_opt2022@stu.pku.edu.cn}
  }

\begin{document}
\maketitle

\begin{abstract}
  We study the minimax optimization problem over \sa{a compact submanifold $\cM$, i.e.,
  $\min_{x\in\cM}\max_y \rv{f_r(x,y):=}f(x,y)\rv{-h(y)}$,} where %the objective $f(x, y)$ 
  $f$ is continuously differentiable \sa{in $(x,y)$}, \rv{$h$ is a closed, weakly-convex (possibly non-smooth) function} and \rv{we assume that the regularized coupling function $-f_r(x,\cdot)$ is either $\mu$-PL for some $\mu>0$ or concave $(\mu=0)$ for any fixed $x$ in the vicinity of $\cM$}. To address the nonconvexity %from 
  \sa{due to} the manifold constraint, we reformulate the problem %via 
  \sa{using} an \textit{exact penalty} %in $x$
  \rv{for the constraint $x\in\cM$ and enforcing a convex constraint $x\in X$ for some $X\supset \cM$ onto which projections can be computed efficiently.} 
  %and introduce a norm constraint to ensure bounded Lipschitz smoothness. 
  Building upon this new formulation \sa{for the manifold minimax problem in question}, a single-loop smoothed \hj{manifold} gradient descent-ascent %(sm-GDA)
  \sa{(\sm)} algorithm is proposed. Theoretically, %sm-GDA
  \rv{any limit point of \sm{} sequence is a stationary point of the manifold minimax problem and} \sm{} \sa{can generate an \rv{$\cO(\epsilon)$-stationary point} of the 
  % penalized 
  \hj{original}
  problem with $\mathcal{O}(\kappa / \epsilon^2)$ and 
  $\rv{\tilde\cO}(l^4/\epsilon^4)$
   complexity for $\mu>0$ and $\mu=0$ scenarios, respectively, where $\kappa=l/\mu$ is the condition number and $l$ denotes the Lipschitz constant of the gradient corresponding to the penalized problem \rv{over $X\times \dom h$}.} \rv{Moreover, for the $\mu=0$ setting, through adopting Tikhonov regularization of the dual, one can improve the complexity to $\cO(l^2/\epsilon^3)$ at the expense of asymptotic stationarity.}
  % Under a mild coercivity condition \sa{on the primal function}, we further prove that any \sa{limit} point is also a stationary point of the \textit{original} problem, 
  \hj{The key component, \rv{common in the analysis of all cases,} is to connect $\epsilon$-stationary points between the penalized problem and the original problem}
  by showing that \rv{the constraint $x\in X$} \sa{becomes inactive and the penalty term tends to $0$ along any convergent subsequence}. To our knowledge, %this
  \sa{\sm} is the first retraction-free algorithm for minimax problems over compact submanifolds, \sa{and this is a very desirable algorithmic property since through avoiding retractions, one can get away with 
  % SVD computations 
  {matrix orthogonalization subroutines} required for computing retractions to manifolds arising in practice, which are not GPU friendly.} 
  %\sa{Compared to \citep{huang2023gradient}, \sa{the only work we are aware of studying a manifold minimax problem with a smooth $f$ without assuming geodesic convexity,} \sm{}  does not require strong-concavity while the results in \citep{huang2023gradient} assumes $\mu>0$ and even for this scenario \sm{} improves the dependence on $\mu$ from $\mathcal{O}(1/\mu^2)$ in \citep{huang2023gradient} to $\mathcal{O}(1/\mu)$.} 
  Experiments on quadratic minimax problems, robust deep neural network training, and superquantile-based learning demonstrate clear advantages over state-of-the-art algorithms \rv{that rely on retraction operation in each iteration.}
\end{abstract}

%\todo[inline]{Things to consider: 1) empty subdifferential at the boundary; 2) Experiment: add $r\approx n_1$ setting to all experiments, and add a new one satisfying PL (possibly nonsmooth); 3) As an example, distance square to closed sets}
\section{Introduction}
Due to the broad applications in machine learning, minimax optimization has attracted %lots of
\rv{significant} attention %, such as 
\rv{arising in the context of} generative adversarial networks (GANs) \citep{goodfellow2014generative}, robust deep neural network training \citep{madry2017towards,kurakin2016adversarial,goodfellow2014explaining}, superquantile-based learning \citep{chen2018robust,curi2020adaptive}, and reinforcement learning \citep{du2017stochastic,dai2018sbeed}. %There are some 
\rv{Minimax problems involving manifold constraints naturally arise in applications} such as robust geometry-aware PCA \citep{horev2016geometry} and \rv{those involving subspace robust Wasserstein distance optimization}~\citep{paty2019subspace,lin2020projection}. Furthermore, adding manifold constraints, such as orthogonality, on the \rv{model} parameters has been found effective in accelerating the training \citep{arjovsky2016unitary,cho2017riemannian,kongmomentum}, %avoiding
\rv{preventing the gradient sequence from} diminishing or exploding, and improving generalization \citep{cogswell2015reducing}. 
% ICA \citep{ijcai2021p345}, dictionary learning \citep{sun2016complete}, neural network trainning \citep{huang2023gradient}, etc.

In this paper, we consider the following %smooth 
minimax optimization problem over \sa{a \textit{compact submanifold} of a normed vector space $(\cX,\norm{\cdot})$ with  $\cX:=\reals^{d_1\times r}$ and $\cY:=\R^{d_2}$}:
\be \label{prob} \min_{x \in \Mcal} \max_{\hjn{y \in \mathcal{Y}}} \rv{f_r(x,y):=f(x,y) \hjn{- h(y)}}, \ee
where $\Mcal:=\{x\in \R^{d_1\times r}: c(x) = 0\}$ for some smooth \rv{function} $c:\reals^{d_1\times r}\to\reals^p$, \hjn{$h:\mathcal{Y} \rightarrow \rv{\R\cup\{+\infty\}}$ is a \rv{proper, closed,} (\rv{possibly nonsmooth}) $\zeta$-weakly convex function\footnote{We say a function $h$ is $\zeta$-weakly convex if $h(\cdot) + \frac{\zeta}{2}\|\cdot\|$ is convex for some $\zeta\geq 0$.} \rv{for some $\zeta\geq 0$} with a \rv{compact, convex} domain $Y:= \{y \in \mathcal{Y}: h(y) < \infty\}$}, and
$f$ is continuously differentiable on some open set containing \rv{$\bar X\times Y$ where $\bar X\supset\cM$ is some compact set that will be defined later.} \rv{In the rest, we used $f_r(\cdot,\cdot)$ to denote $f(\cdot,\cdot)$ regularized with $h(\cdot)$.} %Moreover, we assume that 
%$\nabla f$ is Lipschitz on $X\times Y$ with constant $L>0$ and 

\rv{We study \eqref{prob} assuming that one of the following conditions holds: \textbf{(i)} $-f_r(x,\cdot)$ is $\mu$-PL, i.e., it satisfies Polyak-Lojasiewicz condition 
%with a constant 
for some $\mu>0$, \rv{uniformly for all $x\in \bar X$} 
%X^\circ$ 
%for some $\mu>0$ 
\rv{(see %Assumption~\ref{assum:sc}
Definition~\ref{def:PL})}, 
%for the definition of weakly convex PL functions)}, 
or \textbf{(ii)} $f_r(x,\cdot)$ is concave for all $x\in \bar X$
%X^\circ$ 
(abusing the notation, we abbreviate this scenario by $\mu=0$).} 
\sa{Throughout the paper 
%while stating our results 
we mainly focus on the setting where $\Mcal:=\{x \in \R^{d_1 \times r}: x^\top x = I_r \}$ is the Stiefel manifold, i.e., $c(x)=x^\top x-I_r$, and %at the end 
\rv{later in \cref{sec:extensions} we %briefly 
discuss how our results can be extended to 
%the more general case in~\eqref{prob} for other 
a more general class of smooth defining functions $c(\cdot)$.}} 

Algorithms for solving the problem \eqref{prob} have been extensively studied
%\nsa{Below we should give the literature results for the deterministic case.} 
in the literature \citep{huang2023gradient,xu2024riemannian,xu2024efficient,jordan2022first,cai2023curvature,hu2024extragradient,han2023riemannian}. Among these works, \citep{huang2023gradient,xu2024riemannian,xu2024efficient} are the only works we are aware of studying a manifold minimax problem with a smooth $f$ \textit{without} assuming geodesic convexity, i.e., when $f$ is not geodesically convex in $x$ for \sa{some} fixed \sa{$y\in Y$}. This distinction is important, as \sa{by the Hopf-Rinow theorem,} any geodesically convex function defined on a \textit{compact} \sa{Riemannian} manifold must be constant~\citep[Theorem 6.13]{lee2006riemannian}. Consequently, the results in \citep{jordan2022first,cai2023curvature,hu2024extragradient,han2023riemannian} are not directly applicable to problem \eqref{prob} when $f(\cdot,y)$ is \sa{not constant for fixed $y\in Y$}. 
%Among these works, only \citep{huang2023gradient} %considers
%\sa{allows for} the case where the function $f$ is not geodesically convex in $x$ for \sa{any} fixed \sa{$y\in Y$}. 

\sa{In~\citep{xu2024riemannian,xu2024efficient}, the authors consider nonconvex-linear minimax problems on Riemannian manifolds, and under the linearity assumption in the dual, they establish $\cO(1/\epsilon^3)$ iteration complexity. Neither RADA method in~\citep{xu2024riemannian} nor ARPGDA in~\citep{xu2024efficient} is retraction-free, they require computing a retraction onto the manifold at each iteration -- ARPGDA is a single-loop method and RADA requires inexactly solving smooth-strongly concave minimax subproblems over the manifold at each iteration. \rv{On the other hand, \citep{huang2023gradient} considers nonconvex-strongly concave minimax problems, and it establishes $\cO(1/\epsilon^2)$ complexity for the proposed RGDA method; that said, RGDA is also not retraction-free and it requires computing a retraction onto $\cM$ at each iteration.} {For the main use case considered in this paper, the Stiefel manifold ~\citep{ablin2022fast}, the retraction always involves some
expensive linear algebra operation, such as
matrix inversion, exponential or square-root, which quickly become expensive as the dimension of the matrices grows (especially for the square Stiefel manifold where $d_1 = r$).} Therefore, in this paper we will investigate the following natural question:}

\sa{\textit{Can one design a retraction-free single-loop first-order method \rv{with cheap per-iteration complexity} to efficiently compute stationary points of \rv{nonconvex-PL and} nonconvex-concave minimax problems on compact submanifolds?}} 

%\sa{In~\citep{huang2023gradient} a Riemannian gradient descent-ascent method is proposed for nonconvex-strongly concave minimax problem over a manifold, and an iteration complexity of $\mathcal{O}(\sa{\kappa_f^2}/\epsilon^2)$ is established for finding a stationary point, under the assumption that $f(x,\cdot)$ is $\mu$-strongly concave for every fixed $x\in\cM$, where $\sa{\kappa_f} := \sa{L(f)}/\mu$ and \sa{$L(f)$} denotes the Lipschitz constant of %$\nabla_y f(x,y)$.
%\sa{$\nabla f$ on $X^\circ\times Y^\circ$}.
%However, in the Euclidean setting, the dependence on \sa{$\kappa_f^2$} is known to be suboptimal. For instance, the smoothed gradient descent-ascent methods in \citep{zhang2020single,yang2022faster} achieve an improved complexity of $\mathcal{O}(\sa{\kappa_f}/\epsilon^2)$. This discrepancy naturally leads to the following question:} 
%{\textit{Can a \sa{single-loop first-order method} achieve \sa{$\kappa_f$ in $\cO(1)$ constant of the gradient complexity, rather than $\cO(\kappa_f^2)$?}}}

% \hj{The difficulty of answering the above question lies in the nonconvexity and compactness of the manifold constraint, which makes the generalization of the analysis of \citep{zhang2020single,yang2022faster} highly nontrivial as arbitrarily large regularization used can not guarantee (geodesic) convexity and desired smoothness.} 
\hj{Our main contributions are listed below.}\vspace*{-2mm}
\paragraph{Retraction-free smoothed manifold GDA algorithm.} We reformulate the manifold-constrained minimax problem using an \textit{exact penalty} for the manifold constraint; \rv{moreover, we also introduce a norm-ball constraint \rv{$x\in X$} to %control smoothness
ensure that $f$ is smooth over the compact set $X\times Y$ with a global Lipschitz constant  for $\nabla f$ \rv{--let $l>0$ denote this constant}.} 
%Building upon the Lipschitz continuity of the penalized objective, 
We propose a smoothed manifold gradient descent-ascent (\sm{}) method that operates entirely in Euclidean space, without requiring retractions or projections onto the manifold. This eliminates  matrix orthogonalization and leads to a GPU-friendly implementation. Numerical experiments on quadratic minimax examples, robust deep neural network training, and superquantile-based learning demonstrate clear advantages over state-of-the-art methods. \rv{To our knowledge, %this
  \sa{\sm} is the first retraction-free algorithm for minimax problems over compact submanifolds.}\vspace*{-2mm}
\paragraph{Convergence guarantee for both merely concave \rv{and PL} settings.} We establish that \sm{} finds an \rv{$\cO(\epsilon)$-stationary point} with $\mod{\tilde{\mathcal{O}}}(l^4/\epsilon^{4})$ complexity in the merely concave case and $\mathcal{O}(\kappa/\epsilon^2)$ in the %strongly concave
\rv{PL} case, where $\kappa = l/\mu$. The key technical ingredients in our analysis are to relate stationarity in the penalized problem to that of the original problem by %using
\rv{exploiting} the exact penalty property \rv{of the penalty function adopted in our reformulation}, and to show that the %added norm 
bound constraint $x\in X$ \sa{eventually becomes inactive under the assumption that the primal function is lower bounded.} Compared to \citep{huang2023gradient}, our method removes the requirement for strong concavity and 
%improves the dependence on $\mu$ from $\mathcal{O}(1/\mu^2)$ to $\mathcal{O}(1/\mu)$. 
\rv{can handle the settings where $f_r(x,\cdot)$ is merely concave, or $-f_r(x,\cdot)=h(\cdot)-f(x,\cdot)$ is $\mu$-PL with a possibly non-smooth weakly convex regularizer $h(\cdot)$}. To our knowledge, \sm{} is the first algorithm with provable convergence guarantees for solving~\eqref{prob} in the merely concave and \rv{non-smooth $\mu$-PL settings} \sa{without resorting to retraction operations}. \rv{Moreover, we also discuss how to relax the compactness assumption on the dual domain $Y$ if we strengthen the $\mu$-PL assumption to strong concavity with modulus $\mu>0$ --in short, we call it $\mu$-concave. We establish that in all of these setting, any limit point of \sm{} sequence is a stationary point of the manifold minimax problem. Finally, for the $\mu=0$ setting, we also show that through adopting Tikhonov regularization of the dual, one can improve the complexity to $\cO(l^2/\epsilon^3)$ for computing an $\cO(\epsilon)$-stationary point at the expense of losing the asymptotic stationarity of the iterate sequence.}
\vspace*{-2mm}
%\nsa{We need to come up with a table to make our contributions clear. Columns can be as follows: complexity for mu=0, complexity for mu>0, Retraction free: yes or no, without geodesic convexity}
% Let us first review some existing single-loop algorithms for solving \eqref{prob}, which are given in Table \ref{tab:comparison}. When reducing to the Euclidean case, i.e., $\Mcal = \R^{d_1}$, the same complexities $\mathcal{O}(\kappa^2 \epsilon^{-2})$ and $\mathcal{O}(\kappa^4 \epsilon^{-4})$ are shown for \sa{alternating GDA~(AGDA)} but the batch size requirement can be relaxed to $\mathcal{O}(1)$ \citep{yang2022faster}. This leads to the following question: 
% {\textit{Can stochastic RGDA-type algorithm achieve \sa{$\kappa$ in $\cO(1)$ constant of the gradient complexity, rather than $\cO(\kappa^2)$?}}}
{\paragraph{Notation.} For a vector $x\in\R^d$, $\|x\|$ denotes the Euclidean norm; for a matrix $x\in\R^{d\times r}$, $\|x\|$ and $\|x\|_2$ denote the Frobenius and spectral norms, respectively. For a closed convex set $X$, let $\delta_X(\cdot)$ denote its indicator function, defined as $\delta_X(x) = 0$ if $x \in X$ and $+\infty$ otherwise. The projection of a point $x$ onto $X$ is denoted by $\mathcal{P}_X(x) := \arg\min_{y \in X} \|x - y\|^2$, and the distance from $x$ to $X$ is defined as $\operatorname{dist}(x, X) := \|x - \mathcal{P}_X(x)\|$. For any square matrix $x\in\reals^{n\times n}$, we define the symmetrization operator as $\operatorname{sym}(x) := \frac{1}{2}(x + x^\top)$; \rv{moreover, ${\rm diag}(x)\in\reals^n$ denotes the vector of diagonal elements of $x$.} Given a 
%convex 
\rv{proper, closed} function $h$, \rv{we use $\partial h(x)$ to denote the Fr\'echet subdifferential\footnote{\rv{When $h$ is proper closed convex, $\partial h$ concides with the convex subdifferential, and if $h$ is differentiable at $x$, $\partial h(x)=\{\nabla h(x)\}$.}} at $x$, i.e., $\partial h(x):=\{s:\ \liminf_{y\to x}\frac{h(y)-h(x)-\fprod{s,y-x}}{\norm{y-x}}\geq 0\}$.}} \rv{When $-h(\cdot)$ is weakly convex, we use $-\partial h(x)$ to denote the Fr\'echet subdifferential of $-h$ at $x$, i.e., $-\partial h(x):=\partial (-h(x))$.}  {Finally, for a point $x \in \mathcal{M}$, $T_x \mathcal{M}$ denotes the tangent space of the manifold $\mathcal{M}$ at $x$. For a differentiable function $h$, we write $\nabla h(x)$ and $\operatorname{grad} h(x)$ for its Euclidean and Riemannian gradients at $x$, respectively. Let $g:\reals^n\to\reals^p$ be a differentiable map at $x$, then we use $\mathbf{J}_g(x)\in \reals^{p\times n}$ to denote the Jacobian matrix at $x$, and for any given $u\in\reals^p$ we use $\nabla g(x)[u]:=\mathbf{J}_g(x)^\top u$.}

\begin{table}[h]
\centering
\setlength{\tabcolsep}{1.2pt}
\caption{\small \textbf{Comparison of convergence guarantees and assumptions.} In the column \textbf{``Single-loop''}, we indicate whether the method consists of a single-loop iteration or not. In \textbf{``GNC $f(\cdot,y)$''} %in the table stands for 
column, we indicate whether \textit{geodesic nonconvexity} \rv{of $f(\cdot,y)$ over $\cM$ is %assumed
allowed when $y$ is fixed; in \textbf{``Nonlinear $f(x,\cdot)$''} column we indicate whether nonlinearity of $f(x,\cot)$ is allowed when $x\in\cM$ is fixed. In the two columns about complexity, $\mu=0$ and $\mu>0$ correspond to the cases where $-f_r(x,\cdot)$ are convex and $\mu$-PL, respectively.} Finally, in the column \textbf{``RF''} we state whether the method is \textit{retraction free} or not. The stationarity metrics used across different works are as follows: \citep{huang2023gradient} considers $\min_{x\in\cM}\max_{y\in Y}f(x,y)$ and adopts the criterion $\|\operatorname{grad} \rv{F}(x_\epsilon)\| \leq \epsilon$, where $\rv{F}(x) := \max_{y \in Y} f(x, y)$ for $x\in\cX$; \citep{jordan2022first} considers $\min_{x\in\cM}\max_{y\in \mathcal{N}}f(x,y)$ where $\cM$ and $\mathcal{N}$ are Riemannian manifolds, and the metric adopted for $\epsilon$-stationarity is $\sqrt{\|x_\epsilon - x^*\|^2 + \|y_\epsilon - y^*\|^2} \leq \epsilon$, where $(x^*, y^*)$ denotes the unique saddle point under the assumption on $f$ satisfying geodesic strong convexity–geodesic strong concavity; \citep{xu2024riemannian} considers \eqref{prob} such that $f(x, y) = g(x) + \langle \mathcal{A}(x), y \rangle$ and $h(\cdot)$ is a closed convex function, and adopts two different criteria for measuring $\epsilon$-stationarity, i.e., $\epsilon$-RGS: $\max \left\{ \|\operatorname{grad}_x f(x_\epsilon, y_\epsilon)\|, \frac{1}{\gamma} \| y_\epsilon - \operatorname{prox}_{\gamma h}(y_\epsilon + \gamma \mathcal{A}(x_\epsilon)) \| \right\} \leq \epsilon$, and $\epsilon$-ROS: $\max \left\{ \operatorname{dist}\left(0, \operatorname{grad} g(x_\epsilon) + \mathcal{P}_{T_{x_\epsilon} \mathcal{M}}\Big( \nabla \mathcal{A}(x_\epsilon)^\top \partial h^*(p_\epsilon)\Big)\right), \|p_\epsilon - \mathcal{A}(x_\epsilon)\| \right\} \leq \epsilon$, where $h^*$ is the Fenchel conjugate of $h$; and \citep{xu2024efficient} considers $\min_{x\in\cM}\max_{y\in Y}f(x,y)$ such that $f(x,\cdot)$ is linear for every $x\in\cM$ fixed. For our \sm{} algorithm we measure stationarity %with respect to the exact penalty function $\tilde{f}$ by dissolving the manifold constraint in $x$, 
using the criterion $\max \left\{  \| \grad_x f(\Pcal_{\Mcal}(x_\epsilon), y_\epsilon)\|, 
%\frac{1}{\tau_2} \| \Pcal_Y(y_\epsilon + \tau_2 \nabla_y f(\Pcal_{\Mcal}(x_\epsilon), y_\epsilon) - y_\epsilon\| 
\rv{{\rm dist}\Big(0, -\nabla_y {f}(\Pcal_{\Mcal}(x_\epsilon),y_\epsilon) +\partial{h}(y_\epsilon) \Big)}\right\} \leq \epsilon$.}
\vspace*{2mm}
\label{tab:comparison}
% \begin{tabular}{lcccccc}
% \toprule
% Work & Single-loop & GNC & Nonlinear in $y$ &Complexity ($\mu = 0$) & Complexity ($\mu > 0$) & Retraction-Free \\
% \midrule
% \citep{huang2023gradient} & \cmark & \cmark & \cmark & N/A & $\cO(\epsilon^{-2})$ & \xmark \\
% \citep{jordan2022first} & \cmark & \xmark & \cmark & $\cO(\epsilon^{-1})$ & $\cO(\log(\epsilon^{-1}))$ & \xmark \\
% \citep{xu2024riemannian} & \xmark & \cmark & \xmark & $\cO(\epsilon^{-3})$ & N/A & \xmark \\
% \citep{xu2024efficient} & \cmark & \cmark & \xmark & $\cO(\epsilon^{-3})$ & N/A & \xmark \\
% \sm{} & \cmark & \cmark & \cmark & $\cO(\epsilon^{-4})$ & $\cO(\epsilon^{-2})$ & \cmark \\
% \bottomrule
% \end{tabular}
{\small
\begin{tabular}{l|c|c|c|c|c|c}
\toprule
\textbf{Algorithm} & \textbf{Single-loop} & \textbf{GNC} $f(\cdot,y)$ & \textbf{Nonlinear $f(x,\cdot)$} &\textbf{Complexity} ($\mu = 0$) & \textbf{Complexity} ($\mu > 0$) & \textbf{RF} \\
\midrule
\rv{RGDA}~\citep{huang2023gradient} & \cmark & \cmark & \cmark & N/A & $\cO(\epsilon^{-2})$ & \xmark \\
RCEG~\citep{jordan2022first} & \cmark & \xmark & \cmark & $\cO(\epsilon^{-1})$ & $\cO(\log(\epsilon^{-1}))$ & \xmark \\
RADA~\citep{xu2024riemannian} & \xmark & \cmark & \xmark & $\cO(\epsilon^{-3})$ & N/A & \xmark \\
ARPGDA~\citep{xu2024efficient} & \cmark & \cmark & \xmark & $\cO(\epsilon^{-3})$ & N/A & \xmark \\
\sm{} & \cmark & \cmark & \cmark & \mod{$\cO(\epsilon^{-3})$ or $\cO(\epsilon^{-4})$$^\diamond$}  & $\cO(\epsilon^{-2})$ & \cmark \\
\bottomrule
\end{tabular}}%
\begin{flushleft}
\footnotesize
$^\diamond$ These two complexity results arise under different algorithmic settings and analysis frameworks. 
\rv{The $\mathcal{O}(\epsilon^{-3})$ %(or more specifically $\mathcal{O}(\mod{l^2}\epsilon^{-3})$)
bound for computing an $\epsilon$-stationary point is achieved by introducing a Tikhonov regularization in the $y$-variable, i.e., given an arbitrary $\bar y\in Y$, solving the nonconvex--strongly concave approximate problem $\min_{x\in \cM}\max_{y\in\cY} f_r(x,y)-\mu\norm{y-\bar y}^2$ for $\mu=\cO(\epsilon)$.} 
In contrast, the $\mathcal{O}(\epsilon^{-4})$ 
%(or more specifically $\mathcal{O}(l^4\epsilon^{-4})$) 
bound corresponds to directly solving the original problem, with additional algorithmic parameters introduced to guarantee \rv{asymptotic} convergence \rv{in terms of any limit point of the iterate sequence being a stationary point of the manifold minimax problem.} Although the former bound is better \sa{for small $\epsilon>0$}, the generated iterates converge only to a solution of the approximate problem rather than the original one \sa{in~\eqref{prob}}.
\end{flushleft}
\vspace*{2mm}
\end{table}

\section{\sa{Methodology \& Smoothed Manifold GDA (\sm) Algorithm}}
Recently, the constraint dissolving method for Riemannian optimization \rv{has been proposed in \citep{xiao2024dissolving} %provides a way 
to cast} %transfer 
a manifold-constrained optimization \rv{problem into} an \textit{unconstrained} one. Inspired by \sa{this methodology, \rv{which is} designed for manifold optimization (in the primal sense), given \rv{a} compact smooth submanifold $\Mcal:=\{x\in \R^{d_1\times r}: c(x) = 0\}$ for some smooth $c:\reals^{d\times r}\to\reals^p$, we can rewrite the %primal-dual 
minimax problem in}~\eqref{prob} as\vspace*{-2mm}
\be \label{prob:pen1}  \min_{x \in %\sa{\R^{d_1\times r}}
\rv{\cX}} \max_{y \in %\sa{Y\subset\R^{d_2}}
\rv{\cY}} \rv{\tilde f_r(x,y):=}\tilde{f}(x,y)\rv{-h(y)},\quad \mbox{where}\quad \rv{\tilde{f}(x,y)}:= f(A(x),y) + \frac{\rho}{4}\|c(x)\|^2, \ee
\rv{where $A:\cX\to \cX$} is the constraint dissolving operator, and $\rho >0$ is \rv{a large enough penalty parameter}--the assumptions on $c(\cdot)$ and $A(\cdot)$ are similar to ~\citep[Assumptions~1.1 and 1.2]{xiao2024dissolving} and will be specified later in \cref{sec:analysis}. \rv{The partial gradient $\nabla_x \tilde f(x,y)$ can be computed using chain rule as follows:
\begin{equation}
    \nabla_x \tilde f(x,y)=\nabla A(x)[\nabla_x f(A(x), y)] + \frac{\rho}{2} \nabla c(x)[c(x)],
\end{equation}
where $\nabla A(x)[u]:=\mathbf{J}_A(x)^\top u$ for $x,u\in\reals^{d_1}$ with $\mathbf{J}_A(x)\in\reals^{d_1\times d_1}$ denoting the Jacobian of $A(\cdot)$ at $x$, and $\nabla c(x)[u]:=\mathbf{J}_c(x)^\top u$ for $x\in\reals^{d_1}$ and $u\in\reals^{p}$ with $\mathbf{J}_c(x)\in\reals^{p\times d_1}$ denoting the Jacobian of $c(\cdot)$ at $x$.}

\rv{For example, when $\Mcal$ is the Stiefel manifold, i.e., $\cM=\{x\in\reals^{d_1\times r}:\ x^\top x=I_r\}$, we can set $A(x) = x(\frac{3}{2}I_r - \frac{1}{2}x^\top x)$ and $c(x)=x^\top x-I_r$, for which we have $\sa{\nabla} A(x) [u] := u (\frac{3}{2}I_r - \frac{1}{2}x^\top x) - x {\rm sym}(x^\top u)$ and $\nabla c(x)[u]=2xu$.} %\nhj{Should $\nabla c(x)[u]$ be $u^\top x + x^\top u$?}

\sa{We first state our assumptions on the manifold minimax problem in~\eqref{prob}.}
\begin{assum}
\label{assum:Y}
    \textbf{(i)} Let $h: \cY \rightarrow \rv{\R\cup\{+\infty\}}$ be a \rv{proper, closed,} $\zeta$-weakly convex function with a closed domain $Y:=\dom h=\{y \in \cY: h(y) < \infty\}$, and $h$ is locally Lipschitz on its domain. \textbf{(ii)} Suppose $Y$ is a %convex, compact 
    bounded set. 
    %moreover, $h$ is Lipschitz continuous over $Y$ with constant $l_h$, i.e., $|h(y_1) - h(y_2)| \leq l_h \|y_1 - y_2\|$ for any $y_1, y_2 \in Y$.
    %\nsa{closedness, weak convexity and compact domain might be sufficient for Lipschitz $l_h$ to exist.}
    %\nhj{Move the lipschitz continuity assumption to the remark.}
\end{assum}
\begin{remark}
    \rv{%Since proper, closed and weakly convex functions are locally Lipschitz continuous, 
    Assumption~\ref{assum:Y} implies that there exists a constant $l_h>0$ such that $h$ is Lipschitz continuous over the compact set $Y$, i.e., $|h(y_1) - h(y_2)| \leq l_h \|y_1 - y_2\|$ for any $y_1, y_2 \in Y$.}
\end{remark}
\begin{defi}
\label{def:bar_X}
    \rv{Suppose $\cM\subset\cX$ is a compact submanifold. Define $\bar X \coloneqq \{A(x): %x\in X
    \norm{x}\leq C \}\subset \cX=\reals^{d_1 \times r}$ 
    %and $X:=\{x\in\reals^{d_1 \times r}:~\norm{x}\leq C\}$ 
    for some $C>\frac{1}{2}+\sup_{x\in\cM}\norm{x}$.}
\end{defi}
\begin{assum} \label{assum:lip-f}
    %Given $Y\subset\cY$, 
    Suppose that $f:\cX\times \cY\to\reals$ is differentiable on an open set containing $\rv{\bar X} \times Y$, \rv{where $\bar X\subset \cX$ is given in Definition~\ref{def:bar_X} and $Y=\dom h\subset \cY$,} and that there exist $L_{xx},L_{xy},L_{yx},L_{yy}\geq 0$ such that
$$
\begin{aligned}
& \left\|\nabla_x {f}\left(x_1, y_1\right)-\nabla_x {f}\left(x_2, y_2\right)\right\| \leq L_{xx}\left\|x_1-x_2\right\|+L_{xy}\left\|y_1-y_2\right\|, \\
& \left\|\nabla_y {f}\left(x_1, y_1\right)-\nabla_y {f}\left(x_2, y_2\right)\right\| \leq L_{yx}\left\|x_1-x_2\right\|+L_{yy}\left\|y_1-y_2\right\|,\quad  x_1, x_2 \in \rv{\bar X},\quad y_1, y_2 \in Y.
\end{aligned}
$$
\rv{Let $L:=\max\{L_{xx},L_{xy},L_{yx},L_{yy}\}$.}
\end{assum}

\begin{defi}
\label{def:primal} 
\rv{$F:\cX\to\reals$ denotes the primal function, i.e., \rv{$F(x):=\max_{y\in \cY}f_r(x,y)$}, and let $F^* := \min_{x \in \cM} F(x)$.} %\nsa{Replace $\Phi$ with $F$}
\end{defi}
% \begin{assum}
% \label{assum:coercive}
% Suppose $F:\cX\to\reals$ denotes the primal function such that \rv{$F(x):=\max_{y\in \cY}f(x,y)-h(y)$}. We assume that $F$ is bounded from below, i.e., $-\infty<F^* := \min_{x \in \sa{\cX}} F(x)$.\nsa{Replace $\Phi$ with $F$}
% \end{assum}

\rv{In the rest, we analyze the convergence properties of the proposed \sm{} algorithm applied to \eqref{prob:pen1} under different assumptions on $f_r$ defined in \eqref{prob}. First, we consider the scenario where $-f_r(x,\cdot)$ satisfies the PL property for fixed $x$.}

\rv{The PL property is usually given for smooth functions~\citep[Theorem 2]{karimi2016linear} and its extension for composite functions $p_s+p_c$ where $p_s$ is smooth  and $p_c$ is a proper, closed convex function, is given in \citep[Eq. (12)]{karimi2016linear}. On the other hand, the definition we adopted from \citep{liao2024error} is more general and applies to a more general class of possibly nonsmooth $\zeta$-weakly convex functions.}
\begin{defi}[\citep{liao2024error}]
\label{def:PL}
    \rv{Let $\varphi:\reals^n\to\reals\cup\{+\infty\}$ be a proper, closed, $\zeta$-weakly convex function, and let $S=\argmin_x \varphi(x)$. Suppose $S\neq\emptyset$ and let $\varphi^*=\varphi(x^*)$ for some $x^*\in S$. $\varphi(\cdot)$ satisfies PL inequality with constant $\mu>0$, i.e., we say $\varphi(\cdot)$ is $\mu$-PL, if $2\mu (\varphi(x)-\varphi^*)\leq {\rm dist}^2(0,\partial \varphi(x))$ for all $x\in\dom \varphi$.}
\end{defi}
\rv{For particular examples of nonsmooth $\zeta$-weakly convex functions that satisfy PL condition, see \citep[Sec. 3.2]{liao2024error}.} %\nsa{Give some examples.}

\rv{Consider %$f_r(x,y):=f(x, y) - h(y)$ for $(x,y)\in\cX\times\cY$.
$f_r$ defined in \eqref{prob}. We assume that \rv{$-f_r(x,\cdot)$} satisfies one of the following conditions uniformly for all $x\in \bar X$: \textit{(i)} $\mu$-PL 
%in $y$ given by  
(see \cref{assum:sc}), 
%strong concavity in $y$, 
\textit{(ii)} $\mu$-strongly convex, and \textit{(iii)} merely convex, i.e., $\mu=0$.}
\begin{assum} 
\label{assum:sc} 
%For any $x \in \cX$, ${f}(x, \cdot)$ is $\mu$-concave \sa{over $Y$ for some $\mu\geq 0$, i.e., for $y_1,y_2\in Y$,  
% %and $\lambda\in[0,1]$, $f(x,\lambda y_1+(1-\lambda)y_2)\geq \lambda f(x,y_1)+(1-\lambda)f(x,y_2)+\frac{\mu}{2}\lambda(1-\lambda)\norm{y_1-y_2}^2$ 
% $\fprod{\nabla_y f(x,y_1)-\nabla_y f(x,y_2),~y_1-y_2}\leq-\mu\norm{y_1-y_2}^2$.}
\rv{There exists $\mu>0$ such that $-f_r(x,\cdot)$} is $\mu$-PL for all $x \in \bar X\subset\cX$, i.e.,
%\nsa{Here assuming for $x\in X$ is not sufficient as we require this property for $f_r(A(x),\cdot)$ in the proof.}  
%\[ \rv{2\mu \Big( h(y)-f(x, y)-\min_{w \in \cY} \{h(w)-f(x,w)\} \Big)  \leq {\rm dist}^2\Big(0, -\nabla_y f(x,y) +\partial h(y) \Big).}   \]
\[ \rv{ \max_{w \in \cY} f_r(x,w) - f_r(x,y)  \leq \frac{1}{2\mu}~{\rm dist}^2\Big(0, -\nabla_y f(x,y) +\partial h(y) \Big),\quad \forall x \in \bar X,\ y \in %\cY
\rv{Y}}.\]
%\nsa{Here, I change $y\in\cY$ to $y\in Y$.}
\end{assum}
\begin{remark}
\label{rem:PL-relations}
    \rv{The solution set $\argmax_{w \in \cY} f_r(x,w)$} is nonempty from the compactness of $Y$. 
    By \citep[Theorem~3.1]{liao2024error}, the assumption on $-f_r(x,\cdot)$ satisfying $\mu$-PL condition is weaker than assuming $f_r(x,\cdot)$ is strongly concave with modulus $\mu>0$; indeed, it is equivalent to \rv{the} %local 
    error–bound condition and further implies quadratic growth of $-f_r(x,\cdot)$. Examples of nonconvex–PL problems include \rv{$\min_x\max_y f(x,By)$ where $f(\cdot,\cdot)$ is nonconvex–strongly concave} and $B$ is an arbitrary matrix, as well as generative adversarial imitation learning for the linear–quadratic regulator \citep{cai2019global}. 
\end{remark}
\begin{assum} 
\label{assum:sc2}
\rv{There exists $\mu>0$ such that $f_r(x,\cdot)$ is $\mu$-strongly concave for all $x \in \bar X\subset \cX$, i.e., for any $y\in Y$ and $g\in\partial h(y)$, it holds that
$f_r(x,y) +\fprod{\nabla_y f(x,y)-g,~w-y}-\frac{\mu}{2}\norm{w-y}^2  \geq f_r(x,w)$ for all $x \in \bar X$, and $w \in \rv{Y}$.}
\end{assum}
\begin{assum} 
\label{assum:mc} 
\rv{$h:\cY\to\reals\cup\{+\infty\}$ is merely convex, i.e., %$h(\cdot)$ satisfies \cref{assum:Y} with $\zeta=0$
$\zeta=0$ in \cref{assum:Y}; moreover, $f(x,\cdot)$ is merely concave for all $x\in \bar X\subset \cX$, i.e., for any $y\in Y$ fixed, $f(x,y)+\fprod{\nabla_y f(x,y),~w-y}\geq f(x,w)$ for all $x\in \bar X$ and $w\in Y$.}
\end{assum}
% \begin{remark}
% \mod{The conditions in 
% \cref{assum:lip-f}, \cref{assum:sc}, \cref{assum:sc2}, and \cref{assum:mc}
% % \cref{assum:lip-f,assum:sc,assum:sc2,assum:mc} 
% can be correspondingly weakened by imposing them only for $x\in \bar X$ (noting $\bar X\subseteq \cX$), instead of for all $x\in \cX$.}
% \end{remark}
\begin{remark}
    One can replace \cref{assum:lip-f} with the assumption that $f$ is twice continuously differentiable on an open set containing $\bar X\times Y$, where $\bar X$ is given in Definition~\ref{def:bar_X}.
    %$\bar X\triangleq\{x\in\reals^{d_1 \times r}:~\norm{x}\leq \bar C\}$ for \rv{any} $\bar C>\frac{1}{2}+\sup_{x\in\cM}\norm{x}$. 
    %\nsa{Can we assume $\nabla f$ is Lipschitz on an open set containing $\bar X\times Y$ in the assumption above rather than $\cX\times Y$? The same question applies to PL assumption as well.}
    %\nhj{Since we already add Remark 3 on a weaker version of this assumption, is it ok to remove this remark?}
    %We will specify what $\bar C>0$ should be based on the analysis coming later on.
\end{remark}
{Given $(x^*,y^*)\in \cX\times Y$,} the results in \citep{xiao2024dissolving} imply that for sufficiently large but fixed \sa{$\rho>0$}, if ${\rm dist}(x^*, \Mcal) \leq \delta$ for \rv{some} small enough $\delta > 0$, then
\begin{align}
    \label{eq:stationary-equiv}
    \nabla_x \tilde{f}(x^*, y^*) = 0\quad \Leftrightarrow\quad \sa{x^*\in\cM},\quad \grad_x f(x^*,y^*) =0,
\end{align}
\sa{where $\grad_x f(x^*,y^*)$ denotes} the \sa{partial} Riemannian gradient of $f$ \sa{with respect to $x$ over 
% the manifold 
$\cM$.} \rv{For any $(\bar x,\bar y)\in\cM\times Y$,  the \sa{partial} Riemannian gradient of $f$ with respect to $x$ can be computed using
\begin{equation}
\label{eq:Riemannian_grad}
    \grad_x f(\bar x,\bar y)=\nabla_x f(\bar x,\bar y)-\bar x \operatorname{sym}(\bar x^\top \nabla_x f(\bar x,\bar y)).
\end{equation}}%
\hj{\rv{For details on the first-order} optimality condition in manifold optimization, we refer to \citep{absil2009optimization,boumal2023introduction,hu2023projected}.}

\begin{defi}
\label{def:M-stationary}
    $(x^*,y^*)\in\cM\times Y$ is a stationary point for the minimax problem in \eqref{prob} if
    %\nsa{define what $\epsilon$-stationary point is by introducing the gradient mapping.} 
\begin{align}
\label{eq:stationarity}
    \grad_x f(x^*,y^*) = 0,\quad 
    %\fprod{\nabla_y f(x^*,y^*),~y-y^*}\leq 0,\quad\forall~y\in Y.
    0\in -\nabla_y f(x^*,y^*) + \rv{\partial %\delta_{Y}(y^*)
    h(y^*)};
\end{align}
\rv{and $(\xe,\ye)\in\cM\times Y$ is an $\epsilon$-stationary point, if $\norm{\grad_x f(\xe,\ye)}\leq \epsilon$ and ${\rm dist}\Big(0,-\nabla_y f(\xe,\ye) + \rv{\partial 
    h(\ye)}\Big)\leq \epsilon$.}
%where $\delta_Y(.)$ denotes the indicator function of the closed convex set $Y$.
%\nsa{This will be more connected to the approximate stationary used in the Theorems, specifically, ${\rm dist} (0, \nabla_y f(x^*, y^*) + \partial \delta_{Y}(y^*)) \leq \| \frac{1}{\tau} (\Pcal_{Y}(y^* - \tau \nabla_y f(x^*,y^*)))\|$}
\end{defi}
%\nhj{Add a nonasymptoic generalization of result in Equation (3). If the Euclidean gradient is $\mathcal{O}(\epsilon)$, we have both constraint violation and the Riemannian gradient at the projected point are in $\mathcal{O}(\epsilon)$}

Note that because of the \sa{penalty term $\rv{P(x)}:=\frac{\rho}{4}\|c(x)\|^2$ used within $\tilde f$}, \sa{$\nabla\tilde{f}$ may not 
%have bounded Lipschitz smoothness constant though $f$ does.
be Lipschitz on $\cX\times \cY$ while $\nabla f$ is.} For example, when $\Mcal$ is the Stiefel manifold, \sa{\rv{we set} $c(x)=x^\top x-I_r$; \rv{hence,} the corresponding penalty term $P(\cdot)$ is quartic and $\nabla P(x)= \rho x(x^\top x-I_r)$ is not Lipschitz on \rv{$\cX=\reals^{d_1\times r}$}.} However, since we are interested in \sa{stationary points $(x^*,y^*)\in\cM\times Y$ as in \eqref{eq:stationarity} with $\cM$ being a compact submanifold of \rv{$\reals^{d_1\times r}$},} 
%$x$ will lie in the compact manifold. 
instead of \eqref{prob:pen1}, we will consider the following reformulation: 
% Note that because of the quadratic term $\|c(x)\|^2$, $\tilde{f}$ may not have bounded Lipschitz smoothness constant though $f$ does. For example, when $\Mcal$ is the Stiefel manifold, the penalty term $\frac{\beta}{2}\|x^\top x - I\|^2$ is quartic and not gradient-Lipschitz. However, the optimal $x$ will lie in the compact manifold. Instead of \eqref{prob:pen1}, we consider the following reformulation 
% (\revise{Get rid of $X$, show Every step moves closely to manifold.} )
\be \label{prob:pen} \min_{x \in X\subset\cX} \max_{y \in \cY} \rv{\tilde{f}_r(x,y)},  \ee
where $X:=\{x \in \rv{\R^{d_1\times r}}: \|x\| \leq C\}$ \sa{for some %large $C > 0$ such that
$C>\frac{1}{2}+\sup_{x\in\cM}\norm{x}$ --for this reformulation, trivially $\nabla\tilde{f}$ is Lipschitz on \rv{the new domain $X\times Y$ that is compact}, and $-\tilde f_r(x,\cdot)$ is %$\mu$-strongly concave
\rv{$\mu$-PL (or convex) for every fixed $x\in X$ since $-f_r(x,\cdot)$ is %$\mu$-strongly concave
assumed to be $\mu$-PL (or convex)} for all $x\in 
\rv{\bar X}$.} %The choice of the constant \( C \) is subtle; 
\sa{It will be shown later that as long as \rv{the constant $C>0$ is sufficiently large, i.e., $C>\frac{1}{2}+\sup_{x\in\cM}\norm{x}$,} the choice does not affect the set of limit points of the proposed algorithm. 
%however, it plays a crucial role in the algorithm design stage. 
Let $l(C)$ denote the Lipschitz constant of $\nabla \tilde f$ on $X\times Y$ --for simplicity of %the
notation, we will suppress its dependence on $C$ in the rest and use $l$ instead.}
% \begin{assum}
% \label{assum:lxx}
%     There exists $l_{xx}>0$ such that $\norm{\nabla_x\tilde f(x_1,y)-\nabla_x\tilde f(x_2,y)}\leq l_{xx}\norm{x_1-x_2}$ for all $x_1,x_2\in X$ and $y\in Y$.
% \end{assum}

\sa{In the rest of this paper, for notational convenience we focus on the Stiefel manifold to derive our results. That said, as mentioned in the introduction, we discuss their extensions to 
%the general case 
\rv{some other important compact manifolds} with $c\in \mathcal{C}^2$; more specifically, the \textit{oblique manifold} and the \textit{generalized Stiefel manifold} will be discussed in \cref{sec:extensions}.}  
%These assumptions are formally stated below.
% where $X:=\{x \in \R^{d_1}: \|x\| \leq C\}$ for a large $C > 0$. Then, $\tilde{f}$ is Lipschitz smooth with respect to $x$ over $X$ and is strongly concave in $y$. The choice of the constant \( C \) is subtle. As will be shown later, while it does not affect the limiting point of the proposed algorithm, it plays a crucial role in algorithm design. 

\begin{lemma} \label{lem:lip-f}
\sa{Let $\cM\subset\cX$ be the Steifel manifold, i.e., $c(x)=x^\top x-I_r$, 
%For any given $C>0$, 
and let $X=\{x \in \R^{d_1\times r}: \|x\| \leq C\}$ for $C>0$ as in Definition~\ref{def:bar_X}. Under Assumption~\ref{assum:lip-f}, the function $\tilde{f}$ defined in \eqref{prob:pen1} with $A(x) = x(\frac{3}{2}I_r - \frac{1}{2}x^\top x)$ is differentiable on an open set containing \rv{$X \times Y$}, and there exist constants $l_{xx},l_{xy},l_{yx},l_{yy}\geq 0$ \rv{such that $l_{yy}=L_{yy}$, $l_{yx}=\cO(C^2 L_{yx})$, $l_{xy}=\cO(C^2 L_{xy})$ and $l_{xx}=\cO(C^2 \rho + C^4 L_{xx})$ and that} for all $x_1, x_2 \in X$ and $y_1, y_2 \in Y$, it holds that
$$
\begin{aligned}
& \left\|\nabla_x \tilde{f}\left(x_1, y_1\right)-\nabla_x \tilde{f}\left(x_2, y_2\right)\right\| \leq l_{xx}\left\|x_1-x_2\right\|+ l_{xy}\left\|y_1-y_2\right\|, \\
& \left\|\nabla_y \tilde{f}\left(x_1, y_1\right)-\nabla_y \tilde{f}\left(x_2, y_2\right)\right\| \leq l_{yx}\left\|x_1-x_2\right\|+ l_{yy}\left\|y_1-y_2\right\|.
\end{aligned}
$$
Clearly, $\nabla \tilde f$ is Lipschitz on $X\times Y$ with constant $l:=\max\{l_{xx},l_{xy},l_{yx},l_{yy}\}$.}
%\nsa{Note that $l_{xx}$ depends on $\rho$; therefore, $l$ depends on $\rho$! However, $l_{yy}=L_{yy}$, $l_{yx}=\cO(C^2 L_{yx})$, and $l_{xy}=\cO(C^2 L_{xy})$.}
\end{lemma}
\begin{proof}
    \rv{See Section~\ref{sec:Lipschitz-proof}.}
\end{proof}
%\subsection{Lipschitz smoothness of $\tilde{f}$}
\begin{remark} \label{rmk:lip}
    If the condition $\|x\| \leq C$ in the definition of $X$ is replaced by $\|x\|_2 \leq C$, the expressions for the Lipschitz constants $l_{xx}, l_{xy}, l_{yx}, l_{yy}$ in \cref{lem:lip-f} remain the same. %\nsa{verify!}  
\end{remark}
\begin{remark}
    The Lipschitz constants of $\tilde{f}$ calculated in the proof of \cref{lem:lip-f} are specific to the case where $\Mcal$ is the Stiefel manifold, i.e., $c(x)=x^\top x-I_r$. That said, they can be easily extended to %the more general case of 
   \rv{other compact smooth submanifolds $\Mcal=\{x\in\cX:\ c(x) = 0\}$ over $X\times Y$ whenever the defining function $c(\cdot)$ is twice continuous differentiable.} 
\end{remark}

\subsection{\sa{Smoothed AGDA~(\sm)} for solving problem \eqref{prob}}
%\nsa{First, briefly talk about \sm{} method in the Euclidean space.}
\hj{The smoothed gradient descent ascent \rv{(sm-GDA)} methods \citep{zhang2020single,yang2022faster} perform single-loop updates on the primal and dual variables, together with an additional update on a smoothing variable. These methods are designed for problems with a nonconvex–concave or \rv{nonconvex–PL (strongly concave)} structure %\sa{with}
\rv{assuming} Lipschitz continuous gradients, and under such conditions, they achieve 
%faster convergence rates 
\sa{better iteration complexity than the} standard gradient descent ascent methods \rv{while computatinal cost per iteration being the same.}} 

It is known that every geodesically convex function over a compact manifold is a constant function \citep{bishop1969manifolds}. \sa{This indicates that for any fixed $y\in Y$ and $z\in X$, $\rv{{f}(x,y;z)} \triangleq \sa{f(x, y)} + \frac{p}{2}\|x - z\|^2$ is not geodesically convex in $x$ over $\cM$ for any $p > 0$.} Although we may use \rv{a retraction operator to design a %\sm
smoothed descent-ascent type method as in \citep{zhang2020single,yang2022faster}}, its convergence analysis 
%will be difficult
\sa{would be very intricate} because \sa{for any given $z\in X$, %the equivalence 
the duality gap between $\min_{x\in\cM}\max_{y\in Y} \rv{f}(x,y;z)$ and $\max_{y\in Y} \min_{x\in\cM} \rv{f}(x,y;z)$ %may
is not necessarily %hold
\textit{zero}} due to the lack of (geodesic) convexity.  This motivates us to use more advanced tools to tackle the  manifold constraint \sa{which leads to a nonconvex minimax problem.} 

Specifically, we consider the penalized problem \rv{in} \eqref{prob:pen}, which employs an exact penalty function \rv{$\tilde f_r(\cdot,y)$ for any $y\in Y$} to eliminate the manifold constraint $x\in \cM$. By introducing \sa{a norm %constraint set $X$ 
ball constraint %on $x$, i.e., $\norm{x}\leq C$ for some $C>0$
$x\in X$,} we %control 
\rv{avoid potential issues that would arise from the gradient of the penalty term $ \frac{\rho}{\sa{4}}\|c(x)\|^2 $ not being Lipschitz continuous over $\cX$. In this setup, we 
%then apply use 
study the convergence behavior of smoothed gradient descent ascent %method
iterate sequence~\citep{zhang2020single,yang2022faster} to a stationary point of the original problem when %to solve
sm-GDA is applied on the penalty problem in~\eqref{prob:pen}. Indeed,  we incorporate the geometry of the manifold through using an exact penalty function, and we refer this particular implementation of sm-GDA framework on \eqref{prob:pen} as \sm{}}. 

Let $p\in\reals_+$ such that $p>l$, e.g., $p=2l$, and define 
\[ \hat{f}(x,y;z) := \tilde{f}(x,y) + \frac{p}{2}\|x - z\|^2,\qquad~\rv{\forall~x\in X,\ y\in Y,\ z\in \cX,} \]
where $\frac{p}{2}\|x-z\|^2$ serves as a regularizer for dual smoothing, \rv{inspired by the Moreau-Yosida regularization, a.k.a. the Nesterov's smoothing~\citep{nesterov2005smooth}.} Our proposed method, \sm{}, is presented in Algorithm \ref{alg:sm-agda-stiefel}, which consists of alternating $x,y$ updates: one projected gradient descent in $x$ \rv{using $\nabla_x \hat f(x_t,y_t;z_t)$}, one %projected
\rv{proximal} gradient ascent in $y$ \rv{using $\nabla_y \hat f(x_{t+1},y_t;z_t)$}, and an averaging step in $z$ \rv{to get $z_{t+1}$}. %for solving \eqref{prob:pen}. 
\rv{On the other hand,} for the case $\mu = 0$, %in addition to the smoothing technique, 
\rv{rather than employing dual smoothing, we set $p=0$ and incorporate an additional regularization term on $y$ for primal smoothing,} %to ensure approximate strong concavity, 
following \citep{xu2023unified,lu2020hybrid}. Indeed, \rv{for the scenario $\mu=0$, rather than computing an ascent step based on $\tilde f(x_{t+1},\cdot)$ for the $y$-update, we consider $\tilde f_{\theta_t}(x_{t+1},\cdot)$ where $\tilde f_{\theta}(x,y)=\tilde f(x,y)-\frac{\theta}{2}\norm{y}^2$ for $x\in X$ and $y\in Y$, i.e., we add a regularization term $-\frac{\theta_t}{2}\norm{y}^2$ to $\tilde{f}$, and use $\nabla_y \tilde f_{\theta_t}(x_{t+1},y_t)$ rather than $\nabla_y \hat f(x_{t+1},y_t;z_t)=\nabla_y \tilde f (x_{t+1},y_t)$,
\rv{where the parameter sequence $\{\theta_t\}\subset\reals_+$} is diminishing to $0$.} In our theoretical analysis, \rv{we show that the limit points of \sm{} iterate sequence are stationary points of  the original problem \eqref{prob} in the sense of Definition~\ref{def:M-stationary}} under appropriate choices of the parameter \rv{$C>0$ appearing in the definition of $X$} and the penalty parameter $\rho$ in $\tilde{f}$.
% \begin{remark}
%     In \cref{alg:sm-agda-stiefel}, $\Pcal_{X}(x) = \frac{x}{\|x\|}$. If we replace $\|x\| \leq C$ with $\|x\|_2 \leq C$ in defining $X$, then $\Pcal_X(x) = u\hat{s}v^\top$, where $u{\rm diag}(s_1, \dots, s_r)v^\top$ is the compact singular value decomposition of $x$ and $\hat{s} ={\rm diag}(\min(s_1, C), \dots, \min(s_r,C))$.
% \end{remark}
\begin{remark}
    One can avoid computing/estimating the Lipschitz constant $l$ by adopting the line-search strategy of \citep{zhang2024agda+}, which employs a carefully designed nonmonotone stepsize-search criterion and requires at most 3 backtracking steps per iteration. \rv{This type of extension would be helpful in practice as it can exploit local curvature through estimating local Lipschitz constants, which would lead to larger steps.}
\end{remark}
%\nsa{In a paragraph define the notation used, e.g., sym. For transpose use ``top" rather than ``T"}

% \begin{algorithm} 
% \caption{Stochastic Smoothed AGDA}
% \begin{algorithmic}[1]
% \STATE Input: $(x_0, y_0, z_0) \in \Mcal \times Y \times \Mcal$, step sizes $\tau_1 > 0, \tau_2 > 0$, parameters $0<\rho < 1$, $C > 0$ and $p > 0$. 

% \FOR{$t = 0, 1, 2, \ldots, T - 1$ do}
%     \STATE Draw two i.i.d. samples $\xi_t^1, \xi_t^2$
%     \STATE $x_{t+1} \gets \pcal_X(x_t - \tau_1 G_x(x_t, y_t, \xi_t^1) - \tau_1 p(x_t - z_t))$
%     \STATE $y_{t+1} \gets \Pcal_{Y}(y_t + \tau_2 G_y(x_{t+1}, y_t, \xi_t^2))$
%     \STATE $z_{t+1} = z_t + \rho(x_{t+1} - z_t)$
% \ENDFOR

% \STATE Output: choose $(\hat{x}, \hat{y})$ uniformly from $\{(x_t, y_t)\}_{t=0}^{T-1}$
% \end{algorithmic} \label{alg:sm-agda}
% \end{algorithm}

\begin{algorithm}[h] 
\caption{Smoothed MGDA \sa{(\sm)}}
\begin{algorithmic}[1]
\STATE \textbf{Input}: $(x_0, y_0, z_0) \in \Mcal \times Y \times \Mcal$, %step sizes 
 $\{\theta_t\},\{\tau_{1,t}\}\subset\R_+$, $\tau_2, p, \rho, C > 0$, $\beta\in(0,1]$ 

\FOR{$t = 0, 1, 2, \ldots, T - 1$}
    \STATE \sa{$g_t\gets \nabla A(x_t)[\nabla_x f(A(x_t), y_t)] + \frac{\rho}{2} \nabla c(x_t)[c(x_t)]$}
    \STATE \sa{$x_{t+1} \gets \pcal_{X}\left(x_t - \sa{\tau_{1,t}} \Big(
    %\nabla_x f(x_t,y_t) - x_t {\rm sym}(x_t^\top \nabla_x f(x_t,y_t)) + \rho x_t (x_t^\top x_t - I)
    g_t+ p(x_t - z_t)\Big)\right)$}
    \STATE $y_{t+1} \gets {{\rm prox}_{\tau_2 h}}\Big(y_t + \tau_2 \Big(\nabla_y f(A(x_{t+1}), y_t)\sa{-\theta_t y_t}\Big)\Big)$
    \STATE $z_{t+1} = z_t + \beta (x_{t+1} - z_t)$
\ENDFOR

\STATE \textbf{Output}: 
% choose $(\hat{x}, \hat{y})$ uniformly from 
$\{(x_t, y_t)\}_{t=0}^{T-1}$.
\end{algorithmic} \label{alg:sm-agda-stiefel}
\end{algorithm}
%\nsa{It is better to be consistent with Niao's paper. So I think we should switch the roles of $\beta$ and $\rho$ throughout the paper.}
%\nhj{Switched.}

\subsection{$\epsilon$-stationary points of the penalty problem}
\sa{Following} \citep[Definition 3.1]{zhang2020single}, we call \sa{$(x_\epsilon,y_\epsilon)\in X\times Y$} an $\epsilon$-stationary point of \eqref{prob:pen} if
\be \label{eq:stationary} {\rm dist}\Big(0, \nabla_x \tilde{f}(x_\epsilon,y_\epsilon) + \partial \delta_{X}(x_\epsilon) \Big)\leq \epsilon, \;\; \sa{{\rm dist}\Big(0, -\nabla_y \tilde{f}(x_\epsilon,y_\epsilon) +\partial {h}(y_\epsilon) \Big)\leq \epsilon}.  \ee
% \sa{where $\delta_{X}(\cdot)$ and $\delta_{X}(\cdot)$ denote the indicator functions of closed convex sets $X$ and $Y$, respectively; and ${\rm dist}(v,B)$ denotes the distance of point $v$ to the closed convex set $B$, i.e., ${\rm dist}(v,B)=\norm{v-\cP_B(v)}$ with $\cP_B(v)$ denoting the Euclidean projection of $v$ onto $B$.}

\sa{\rv{Given $(\xe,\ye)\in X\times Y$ such that $\|\xe\| < C$, then $0 \in \nabla_x \tilde{f}(\xe,\ye) + \partial \delta_{X}(\xe)$ %implies that
if and only if} $\nabla_x \tilde{f}(\xe,\ye)=0$; hence, if \eqref{eq:stationary} holds with $\epsilon = 0$ for some $(\xe,\ye)\in X\times Y$ such that $\|\xe\| < C$, then
\be \label{eq:stationary-org} \nabla_x \tilde{f}(\xe, \ye) = 0,\quad 0 \in -\nabla_y \tilde f(\xe,\ye) + \partial {h}(\ye). \ee
Thus, it follows from \eqref{eq:stationary-equiv} that for $\rho>0$ large enough and $C>\delta+\sup_{x\in\cM}\norm{x}$, there exists $\delta>0$ such that if ${\rm dist}(\xe, \Mcal) \leq \delta$ and \eqref{eq:stationary} holds with $\epsilon=0$, then $(\xe,\ye)\in\cM\times Y$ and it is a stationary point for the original manifold constrained minimax problem in \eqref{prob}, i.e., \eqref{eq:stationarity} holds with $(x^*,y^*)=(\xe, \ye)$.}
%In this case, $(x,y)$ is also a pair of stationary points of problem \eqref{prob}. 
%Generally speaking, the (projected) gradient descent ascent method for \eqref{prob:pen} can only be guaranteed to converge to a stationary point satisfying \eqref{eq:stationary}. 
% \sa{Therefore, %given $\epsilon>0$, 
% \rv{the choice of $C>0$ is not arbitrary and it should be chosen carefully depending on $\delta$ in order} to exclude the possibility of computing an %$\epsilon$-stationary point 
% $(x_\epsilon,y_\epsilon)\in X\times Y$ as in \eqref{eq:stationary} with $\|x_\epsilon\|= C$.} 
% we require %the coercivity of 
% that the primal function $\Phi(\cdot)$ is bounded from below, \rv{i.e., we adopt \cref{assum:coercive}.} 
%It is to verify that $\hat{f}(x,y)$ is also coercive for fixed $y$. Then, for any \eqref{eq:stationary} with $\|x\|= C$, it holds $\tilde{f}(x,y) \geq M$ for some large $M > 0$.

\hj{In practice, since the algorithm runs for only a finite number of iterations, we are also  interested in how the $\epsilon$-stationary point of the penalized problem relates to that of the original manifold-constrained problem. The following lemma describes this connection \rv{for the Steifel manifold.}}
\begin{lemma} \label{lem:equiv-stationary}
    \rv{Let $\cM\subset\reals^{d_1\times r}$ be the Steifel manifold.} For any given $\epsilon>0$, let $(x_\epsilon, y_\epsilon)\in\cX\times Y$ be such that $\| \nabla_x \tilde{f}(x_\epsilon, y_\epsilon)\| \leq \epsilon$, $\rv{{\rm dist}(0, -\nabla_y \tilde{f}(x_\epsilon,y_\epsilon) +\partial {h}(y_\epsilon) )\leq \epsilon}$ 
    % $\frac{1}{\tau_2}~\|%\left(
    % {\rm prox}_{\tau_2 h}(y_\epsilon + \tau_2 \nabla_y \tilde{f}(x_\epsilon, y_\epsilon) ) - y_\epsilon %\right) 
    % \| \leq \epsilon$ \rv{for some $\tau_2\in(0,\frac{1}{\zeta})$} 
    and ${\rm dist}(x_\epsilon, \Mcal) \leq \frac{1}{2}$, then 
    \[ 
    \begin{aligned}
        &\| x_\epsilon - \Pcal_{\Mcal}(x_\epsilon)\| \leq \frac{3}{\rho} \epsilon, \quad \|\sa{\grad_x} f(\Pcal_{\Mcal}(x_\epsilon), y_\epsilon)\| \leq  \rv{\epsilon + \frac{11}{\rho}\Big(L_{xx}+L_x(\ye)\Big)\epsilon},\\
    &
    %\frac{1}{\tau_2}\|{\rm prox}_{\tau_2 h}(y_\epsilon + \tau_2 \nabla_y f(\Pcal_{\Mcal}(x_\epsilon), y_\epsilon) ) - y_\epsilon\| 
    \rv{{\rm dist}\Big(0, -\nabla_y {f}(\Pcal_{\Mcal}(x_\epsilon),y_\epsilon) +\partial{h}(y_\epsilon) \Big)}\leq \Big(1 + \frac{3}{\rho} L_{yx} \Big) \epsilon,
    \end{aligned}
    \]
    whenever $\rho \geq \rv{36 L_x(\ye)}$, where $\rv{L_x(y)}:= \max\{\|\nabla_x f(x,y)\|:\ \|x\|_2 \leq 
    \rv{1}\}$ is defined for $y\in Y$. \rv{Thus, $(\Pcal_{\Mcal}(x_\epsilon),y_\epsilon)\in\cM\times Y$ is an $\cO(\epsilon)$-stationary point for the minimax problem in \eqref{prob} in terms of Definition~\ref{def:M-stationary}.}
    % $\rho \geq \mod{36} L_g(y_\epsilon)$, where $L_g(y) := \max\{\|\nabla_x f(x,y)\|:\ \|x\|_2 \leq %3
    % \rv{1}\}$ for $y\in Y$.
\end{lemma}
\begin{proof}
    \rv{See Section~\ref{sec:eps-stationary-connection-proof}.}
\end{proof}
For any given $\epsilon>0$, in order to compute $(\xe,\ye)\in X\times Y$ satisfying the hypothesis of Lemma~\ref{lem:equiv-stationary} we will employ \sm{} on the penalty problem in~\eqref{prob:pen} with a properly chosen parameter $C>0$, i.e., \rv{the choice of $C>0$ is not arbitrary and it should be chosen carefully depending on $\delta$ in order to exclude the possibility of computing an %$\epsilon$-stationary point 
$(x_\epsilon,y_\epsilon)\in X\times Y$ as in \eqref{eq:stationary} with $\|x_\epsilon\|= C$. Indeed, Lemma~\ref{lem:equiv-stationary} shows that when $\cM\subset\reals^{d_1\times r}$ is the Steifel manifold, any $C>\frac{1}{2}+r$ works as $\delta=1/2$ and $\sup_{x\in\cM}\norm{x}=\sqrt{r}$.}

\section{Convergence analysis}
\label{sec:analysis}

In this section, we present the convergence guarantees of our \sm{} algorithm under \rv{three different scenarios: We assume that \rv{$-f_r(x,\cdot)$} satisfies one of the following conditions uniformly for all $x\in \bar X$: \textit{(i)} $\mu$-PL 
%in $y$ given by  
(see \cref{assum:sc}), 
%strong concavity in $y$, 
\textit{(ii)} $\mu$-strongly convex, and \textit{(iii)} merely convex, i.e., $\mu=0$.}
%mere concavity in $y$.
\begin{defi}
    \rv{Let $-\infty<\bar F:=\min_{x\in X}F(A(x))$ where $F(\cdot)=\max_{y\in\cY}f_r(\cdot,y)$.}
\end{defi}
\subsection{%$\mu$-PL in $y$
\rv{Complexity for nonconvex-PL problems}}
\rv{In this section, we establish the convergence guarantees %of our \sm{} algorithm 
for computing an $\epsilon$-stationary point of the manifold minimax problem in~\eqref{prob}, and we also provide asymptotic convergence results. To achieve this goal, we first extend the analysis of \citep{yang2022faster} to handle weakly convex (possibly non-smooth) regularizer $h(\cdot)$ --the method proposed in~\citep{yang2022faster} can only handle \textit{smooth} problems of the form $\min_{x\in\cX}\max_{y\in\cY}f(x,y)$ without any manifold constraint, where $f$ is a smooth function satisfying \cref{assum:lip-f} on the vector space $\cX\times \cY$ such that $-f(x,\cdot)$ is $\mu$-PL for all $x\in\cX$.} 

\rv{We next provide a convergence rate result for \sm{} %for the setting satisfying
under Assumptions~\ref{assum:Y},\ref{assum:lip-f} and \ref{assum:sc}. %Note that 
For this setting there is no clear relation between $l$ and $\mu$; hence, we define the modified condition number $\bar\kappa :=\max\{1,\kappa\}$ where $\kappa:=l/\mu$. Furthermore, we also define some other important quantities arising in our analysis.}
\begin{defi}
    For any $z\in\cX$, let $x^*(z):=\argmin_{x \in \sa{X}} \Phi(x ; z)$ and ${Y}^*(z):=\argmax_{y \in {\cY}} \Psi_r(y;z)$, where $\Phi(x ; z):=\max_{y\in\cY}\hat f_r(x,y;z)$ defined for any $x\in X$, $\Psi_r(y;z):=\rv{\Psi(y ; z)-h(y)}$ and $\Psi(y ; z):=\min_{x\in \sa{X}} \hat{f}(x, y ; z)$ defined for any $y\in \rv{Y}$. Finally, $P(z):=\min_{x\in X}\Phi(x;z)$ for $z\in\cX$.
\end{defi}
%Leveraging the weak convexity of $h$ and $-F(x,\cdot)$ being $\mu$-PL for all $x\in X$, we  to our \sm{} algorithm and establish the following convergence guarantee. We note that our setting differs slightly due to the presence of a norm ball constraint on $x$. \nhj{add the definition of $G_t^x$.}
\begin{thm} \label{thm}
    % \revise{(Need to consider deterministic setting)}
%\nsa{Now that we will focus on the deterministic scenario, we should check if we can use closed convex regularizer for the primal part and a simple closed convex set for the dual part.}
    \sa{Suppose Assumptions~\ref{assum:Y},\ref{assum:lip-f} and \ref{assum:sc} hold. For any given $C>0$, let $X=\{x\in\cX:\norm{x}\leq C\}$, and let $\{x_t,y_t,z_t\}_{t\geq 0}$ be the \sm~iterate sequence generated by Algorithm~\ref{alg:sm-agda-stiefel}, initialized from an arbitrary $(x_0,y_0,z_0)\in X\times %\cY
    \rv{Y}\times \cX$, using the following parameters: $\tau_{1,t}=\tau_1$ and $\theta_t=0$ for all $t\geq 0$ for some 
    %$\tau_1=\min \left\{\frac{\sqrt{\Delta}}{2 \sigma \sqrt{T l}}, \frac{1}{3 l}\right\}$, $\tau_2=\min \left\{\frac{\sqrt{\Delta}}{96 \sigma \sqrt{T l}}, \frac{1}{144 l}\right\}, 
    $\tau_1\in(0,\frac{1}{\rv{3l}}]$, $\tau_2=\rv{\frac{1}{16}(\frac{3}{\tau_1}+\zeta)^{-1}}$, $p=2 l$ and $\beta=\alpha\min\{\mu,l\}\tau_2$ for some {$\alpha\in(0,1/2306)$}, where the constant $l>0$ is defined in \cref{lem:lip-f}.} \rv{Then, for any $T\geq 1$, it holds that
    \begin{equation}
    \label{eq:rate-result-PL}
        \frac{1}{T} \sum_{t=1}^{T} \Big(\left\| G_{t}^x \right\|^2 +\rv{\bar\kappa} \left\| G_{t}^y \right\|^2\Big) \leq \frac{O(1) \rv{\bar\kappa}}{T}\rv{\Big(\frac{1}{\tau_1}+\zeta\Big)\Big(P(z_0)-\bar F+\Delta_0\Big)},
    \end{equation}
    where $\Delta_0:=\hat{f}_r(x_0, y_0 ; z_0)+P(z_0)-2\Psi_r(y_0 ; z_0)$, and for all $t\geq 0$, $G_t^x,G_t^y$ are defined as 
    \begin{equation}
\label{eq:subgradient-G}
\begin{aligned}
    G_{t+1}^x &:= \frac{x_{t} - x_{t+1}}{\tau_1} + \nabla_x \tilde f(x_{t+1}, y_{t+1}) - \nabla_x \tilde f(x_t,y_t) + p(z_t - x_t),\\
    G_{t+1}^y &:= \frac{y_{t+1}-y_{t}}{\tau_2} + \nabla_y \tilde f(x_{t+1}, y_{t+1}) - \nabla \tilde f(x_{t+1}, y_t).
\end{aligned}
\end{equation}
Furthermore, $G_t^x\in\nabla_x\tilde f(x_t,y_t)+\partial \delta_X(x_t)$ and $G_t^y\in\nabla_y\tilde f(x_t,y_t)-\partial h(y_t)$ for all $t\geq 1$, which also implies that
    \[ \min_{\rv{1} \leq t \leq T} \max\Big\{ {\rm dist}\Big(0, \nabla_x \tilde{f}(x_t,y_t) + \partial \delta_{X}(x_t)\Big), \; \sqrt{\bar \kappa} {\rm dist}\Big(0, -\nabla_y \tilde{f}(x_t,y_t)+\partial h(y_t) \Big)\Big\} \sa{=} \mathcal{O}\Big(\sqrt{\bar \kappa / T}\Big).\]
    Finally, $\Delta_0\geq 0$ can be bounded as $\Delta_0\leq 2~\rv{\operatorname{gap}_{\hat{f}_r}\left(x_0, y_0;z_0\right)}$, where $\rv{\operatorname{gap}_{\hat{f}_r}\left(x_0, y_0;z_0\right):=}
    %\max_{y\in Y}\hat f_r(x_0,y;z_0)-\min_{x\in X} \hat f_r(x,y_0;z_0)=
    \rv{\Phi(x_0;z_0)-\Psi_r(y_0;z_0)}$.}% 
    %\rv{Finally, for a particular choice of $(x_0,y_0,z_0)$ such that $x_0=z_0\in\cM$ and $y_0\in Y^*(z_0)$, one has $\Delta_0\leq 2(F(z_0)-\bar F)$}.
\end{thm}
\begin{proof}
    \rv{See \cref{sec:pl-proof}.}
\end{proof}
\begin{remark}
\label{rem:init}
If one chooses $x_0=z_0$ and $y_0\in Y^*(z_0)$ for any $z_0\in\cM$, then $\hat f_r(x_0,y_0;z_0)=f_r(z_0,y_0)\leq F(z_0)$. Moreover, %for any $x_0\in X$ and $z_0\in\cX$, 
choosing $y_0\in Y^*(z_0)$ implies that $P(z_0)+\Delta_0=\hat f_r(x_0,y_0;z_0)$ since $P(z_0)=\Psi_r(y_0;z_0)$ --see Lemma~\ref{lem:minmax}. Therefore, $P(z_0)+\Delta_0-\bar F\leq F(z_0)-\bar F$.

\rv{On the other hand, if one chooses $x_0=x^*(z_0)\in X$ and $y_0\in Y^*(z_0)$ for some $z_0\in\cM$, then one has $\Delta_0=0$ --see Lemma~\ref{lem:minmax}. Moreover, since $z_0\in \cM$, we have $P(z_0)=\min_{x\in X}\Phi(x;z_0)\leq \Phi(z_0;z_0)=F(z_0)$; therefore, $P(z_0)+\Delta_0-\bar F\leq F(z_0)-\bar F$ as well.}
\end{remark}
\sa{Next, we argue that for both $C, \rho>0$ sufficiently large, $\norm{x_t}<C$ for all $t\geq 0$; thus, \cref{thm} implies that $\nabla f(x^*,y^*)=0$ for any limit point $(x^*,y^*)$ of the \sm{} iterate sequence $\{(x_t,y_t)\}$.}
\begin{thm}
    \label{thm:mu-pl}
    %Assume that $Y$ is compact and $\Phi(\cdot)$ is coercive. 
    %with respect to $x$. 
    \sa{Under the premise of \cref{thm}, \rv{suppose \sm{} is initialized from $(x_0,y_0,z_0)\in X\times {Y}\times \cX$ \rv{such that $y_0\in Y^*(z_0)$ and $x_0$ is set to either $z_0$ or $x^*(z_0)$ for some arbitrary $z_0\in\cM$}. Then,} 
    %and \cref{assum:coercive}, 
    for any $C>\frac{1}{2}+\sup_{x\in\cM}\norm{x}$, there exists $\bar\rho>0$ such that within $\cO(\rv{\bar\kappa}/\epsilon^2)$ gradient evaluations, \emph{\sm{}} with given parameters $\rho\geq \bar\rho$ and $C>0$ can generate $(x_\epsilon,y_\epsilon)\in X\times \cY$ such that $\| x_\epsilon - \Pcal_{\Mcal}(x_\epsilon)\| \leq \frac{3}{\rho} \epsilon$, $\|\sa{\grad_x} f(\Pcal_{\Mcal}(x_\epsilon), y_\epsilon)\| = \mathcal{O}(\epsilon)$, and $\rv{{\rm dist}(0, -\nabla_y f(\Pcal_{\Mcal}(x_\epsilon), y_\epsilon) + \partial h(y_{\epsilon}))}= \mathcal{O}(\epsilon)$.} 
    
    Moreover, any limit point $(x^*,y^*)$ of the \emph{\sm} sequence $\{(x_t,y_t)\}_{t\geq 0}$ is a stationary point for the minimax problem in \eqref{prob}, i.e., $x^*\in\cM$, $\grad_x f(x^*, y^*) =0$ and \rv{$0\in-\nabla_y f(x^*, y^*)+\partial h(y^*)$}.
    %$\exists\bar C>0$ such that $\forall~C>\bar C$, any limit point $(x^*,y^*)$ of $\{(x_t,y_t)\}_{t\geq 0}$ generated} by Algorithm \ref{alg:sm-agda-stiefel} satisfies $\|x^*\| < C$ and \sa{$\nabla f(x^*,y^*)=0$}.
    %\[ \nabla_x \tilde{f}(x^*,y^*) = 0,\quad \Pcal_{Y}(y^* + \tau_2 \nabla_y \tilde{f}(x^*,y^*)) = y^*. \]
\end{thm}
\begin{proof}
    \rv{See \cref{sec:PL-limit-proof}.}
\end{proof}
\begin{remark}
    \rv{%Consider a particular initialization: 
    For any given $z_0\in\cM$, let $x_0=x^*(z_0)\in X$ and $y_0\in Y^*(z_0)$. Then, for any $C>\frac{1}{2}+\sup_{x\in\cM}\norm{x}$, 
    %one can choose  
    choosing $\rho\geq %\bar\rho:=
    \rv{16\Big(F(z_0)-\rv{\bar F}+ (\bar l_y+l_h)D_Y\Big)}$, where $\bar l_y:=\max_{x \in X, y \in Y} \|\nabla_y f(A(x),y) \|<\infty$ and $D_Y:=\sup_{y_1,y_2\in Y}\norm{y_1-y_2}$, %Choosing $\rho\geq \bar\rho$ 
    implies that $\norm{c(x_t)}\leq\frac{1}{2}$, which ensures ${\rm dist}(x_t,\cM)\leq \frac{1}{2}$ for all $t\geq 0$.}
\end{remark}
{\begin{remark}
    The above result corresponds to the case where $X$ is defined using the Frobenius norm.  Instead, if $X$ is defined by the spectral norm ball, the same iteration complexity bound continues to hold for any $C > \frac{1}{2} + \sup_{x\in\cM}\|x\|_2 = \frac{3}{2}$. Since the expressions for the Lipschitz constants remain unchanged (see \cref{rmk:lip}), this smaller value of $C$ (in contrast to $\frac{1}{2} + \sqrt{r}$) yields an improved iteration complexity  \rv{$\mathcal{O}(\max\{L/\mu,1\}/\epsilon^2)$ as $L\leq l$, where $L$ is defined in Assumption~\ref{assum:lip-f}.}  However, it should be noted that in this setting the projection operator $\Pcal_X$ requires computing an SVD \rv{of an $n_1\times r$ matrix at each iteration, which can be substantially more expensive than the projection onto the Frobenius norm ball, especially when $r\approx n_1$, e.g., $n_1=r$.} 
\end{remark}}

\subsection{%Strong concavity in $y$
\rv{Complexity for nonconvex-strongly concave problems}}
\rv{In this section, we replace Assumption~\ref{assum:sc} with Assumption~\ref{assum:sc2} which is clearly a stronger one; indeed, \citep[Theorem~3.1]{liao2024error} shows that Assumption~\ref{assum:sc2} implies Assumption~\ref{assum:sc} --see also Remark~\ref{rem:PL-relations}. That said, while adopting a stronger condition of $f_r$, we now relax the compactness requirement on $Y=\dom h$, which is necessary for our analysis in the $\mu$-PL case, i.e., we will study \eqref{prob} under Assumptions~\ref{assum:Y}.\textbf{(i)},~\ref{assum:lip-f} and \ref{assum:sc2}.}
\begin{defi}
\label{rem:Phi-connection}
Let $\Phi(x):=\max_{y\in {\cY}}\tilde f_r(x,y)$ and $\rv{\Phi^*}=\min_{x\in X}\Phi(x)$. Note that $\Phi(x)=F(A(x))+\frac{\rho}{4}\norm{c(x)}^2$; hence, $\rv{\Phi^*}\geq \rv{\bar F}$. 
%\Phi^*. 
Moreover, \sa{define $r^*(x) \triangleq \arg \max_{y\in\cY} \tilde{f}_r(x, y)=\arg \max_{y\in\cY} \rv{f_r(A(x), y)}$ for all $x\in X$, and let $\delta_t:=\norm{y_t-r^*(x_t)}^2$ for $t\geq 0$.}
\end{defi}

% The above theorem establishes convergence to an approximate stationary point of the penalized problem. Moreover, \sa{under a mild %coercivity 
% condition stated in \cref{assum:coercive},} the norm constraint becomes inactive at any limit point generated by the \sm{} algorithm. 
{First, we argue that when $f_r(x,\cdot)$ is strongly concave on $Y$ for all $x\in\bar X$, the \sm{} iterate sequence is bounded. Note that for every $z\in\cX$, due to strong concavity $Y^*(z)$ is a singleton, let $Y^*(z)=\{y^*(z)\}$.}
\begin{lemma}
\label{lem:bounded-sequence}
    %Suppose that Assumptions~\ref{assum:lip-f},\ref{assum:sc},\ref{assum:Y} hold with $\mu>0$.
    \rv{Suppose %that 
    Assumptions~\ref{assum:Y}\textbf{.(i)},\ref{assum:lip-f} and \ref{assum:sc2} hold.}
    \sa{Given an arbitrary $C>0$ such that $X\supset \cM$, let $\{x_t,y_t,z_t\}\subset \rv{X\times \cY\times \cX}$ be generated by \sm{} \rv{using $\tau_1,\tau_2,p,\beta,\alpha$ as stated in~\cref{thm}.} 
    %i.e., $\mu>0$ and $Y=\cY$. 
    Then starting from any $(x_0,y_0,z_0)\in X\times \cY\times X$, $\{x_t,y_t,z_t\}_{t\geq 0}$ is a bounded sequence; indeed, $\max\{\norm{x_t}, \norm{z_t}\}\leq C$. \rv{If 
    %one uses $x_0=x^*(z_0)$ and $y_0=y^*(z_0)$ for some $z_0\in\cM$ as the initial point of 
    one initializes \sm{} from $(x_0,y_0,z_0)$ as in Theorem~\ref{thm:mu-pl},} then $\norm{y_t-y_0}\leq \sqrt{\frac{2}{\mu}(\rv{F}(z_0)-\rv{\Phi^*})}+2\kappa_{yx}\norm{z_t-z_0}$ for all $t\geq 0$, where $\kappa_{yx}:=l_{yx}/\mu$.}   
\end{lemma}
\begin{proof}
    \rv{See~\cref{sec:proof-bounded-sequence}.}
\end{proof}
\sa{Next, we state a result that establishes $\{\norm{y_t-r^*(x_t)}^2\}\to 0$ as $t\to\infty$.}
\begin{lemma}
\label{lem:delta}
    %For all $T\in\mathbb{Z}_{++}$, it holds that $\sum_{t=0}^{T-1}\delta_t\leq 2\kappa\delta_0+2\kappa^4\sum_{t=0}^{T-1}\norm{x_{t+1}-x_t}^2$; 
    \sa{For $t\geq 0$, $\delta_{t+1}\leq (1-\tau_2\mu/2)\delta_t+\frac{(1-\tau_2\mu)(2-\tau_2\mu)}{\tau_2\mu}\frac{l^2_{yx}}{\mu^2}\norm{x_{t+1}-x_t}^2$; moreover, \rv{$\sum_{t=0}^{\infty}\delta_t<\infty$, which implies that} $\delta_t\to 0$ as $t\to\infty$.}
\end{lemma}
\begin{proof}
    \rv{See~\cref{sec:proof-delta}.}
\end{proof}
\sa{Next, we argue that for both $C, \rho>0$ sufficiently large, $\norm{x_t}<C$ for all $t\geq 0$; thus, \cref{thm} implies that 
%$\nabla f(x^*,y^*)=0$ for any limit point of $\{x_t,y_t\}$.
\rv{any limit point of $\{x_t,y_t\}$ is a stationary point for the minimax problem in \eqref{prob} in terms of Definition~\ref{def:M-stationary}.}}
\begin{thm}
    \label{coro} 
    %Suppose that Assumption~\ref{assum:Y} without the compactness of $Y$, and Assumption~\ref{assum:lip-f} hold. Further, assume that $-f(x, \cdot) + h(\cdot)$ is $\mu$-strongly convex for all $x \in \mathcal{X}$.
    \rv{Under Assumptions~\ref{assum:Y}\textbf{.(i)},\ref{assum:lip-f} and \ref{assum:sc2}, there exists some $\bar\rho>0$ such that for every $\rho>\bar\rho$ there is $\bar T_\rho\in\mathbb{Z}_+$, which is non-increasing in $\rho$, such that the results of \cref{thm:mu-pl} continue to hold for all $t\geq \bar T_\rho$.}
    % For any $C>\frac{1}{2}+\sup_{x\in\cM}\norm{x}$, there exists $\bar\rho>0$ such that within $\cO(\kappa/\epsilon^2)$ gradient evaluations, \emph{\sm{}} with the parameters $\rho\geq \bar\rho$ and $C>0$ can generate $(x_\epsilon,y_\epsilon)\in X\times \cY$ such that $\| x_\epsilon - \Pcal_{\Mcal}(x_\epsilon)\| \leq \frac{3}{\rho} \epsilon$, $\|\sa{\grad_x} f(\Pcal_{\Mcal}(x_\epsilon), y_\epsilon)\| = \mathcal{O}(\epsilon)$, and ${\rm dist}\Big(0, -\nabla_y f(\Pcal_{\Mcal}(x_\epsilon), y_\epsilon) + \partial h(y_{\epsilon})\Big)= \mathcal{O}(\epsilon)$. Moreover, for any limit point $(x^*,y^*)$ of the \emph{\sm} sequence $(x_t,y_t)$, it holds that $x^*\in\cM$, $\grad_x f(x^*, y^*) =0$ and \rv{$0\in-\nabla_y f(x^*, y^*)+\partial h(y^*)$}.
\end{thm}
\begin{proof}
    See \cref{sec:proof-convergence}.
\end{proof}

\subsection{%Mere concavity in $y$
\rv{Complexity for nonconvex-merely concave problems}}
Beyond the $\mu$-PL and strong concavity settings for $f_r(x,\cdot)$, we can also establish an \rv{$\tilde\cO(\epsilon^{-4})$ complexity of our \sm{} algorithm to compute an $\epsilon$-stationary point of the original problem under mere concavity of $f_r(x,\cdot)$ for any $x\in \bar X$. This result follows from Theorem~\ref{coro} and the potential function construction provided in \citep[Theorem 3.2]{xu2023unified}.}

\begin{thm} \label{thm:concave}
    %Suppose that Assumptions~\ref{assum:Y} and \ref{assum:lip-f} hold. Further, assume that both $-f(x, \cdot)$ and $ h(\cdot)$ is convex for all $x \in \mathcal{X}$.  
    \rv{Suppose Assumptions~\ref{assum:Y},\ref{assum:lip-f} and \ref{assum:mc} hold.}
    For any $C>\frac{1}{2}+\sup_{x\in\cM}\norm{x}$, there exists $\bar{\rho}>0$ such that within $\rv{\tilde\cO}(l^4\epsilon^{-4})$ gradient evaluations, initialized from an arbitrary $(x_0,y_0,x_0)\in \Mcal \times Y\times \Mcal$ and using parameters 
    % $\rho > \bar{\rho}, \tau_2 \leq \frac{1}{10l}, \sa{\tau_{1,t}} \leq \frac{1}{37 \tau_2 l^2+\sa{l\sqrt{t}}}, \theta_t = \frac{19}{20%k
    % \sa{t^{1/4}}}, p = 0$, %$\beta =0$
    % and $\beta=1$,
    {$\rho > \bar{\rho}, \tau_2 \leq \frac{1}{10 \rv{l_{yy}}}, \sa{\tau_{1,t}} = \frac{2}{2 \tau_2 l^2(1 + b \rv{\sqrt{t+1}} ) - l}, b > \max\left\{ \frac{2}{\tau_2 l}-1, \frac{32\cdot 20^2}{19^2} \right\}, \theta_t = \frac{19}{20}\cdot\frac{1}{\tau_2}\cdot
    \rv{\frac{1}{(t+1)^{1/4}}}, p = 0$, %$\beta =0$
    and $\beta=1$,} %\nsa{Previously $\theta_t = \frac{19}{20 \tau_2\sa{t^{1/4}}}$ was used.}
    \emph{\sm{}} can generate $(x_\epsilon,y_\epsilon)\in X\times \cY$ such that $\| x_\epsilon - \Pcal_{\Mcal}(x_\epsilon)\| \leq \frac{3}{\rho} \epsilon$, $\|\sa{\grad_x} f(\Pcal_{\Mcal}(x_\epsilon), y_\epsilon)\| = \mathcal{O}(\epsilon)$, and $\rv{{\rm dist}\Big(0, -\nabla_y f(\Pcal_{\Mcal}(x_\epsilon), y_\epsilon) + \partial h(y_{\epsilon})\Big)}= \mathcal{O}(\epsilon)$. Moreover, for any limit point $(x^*,y^*)$, %of the \emph{\sm} sequence $\{(\bar x_t,\bar y_t)\}$, 
    it holds that $x^*\in\cM$, $\grad_x f(x^*, y^*) =0$ and \rv{$0\in-\nabla_y f(x^*, y^*)+\partial h(y^*)$}.
    % 
%     Then, for sufficiently large $T$, $\|x_t - \Pcal_{\Mcal}(x_t) \| = \mathcal{O}(T^{-\frac{1}{4}})$ and
% \be  \label{eq:complexity-concave} \min_{0 \leq t \leq T} \max\{ \| \grad_x f(\Pcal_{\Mcal}(x_t),y_t) \|, \frac{1}{\tau_2}\| \Pcal_{Y}( y_t + \tau_2 \nabla_x f(\Pcal_{\Mcal}(x_t),y_t)) - y_t \| \} \sa{=} \mathcal{O}(T^{-\frac{1}{4}}). \ee
\end{thm}
\begin{proof}
    \rv{See~\cref{sec:proof-concave} in the appendix.}
\end{proof}
% \begin{corollary}
%     Suppose that Assumption \ref{lem:lip-f} and \ref{assum:sc}.ii) holds. Let ${x_t,y_t}$ be generated by Algorithm \ref{alg:sm-agda-stiefel}. Then, under $p>3l, \tau_1 < \frac{1}{p+L}$, 
% $
% \tau_2<\min \left\{\frac{1}{11 l}, \frac{\tau_1^2(p-l)^2}{4 l(1+\tau_1(p-l))^2}\right\}, \beta \leq \min \left\{\frac{1}{36}, \frac{(p-l)^2}{384 p(p+l)^2}, \frac{1}{\sqrt{T}}\right\}
% $ with iteration counter $T$, 
% % there exists a $t \in\{1,2, \cdots, T\}$ such that $\left(x^{t}, y^{t}\right)$ is a $\mathcal{O}\left(T^{-1 / 4}\right)$-stationary solution of $\min_{x \in X} \max_{y \in Y} \; \tilde{f}(x,y)$, namely,
% \[ \min_{0 \leq t \leq T} \max\{ {\rm dist}(0, \nabla_x \tilde{f}(x_t,y_t) + \partial \delta_{X}(x_t) ), \; {\rm dist}(0, \nabla_y \tilde{f}(x_t,y_t) + \partial \delta_{Y}(y_t) ) \} \leq \mathcal{O}(l^4 D_Y^2\epsilon^{-4}), \] 
% where $D_Y:=\max_{x,y\in Y}\|x- y\|$ is the diameter of $Y$.  

% Moreover, assume that $\max_{y\in Y} f(x,y)$ is coercive with respect to $x$. For large enough $C$ and $\rho$, any limiting point $(x^*,y^*,z^*)$ is also a stationary point of $\min_{x\in X}\max_{y \in Y} f(x,y)$, i.e., 
%     \[ \grad_x f(x^*,y^*) = 0, \;\; \Pcal_{Y}(y^* + \tau_2 \nabla_y f(x^*,y^*)) - y^* = 0. \] 
% \end{corollary}

\begin{remark}
    The above theorem \sa{establishes} the first complexity result \sa{for a retraction-free method to solve} %smooth 
    nonconvex-merely-concave \rv{(NCMC)} minimax problems over compact submanifolds. 
\end{remark}
\begin{remark}
\label{rem:regularization}
    In addition to the $\mathcal{O}(l^4 \epsilon^{-4})$ complexity result established in \cref{thm:concave}, we can also obtain an improved $\mathcal{O}(\epsilon^{-3})$ bound by applying \sm{} to the regularized problem:
    \[
        \min_{x \in \Mcal} \max_{y \in \cY} \; \rv{f_r^\epsilon(x,y):=} f(x,y) - h(y) - \frac{\epsilon}{4\rv{D_Y}} \|y - y_0\|^2,
    \]
    where $\rv{D_Y} := \max_{y_1,y_2 \in Y} \|y_1 - y_2\|$ and $y_0 \in Y$. 
    %The function 
    \rv{$f_r^\epsilon(x,\cdot)$ is strongly concave with modulus $\mu = \tfrac{\epsilon}{2D_Y}$ for any fixed $x$.} Therefore, by \cref{thm:mu-pl}, \sm{} yields a pair $(x_\epsilon, y_\epsilon)$ such that
    $\|x_\epsilon - \Pcal_{\Mcal}(x_\epsilon)\| \leq \tfrac{3}{\rho} \epsilon, \; 
        \|\grad_x f(\Pcal_{\Mcal}(x_\epsilon), y_\epsilon)\| = \mathcal{O}(\epsilon), \;
        \mathrm{dist}\!\left(0, -\nabla_y f(\Pcal_{\Mcal}(x_\epsilon), y_\epsilon) + \tfrac{\epsilon}{2\rv{D_Y}} (y_\epsilon - y_0) + \partial h(y_\epsilon)\right) = \mathcal{O}(\epsilon).$
    Since $\rv{\bar\kappa=}\kappa = l/\mu = \mathcal{O}(l \epsilon^{-1})$ and $\rv{D_Y} \geq \|y_\epsilon - y_0\|$, it follows that \rv{$\mathrm{dist}\!\left(0, -\nabla_y f(\Pcal_{\Mcal}(x_\epsilon), y_\epsilon) + \partial h(y_\epsilon)\right) = \mathcal{O}(\epsilon)$; therefore,} $(x_\epsilon, y_\epsilon)$ is an $\epsilon$-stationary point of the original \rv{NCMC minimax problem $\min_{x\in\cM}\max_{y\in\cY}f_r(x,y)$}, and it can be computed within $\mathcal{O}(\rv{l^2}\epsilon^{-3})$ iterations.
\end{remark}
\subsection{Extension to \rv{other} compact submanifolds}
\label{sec:extensions}
In previous sections, for the theoretical analysis of \sm{}, we focus on the case where $\mathcal{M}$ is the Stiefel manifold. However, the concept of exact penalty functions introduced in \citep{xiao2018regularized} applies more broadly to all compact submanifolds satisfying some mild regularity conditions. In the following, we discuss how our algorithm and %theoretical results 
its guarantees can be extended to some other important compact submanifolds.
%\nsa{Discuss these ideas extend to compact manifolds defined by $c\in \mathcal{C}^2$.}
 \paragraph{Oblique manifold:} $\cM=\{x \in \R^{d\times r}: \diag(x^\top x) = I_r\}$. In this case, we define 
\[
A(x) = 2x \left(I_r + \diag(x^\top x)\right)^{-1}, \quad c(x) = \diag(x^\top x) - I_r.
\]  
To ensure that the same convergence guarantees hold, it suffices to (i) validate Lemma~\ref{lem:equiv-stationary}, and (ii) ensure that small constraint violation \(c(x)\) still implies that \(x\) is close to the manifold.  

For (i), we can %directly apply
use \citep[Theorem 1]{xiao2024dissolving} to control the feasibility error and Riemannian gradient norm by gradient norm of the penalized objective function %is small 
when \(x\) is sufficiently close to the manifold and \(\rho\) is chosen large enough.

For (ii), let \(\bar{x}\) denote the projection of \(x\) onto \(\Mcal\), given by  
$ \bar{x} = x \, \diag(x^\top x)^{-\frac{1}{2}}.$
Then, $
\|x - \bar{x}\| = \| x \left(I_r - \diag(x^\top x)^{-\frac{1}{2}} \right) \|.$
If \(\|c(x)\| \leq \epsilon\) with \(\epsilon \leq \frac{1}{2}\), then we have  
$1 - \epsilon \leq \max_{1 \leq i \leq r} \|x_i\| \leq 1 + \epsilon$,  and  
\[
\|x - \bar{x}\| \leq \|x\|_2 \left\|I_r - \diag(x^\top x)^{-\frac{1}{2}} \right\| \leq \frac{3}{2} \left\| I_r - \diag(x^\top x)^{-\frac{1}{2}} \right\| \leq \frac{3}{2} \sqrt{r \epsilon}.
\]  

\paragraph{Generalized Stiefel Manifold:} \(\cM=\{x \in \R^{d \times r} : x^\top B x = I_r\}\) where \(B \in \R^{d \times d}\) is a positive definite matrix.   In this case, we define  
\[
A(x) = x \left(\tfrac{3}{2} I_r - \tfrac{1}{2} x^\top B x\right), \quad c(x) = x^\top B x - I_r.
\]  
Following \citep[Theorem 1]{xiao2024dissolving}, we can establish results analogous to Lemma \ref{lem:equiv-stationary}, ensuring approximate stationarity under the exact penalty framework.  

To show that small \(c(x)\) implies small ${\rm dist}(x,\Mcal)$, consider the compact SVD of \(B^{\frac{1}{2}} x\):  
$B^{\frac{1}{2}} x = u s v^\top.$  Define  
$\bar{x} = B^{-\frac{1}{2}} u v^\top.$  
Then,  
$\|x - \bar{x}\| = \left\| B^{-\frac{1}{2}} \left( u s v^\top - u v^\top \right) \right\| \leq \| B^{-\frac{1}{2}} \| \| s - I_r \|.$ 
Noting that  
$c(x) = x^\top B x - I_r = v s^\top s v^\top - I_r = v (s^2 - I_r) v^\top$,   
we obtain \(\| s^2 - I_r \| = \| c(x) \|\), which implies \(\| s - I_r \| \leq \| c(x) \|\).  
Thus, we establish the desired  result by having $
\| x - \bar{x} \| \leq \| B^{-\frac{1}{2}} \| \, \| c(x) \|$.

\section{Numerical experiments}
\vspace{-0.3cm}
 In this section, we conduct experiments on three tasks, i.e., the quadratic nonconvex-merely concave \rv{(NCMC)} problem, the quadratic nonconvex-strongly concave (NCSC) problem, and robust DNNs training over a Riemannian manifold.
% Although \sm{} is a deterministic algorithm, we also test its efficiency in stochastic scenarios such as applications related to deep learning. 
In experiments, we compare our \sm{} with MGDA, (i.e., $p = 0$ and $\beta = 0$ in Algorithm \ref{alg:sm-agda-stiefel}), RGDA \citep{huang2023gradient,lin2025two}, and its stochastic version RSGDA \citep{huang2023gradient} as the comparison baselines. We use constant primal step sizes for \sm{}, i.e., $\tau_{1,t} = \tau_1$ for all $t=0,1,\dots$. For all figures in this section, we run each algorithm several times independently, and plot the geometric mean and standard deviation in solid lines and shaded regions, respectively. 
% For additional experiments on superquantile-based learning, we refer the readers to Section \ref{A-2}.

% \hj{ToDO:}

% \begin{itemize}
%     \item First plot the loss $f(x_k, y^*(x_k))$
%     \item Gradient mapping: plot the residuals of gradient mapping for $x$ and $y$, separately. 
%     \item Issues on $p$: Generally should be $2L$ for the guarantee of strong convexity. The current choice of $p = 1e-2$ could be too small, similar to GDA. 
%     \item Condition number for the dual problem is not too bad in the experiments. This does not differ the performance between GDA and smoothe AGDA too much, as $\kappa$ is small. 

%     To resolve this, we could test a toy example. For this toy problem, we know the parameters, e.g., Lipschitz constant, dual condition mumber. Then, the step size and $p$ value could be chosen from the theory. We can also set the parameters, and play with them. 

%     Practical evidence of effiency on the toy examples first and then go to the NN problem if working well. 

%     Try Toy example first and then DRO with FashionMNIST first?
    
% \end{itemize}
\vspace{-0.3cm}
\subsection{Quadratic \rv{NCMC} problem ($\mu=0$)} \label{ncc}
\subsubsection{Case 1: Vector variables} \label{sec-Vector1}
In the first scenario, we consider the following quadratic \rv{NCMC} minimax problem:
\begin{equation} \label{toy-example}
\min_{x \in \mathcal{M} }\max_{y \in Y} \mathcal{L}(x, y) = \frac{1}{2} x^\top Q x + x^\top A y ,
\end{equation}
where $Y =\{y \in \mathbb{R}^n:\ \|y\| \le 1 \},~\mathcal{M} = \{x \in \mathbb{R}^n:\ x^{\top}x = 1 \}$ denotes the sphere manifold, and $A,Q \in \mathbb{R}^{n \times n}$ are symmetric random matrices. We set $ n = 500$, and randomly generate $A$ and $Q$ such that $A = V \Lambda_A V^\top$ and $Q = V \Lambda_Q V^\top$, where $\Lambda_A,\Lambda_Q$ are diagonal matrices, and $V \in \mathbb{R}^{500 \times 500}$ is an orthogonal matrix, i.e., $VV^\top=V^\top V=I$. We set
$
\Lambda_Q = \frac{L \cdot \Lambda_Q^0}{\|\Lambda_Q^0\|_2} 
$
for $L \in \{5, 10\}$, where $\Lambda_Q^0$ is a random diagonal matrix with diagonal elements being sampled uniformly at random from the interval $[-1, 1]$ and $\|\Lambda_Q^0\|_2$ denotes the spectral norm of $\Lambda_Q^0$; the diagonal $\Lambda_A$ is generated the same way, except we ensure that $\norm{\Lambda_A}_2=1$. Thus, $\nabla\mathcal{L}$ is Lipschitz with constant $L=\max\{\norm{Q}_2,\norm{A}_2\}$.  
% Given $x$, we can compute $y^*(x) = \frac{1}{\|A^\top x \|} A^\top x$. Consequently, 
The primal loss function is
$
\rv{F}(x) = \rv{\frac{1}{2} x^\top Q x + \| A x\|}$. 
%\nsa{Check this function, I corrected it.}
Since $x$ lies on the sphere, the minimum value of $\rv{F}(x)$ is $%\Phi_{\min}
\rv{F^*}= \lambda_{\min} \left(\rv{\frac{1}{2}} Q + %(AA^{\top})^{\frac{1}{2}}
A\right)$.
%\nsa{I do not think that $\frac{1}{2}\lambda_{\min} \left(Q + (AA^{\top})^{\frac{1}{2}} \right)$ is correct.}
% where $\lambda_{\min}(\cdot)$ denotes the smallest eigenvalue. 
We randomly regenerated $10$ instances as described above and the results are shown in Figure \ref{fig:toy-example} for the algorithms running on these instances. The relative gradient norm denotes $\min_{\rv{1} \leq t \leq T} \Big\{ \|G_t^x \|^2 + \|G_t^y\|^2 \Big\} / ( \|G_1^x \|^2 + \|G_1^y\|^2 ) $, where $G_t^x$ and $G_t^y$ are defined by \eqref{eq:subgradient-G}. In the experiments, we set $\beta =0.9$, $p = 2L$, $\tau_1 = \frac{1}{6L}$, $ \tau_2 = \frac{\tau_1}{3}$, $\rho = 10$, and $C = 1000$ for \sm{}, and we use the same $\tau_1$ and $\tau_2$ for MGDA, i.e., $\tau_1 = \frac{1}{6L}$ and $ \tau_2 = \frac{\tau_1}{3}$. For RGDA, we set $\tau_1 = \frac{1}{16(\mu_y+1)^2L}$ and $\tau_2 = \frac{1}{L}$ according to \citep[Theorem 17]{lin2025two}. We observe that \sm{} %converges the fastest compared with 
is competitive against MGDA and RGDA. 
%in terms of primal objective function values, gradient norms in $x$ and $y$.
Furthermore, the manifold constraint violations show that the iterates $x_t$ gradually land on the manifold and the norm constraint, $\norm{x}\leq C$, is always inactive.

\begin{figure}[htbp]
\centering
\begin{minipage}{0.48\textwidth}
    \centering
    \includegraphics[width=\linewidth]{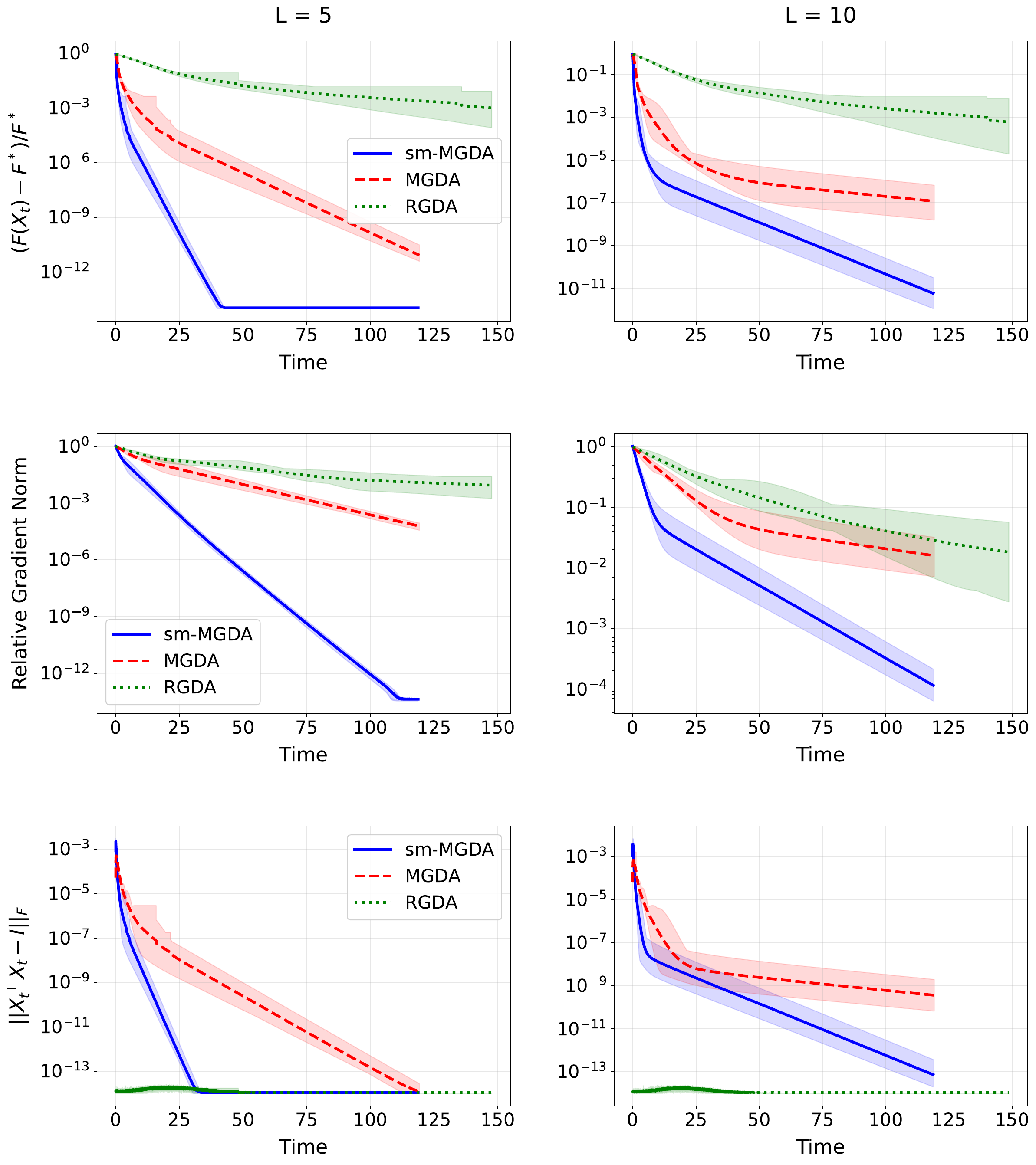}
    \caption{Comparison of \sm{} with MGDA and RGDA on \eqref{toy-example} for $\mu = 0$ with vector variables $x\in\reals^{n}$ with $n=500$ for $L\in\{5,10\}$. 
    % The top row shows the relative primal loss, the middle row shows the gradient norm, and the bottom row shows the  manifold constraint violation.
    }
    \label{fig:toy-example}
\end{minipage}
\hfill
\begin{minipage}{0.48\textwidth}
    \centering
    \includegraphics[width=\linewidth]{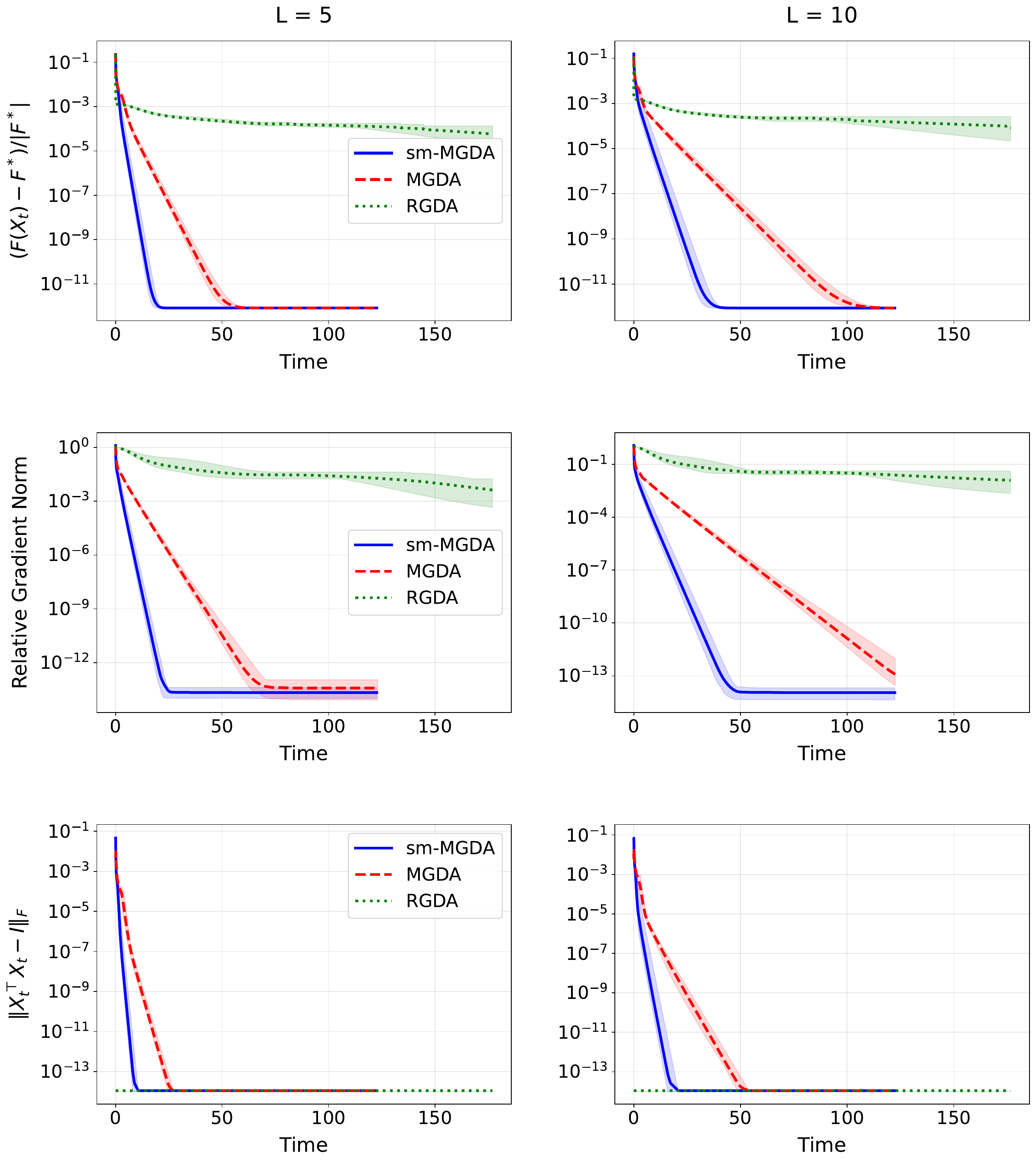}
    \caption{Comparison of \sm{} with MGDA and RGDA on \eqref{eq:toy-stiefel} for $\mu =0$ with matrix variables $x\in\mathrm{St}(500,450)$ for for $L\in\{5,10\}$. 
    % The top row shows the relative primal loss, the middle row shows the gradient norm, and the bottom row shows the manifold constraint violation.
    }
    \label{fig:toy-example2}
\end{minipage}
\end{figure}

\subsubsection{Case 2: Matrix variables}
In the second scenario, we consider the following quadratic $\mathrm{NCMC}$ minimax problem:
\begin{equation}
\label{eq:toy-stiefel}
\min_{x \in \mathrm{St}(n,k)}\ \max_{y \in Y }\ 
\mathcal{L}(x, y)
\;=\; \tfrac{1}{2}\,\operatorname{tr}\!\big(x^\top Q x\big)\;+\;\operatorname{tr}\!\big(x^\top A y\big),
\end{equation}
where $A,Q \in \mathbb{R}^{n \times n}$ are symmetric matrices, $\mathrm{St}(n,k)=\{x\in\mathbb{R}^{n\times k}: x^\top x=I_k\}$ is the Stiefel manifold, and $Y = \{ y\in\reals^n :\ \|y\| \le 1 \}$.   
We set $n=500$ and $k=0.9n=450$. The matrices $A,Q\in\mathbb{R}^{500\times 500}$ are randomly generated such that
$A = V \Lambda_A V^{\top}$ and $Q = V \Lambda_Q V^{\top}$, where $V\in\mathbb{R}^{500\times 500}$ is orthogonal and $\Lambda_A,\Lambda_Q$ are diagonal. We scale
$\Lambda_Q \;=\; \frac{L\,\Lambda_Q^{0}}{\lVert \Lambda_Q^{0}\rVert_2}$ for $L\in\{5,10\}$,
where the diagonal of $\Lambda_Q^{0}$ is sampled uniformly from $[-1,1]$.
Maximizing over $ y \in Y$ yields the primal objective
$F(x)\;=\;\frac{1}{2}\,\operatorname{tr}(x^\top Q x)\;+\;\lVert A^\top x\rVert_*$,
where $\lVert\cdot\rVert_*$ denotes the nuclear norm. 
The parameters are chosen as 
$\beta=0.9$, $p=2L$, $\tau_1=\frac{1}{6L}$, $\tau_2=\frac{\tau_1}{3}$, $\rho=10$, and $C=1000$ for \sm, 
and we use the same $\tau_1$ and $\tau_2$ for MGDA, i.e., $\tau_1 = \frac{1}{6L}$ and $ \tau_2 = \frac{\tau_1}{3}$.  For \textsc{RGDA}, we set $\tau_1=\frac{1}{16(\mu_y+1)^2L}$ and $\tau_2=\frac{1}{L}$ according to \citep[Theorem~17]{lin2025two}. The results in Figure~\ref{fig:toy-example2} show that \sm\, converges the fastest among \textsc{MGDA} and \textsc{RGDA} in terms of primal values and gradient norms. Furthermore, the manifold constraint violations, i.e., $\norm{x_t^\top x_t-I}$, show that the iterates $x_t$ gradually land on the manifold.

\begin{figure}[htbp!]
\centering
\begin{minipage}{0.48\textwidth}
    \centering
    
    \includegraphics[width=\linewidth]{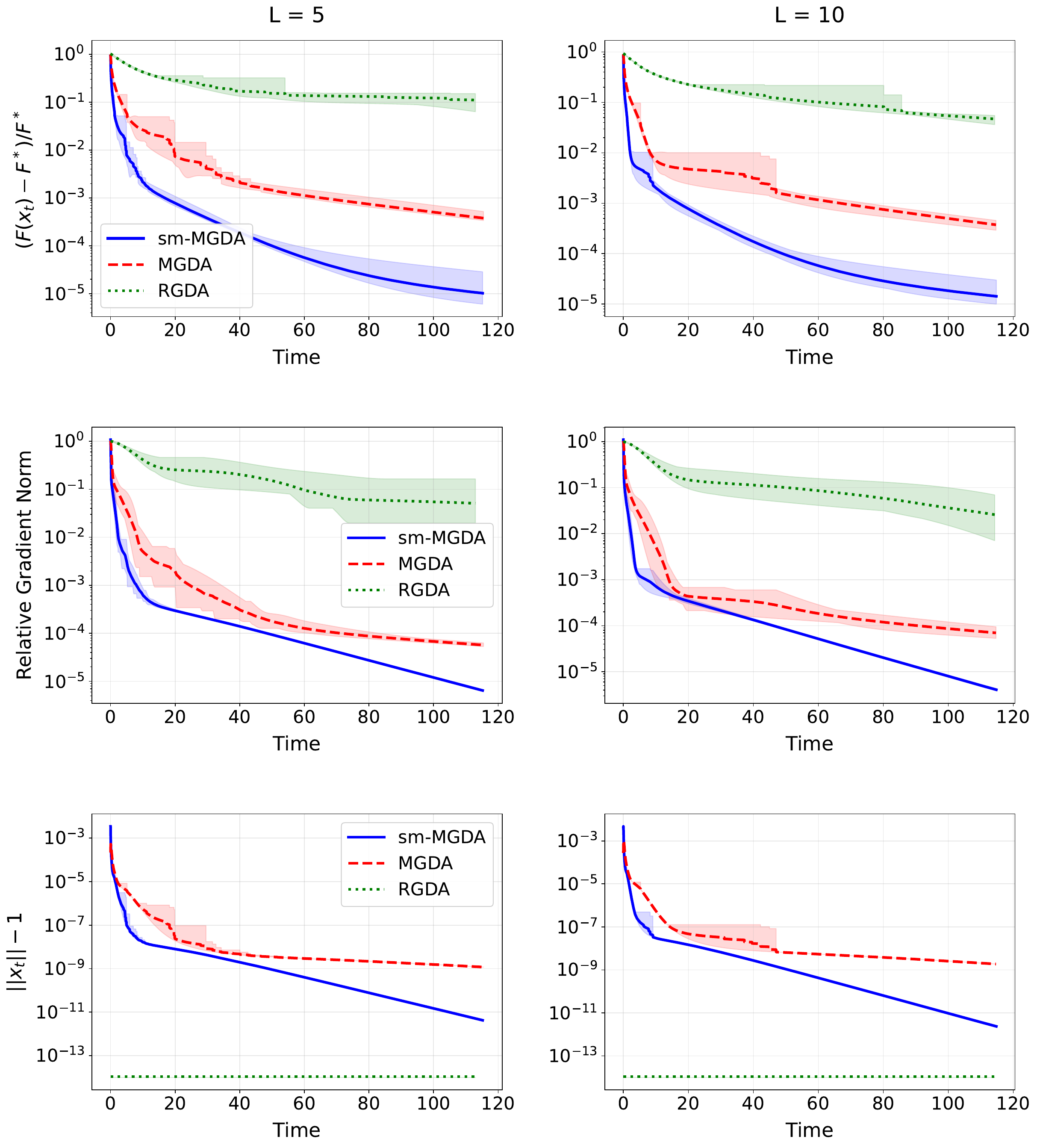}
    \caption{Comparison of \sm{} with MGDA and RGDA on \eqref{toy-example12} for $\mu = 1$ with vector variables $x\in\reals^{n}$ with $n=500$ for $L\in\{5,10\}$. 
    % The top row shows the relative primal loss, the middle row shows the gradient norm, and the bottom row shows the  manifold constraint violation.
    }
    \label{fig:toy-example12}
\end{minipage}
\hfill
\begin{minipage}{0.48\textwidth}
    \centering
    
    \includegraphics[width=\linewidth]{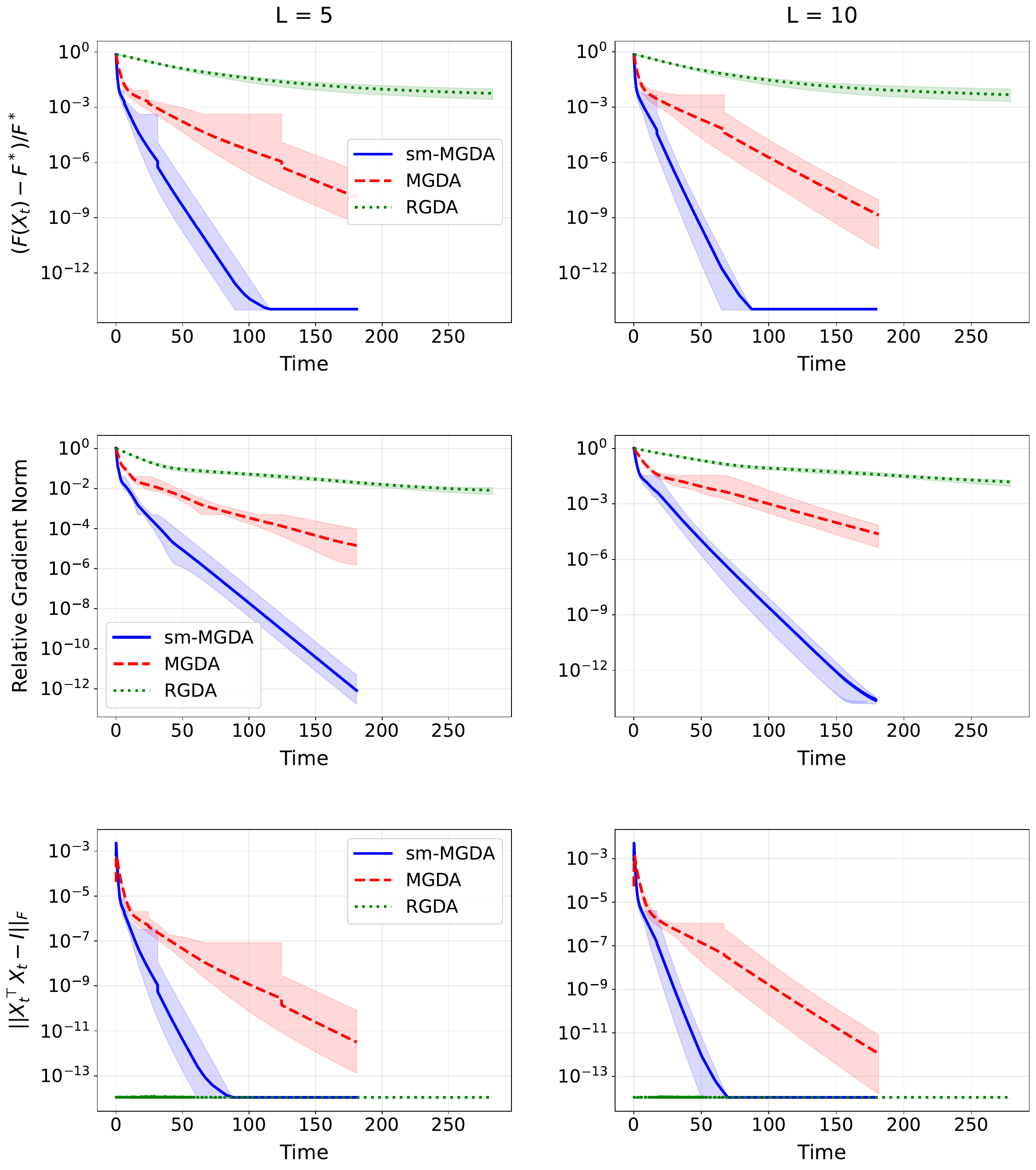}
    \caption{Comparison of \sm{} with MGDA and RGDA on \eqref{toy-example22} for $\mu =1$ with matrix variables $x\in\mathrm{St}(500,450)$ for for $L\in\{5,10\}$. 
    % The top row shows the relative primal loss, the middle row shows the gradient norm, and the bottom row shows the manifold constraint violation.
    }
    \label{fig:toy-example22}
\end{minipage}
\end{figure}
\subsection{Quadratic NCSC problem ($\mu>0$)}
\subsubsection{Case 1: Vector variables}
Again in the first case, we consider the following quadratic NCSC minimax problem:
\begin{equation} \label{toy-example12}
\min_{x \in \mathcal{M}}\max_{y\in\reals^n} \mathcal{L}(x, y) = \frac{1}{2} x^\top Q x + x^\top A y - \frac{\mu_y}{2} \|y\|^2,
\end{equation}
where $\mathcal{M} = \{x \in \mathbb{R}^n:\ x^{\top}x = 1 \}$ denotes the sphere manifold, $ A,Q \in \mathbb{R}^{n \times n}$ are symmetric matrices and $\mu_y > 0$. This class of problems includes robust principal component analysis~\citep{jordan2022first}, PL game~\citep{chen2022faster,zhang2024jointly}, image processing~\citep{chambolle2011first}, and robust regression~\citep{xu2008robust}. 
Except for setting $\mu_y = 1$ and removing the constraint on $y$, we randomly generate both $Q$ and $A$ as described in Section \ref{ncc} with $n=500$. Given $x\in\reals^n$, we can compute $y^*(x) = \frac{1}{\mu_y } A x$. Consequently, the primal loss function is
$\rv{F} (x) = \frac{1}{2} x^\top \left(Q + \frac{1}{\mu_y} A^2 \right) x$.
Since $x\in\reals^n$ lies in the sphere manifold, the minimum value of $F(\cdot)$ over $\cM$ is $F^* = \frac{1}{2}\lambda_{\min} \left(Q + \frac{1}{\mu_y}A^2 \right)$.
%\nsa{Check these calculations!!} 
We adopt the same hyperparameters from Section \ref{ncc} for the three algorithms. The results are shown in Figure \ref{fig:toy-example2} for $10$ randomly generated problems.  Similar to Section \ref{ncc}, our \sm{} performs better than MGDA and RGDA.
% \vspace{-0.3cm}
\subsubsection{Case 2: Matrix variables}
In the second case, we consider the following quadratic NCSC minimax problem:
\begin{equation}
\label{toy-example22}
\min_{x \in \mathrm{St}(m,k)}\ \max_{y \in \mathbb{R}^{n \times k}}\ 
\mathcal{L}(x, y)
\;=\; \tfrac{1}{2}\,\operatorname{tr}\!\big(x^\top Q x\big)\;+\;\operatorname{tr}\!\big(x^\top A y\big)\;-\;\tfrac{\mu_y}{2}\,\|y\|_F^2,
\end{equation}
where $\mathrm{St}(n,k)=\{x\in\mathbb{R}^{n\times k}: x^\top x=I_k\}$ is the Stiefel manifold,   $A,Q\in\mathbb{R}^{n\times n}$ symmetric matrices and $\mu_y>0$. 
% and generate them as 
% $A = V \Lambda_A V^{\top}$ and $Q = V \Lambda_Q V^{\top}$, 
% where $V\in\mathbb{R}^{n\times n}$ is orthogonal and $\Lambda_A,\Lambda_Q$ are diagonal,
% $n=500$ and $k=0.9n$. 
% We scale
% $
% \Lambda_Q \;=\; \frac{L\,\Lambda_Q^{0}}{\lVert \Lambda_Q^{0}\rVert_2},  L\in\{5,10\},
% $
% where the diagonal of $\Lambda_Q^{0}$ is sampled i.i.d.\ from $[-1,1]$ and $\lVert\cdot\rVert_2$ denotes the spectral norm.
Except for setting $\mu_y=1$ and removing the constraint on $y$, we randomly generate both $Q$ and $A$ as described in Section \ref{ncc} with $n=500$. We also set $k =0.9n=450$. 
Given $x\in\reals^{n\times k}$, maximizing \eqref{toy-example22} over $y\in\reals^{n\times k}$ yields 
$
y^*(x) \;=\; \frac{1}{\mu_y}\,A^\top x,
$
and thus the primal objective is
$
F(x)\;=\;\tfrac{1}{2}\,\operatorname{tr}\!\big(x^\top Q x\big)\;+\;\frac{1}{2\mu_y}\,\|A^\top x\|_F^2
\;=\;\tfrac{1}{2}\,\operatorname{tr}\!\Big(x^\top\!\Big(Q+\tfrac{1}{\mu_y}AA^\top\Big)x\Big).
$
Since $x \in \mathrm{St}(n,k)$, the minimum value is
$
F^* \;=\; \tfrac{1}{2}\,\sum_{i=1}^{k}\lambda_i\!\Big(Q+\tfrac{1}{\mu_y}AA^\top\Big).
$
We adopt the same hyperparameters from Section~\ref{ncc} for the three algorithms. 
The results are shown in Figure \ref{fig:toy-example22} for $10$ randomly generated problems.  
Similar to Section~\ref{ncc}, our \sm{} algorithm exhibits the fastest convergence behavior in terms of relative gradient norms and primal function values.
\subsection{Robust DNN training} \label{4-3}
We cast the original robust training against adversarial attacks problem into the following nonconvex-concave problem:
\be\label{eq:nonconvex_concave}
\min_{x \in \mathcal{M}} \max_{u \in \mathcal{U}} \frac{1}{n} \sum_{i=1}^{n} \sum_{j=1}^{C} u_j \ell\left(h\left(a_{ij}^K; x\right), b_i\right) - r(u), \;\;
\text{s.t.} \quad \mathcal{U} = \{ u \in \mathbb{R}^C \mid u \geq 0, \, \|u\|_1 = 1 \},
\ee
where $a_{ij}^K$ is the permuted sample after $K$ iterations of Projected Gradient Descent (PGD) attack~\citep{kurakin2017adversarial}, and $C$ is the number of classes for the dataset. Here $r(u)$ is a convex regularization term, e.g., $r(u) = \alpha \| u - 1/C \|^2$, where $\alpha \geq 0$ is a hyper-parameter. In the experiment, we use Stiefel manifold $\mathcal{M} = \mathrm{St}(r,d) = \left\{ X \in \mathbb{R}^{d \times r} \; : \; X^\top X = I_r \right\}$ on parameters $x$ of DNNs (convolution layers and linear layers). 
%see Table \ref{tab:dnn_architecture} in appendix for details.
The DNN architecture used in our experiments %Sections \ref{A-2} and \ref{4-3} 
is summarized in Table~\ref{tab:dnn_architecture}. To simplify the operation of the Retraction operator, the network structure calls \texttt{mctorch} package \citep{meghwanshi2018mctorch}.
\vspace*{-2mm}
\begin{table}[h]
\centering
\caption{The DNN architecture used in our experiments. $C$ is the number of classes, and $d$ is the number of channels for inputs.}
\label{tab:dnn_architecture}
\begin{tabular}{ll}
\toprule
\textbf{Layer} & \textbf{Configuration} \\
\midrule
Inputs & $d$ channels \\
Conv & $d \rightarrow 32$, Batchnorm, ReLU \\
Conv & $32 \rightarrow 64$, Batchnorm, ReLU \\
Conv & $64 \rightarrow 128$, Batchnorm, ReLU \\
Max Pool & \\
Linear & $512 \rightarrow 256$, Batchnorm, ReLU \\
Linear & $256 \rightarrow 128$, Batchnorm, ReLU  \\
Linear & $128 \rightarrow C$ \\
Outputs & \\
\bottomrule
\end{tabular}
\end{table}

\begin{figure}[h]
    \centering
    \includegraphics[width=1.0\linewidth]{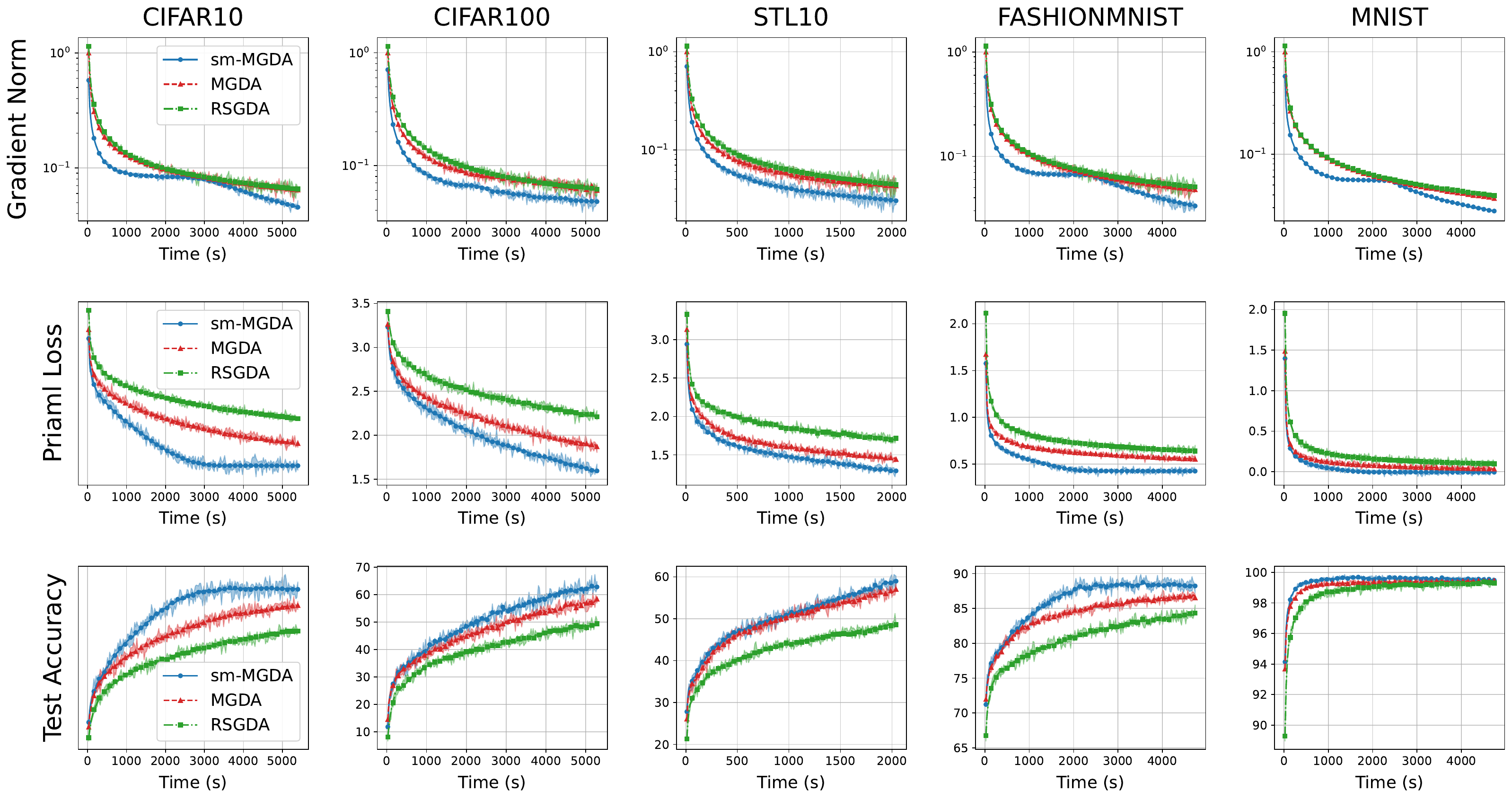}
    \caption{Gradient norm, primal loss, and test accuracy on robust DNN training problems %of tested algorithms 
    over 3 runs. }
    % \hj{Use log scale in y axis, Multiple run with different seeds from the same initial point, 10/20 runs? Plot the mean and the range around it. Focus on FashionMNIST first}}
    \label{fig:datasets-robust}
\end{figure}

 We choose five datasets for this experiment: CIFAR10, CIFAR100,  STL10, FashionMNIST, and MNIST. For the CIFAR100 dataset, we selected the data from the first 20 classes for training and testing.
% The DNN architecture and all hyperparameters are the same as those used in Section~\ref{A-2}.
We set $K = 5$, corresponding to five attacks on the data. 
The numerical results against different attacks (i.e., PGD
attack \citep{kurakin2017adversarial} and Fast Gradient Sign Method (FGSM) attack \citep{goodfellow2014explaining}
) are shown in Table \ref{tab:attack_accuracy}.  We set $\tau_1 = \tau_2 = 10^{-3}$, $\beta =0.9$, $p = 1$, $\rho = 10, C = 1000$ for \sm{} with the same $\tau_1$ and $\tau_2$ for MGDA and RSGDA.  In \sm{}, we use an inexact form of $g_t$, given by $\nabla A(x_t)[\nabla_x f(x_t, y_t)] + \frac{\rho}{2}\nabla c(x_t)[c(x_t)]$, to simplify computation. This approximation quality is controlled by the small difference between $x_t$ and $A(x_t)$ when $x_t$ lies close to the manifold. The detailed training trajectories are listed in Figure \ref{fig:datasets-robust}. It is shown that RSGDA has worse performance compared with \sm{} and MGDA, while \sm{} has the best performance in the sense of gradient norm, primal loss, and test accuracy. This demonstrates the enhanced robustness and practical utility of \sm{} in robust DNN training. The results of the test accuracy of the compared algorithms under the attack are summarized in Table \ref{tab:attack_accuracy}, which demonstrates that \sm{} has the highest test accuracy compared with MGDA and RSGDA. 

The manifold error of the model with epoch on robust DNN training task is shown in Figure \ref{fig:stiefel_constraints3}. It is shown that the error has a tendency to decrease with the number of epochs, indicating that the parameter falls on the manifold.

\begin{figure}[h]
    \centering
    \includegraphics[width=1.0\linewidth]{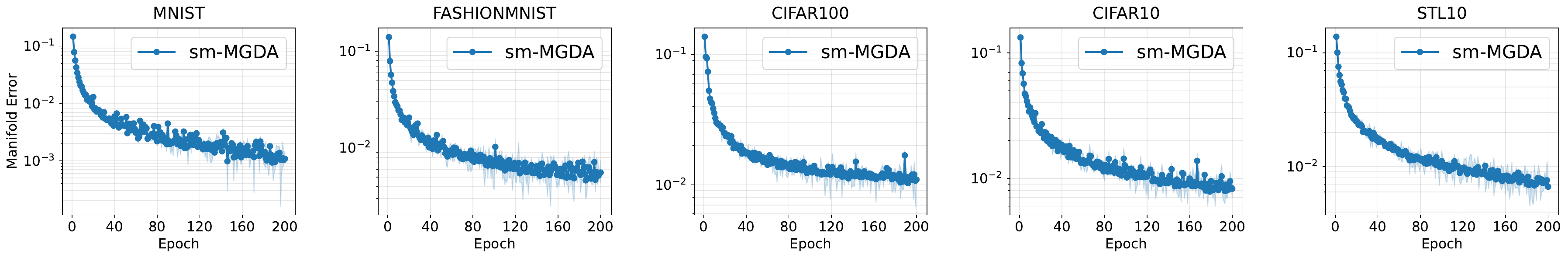}
    \caption{Manifold error of the model with epoch on robust DNN training task.}
    \label{fig:stiefel_constraints3}
\end{figure}

\begin{table}[bht!]
\centering
\small
\setlength{\tabcolsep}{4pt}
\caption{Test accuracy against nature images and different attacks for (from top to bottom) \textbf{CIFRA10, CIFAR100, STL10, FashionMNIST, MNIST} datasets.}
\begin{tabular}{cc|ccc|ccc}
\hline
\textbf{Methods} & \textbf{Original Image} & \multicolumn{3}{c|}{\textbf{PGD$^{40}$ $L_\infty$}} & \multicolumn{3}{c}{\textbf{FGSM $L_\infty$}} \\
\cline{3-8}
% \textbf{Methods} & \textbf{Original Image} & \multicolumn{3}{c|}{\textbf{PGD$^{40}$ $L_\infty$}} & \multicolumn{3}{c}{\textbf{FGSM $L_\infty$}} \\
% \cline{3-8}
% \textbf{Methods} & \textbf{Original Image} & \multicolumn{3}{c|}{\textbf{PGD$^{40}$ $L_\infty$}} & \multicolumn{3}{c}{\textbf{FGSM $L_\infty$}} \\
% \cline{3-8}
& & $\epsilon=0.005$ & $\epsilon=0.01$ & $\epsilon=0.02$ & $\epsilon=0.005$ & $\epsilon=0.01$ & $\epsilon=0.02$ \\
\hline
MGDA       & 75.06\% & 72.74\% & 70.48\% & 66.01\% & 72.76\% & 70.56\% & 66.41\% \\
RSGDA        & 66.22\% & 64.04\% & 61.31\% & 56.19\% & 64.06\% & 61.44\% & 56.77\% \\
\textbf{\sm{}}  &  \textbf{80.22\%} & \textbf{77.73\%} & \textbf{75.06\%} & \textbf{69.11\%} & \textbf{77.78\%} & \textbf{75.29\%} & \textbf{70.12\%} \\
\hline
\hline
% \textbf{Methods} & \textbf{Original Image} & \multicolumn{3}{c|}{\textbf{PGD$^{40}$ $L_\infty$}} & \multicolumn{3}{c}{\textbf{FGSM $L_\infty$}} \\
% \cline{3-8}
 & & $\epsilon=0.005$ & $\epsilon=0.01$ & $\epsilon=0.02$ & $\epsilon=0.005$ & $\epsilon=0.01$ & $\epsilon=0.02$ \\
\hline
MGDA       & 58.40\% & 55.15\% & 51.80\% & 45.50\% & 55.15\% & 52.10\% & 46.15\% \\
RSGDA        & 47.20\% & 47.27\% & 43.10\% & 31.80\% & 49.98\% & 46.67\% & 37.40\% \\
\textbf{\sm{}}  & \textbf{62.35\%} & \textbf{59.55\%} & \textbf{57.00\%} & \textbf{50.90\%} & \textbf{59.55\%} & \textbf{56.95\%} & \textbf{50.90\%} \\
\hline
\hline
% \textbf{Methods} & \textbf{Original Image} & \multicolumn{3}{c|}{\textbf{PGD$^{40}$ $L_\infty$}} & \multicolumn{3}{c}{\textbf{FGSM $L_\infty$}} \\
% \cline{3-8}
& & $\epsilon=0.005$ & $\epsilon=0.01$ & $\epsilon=0.02$ & $\epsilon=0.005$ & $\epsilon=0.01$ & $\epsilon=0.02$ \\
\hline
MGDA       & 57.04\% & 54.26\% & 52.18\% & 47.02\% & 54.26\% & 54.18\% & 46.19\% \\
RSGDA        & 51.86\% & 47.90\% & 45.85\% & 34.90\% & 53.95\% & 51.10\% & 46.05\% \\
\textbf{\sm{}}  & \textbf{58.49\%} & \textbf{55.85\%} & \textbf{53.29\%} & \textbf{47.88\%} & \textbf{55.85\%} & \textbf{53.41\%} & \textbf{48.76\%} \\
\hline
\hline
& & $\epsilon=0.05$ & $\epsilon=0.1$ & $\epsilon=0.2$ & $\epsilon=0.05$ & $\epsilon=0.1$ & $\epsilon=0.2$ \\
\hline
MGDA       & 86.53\% & 80.70\% & 72.82\% & 49.08\% & 81.12\% & 77.70\% & 67.77\% \\
RSGDA      & 84.72\% & 79.40\% & 70.52\% & 51.46\% & 78.41\% & 70.52\% & 61.15\% \\
\textbf{\sm{}}  & \textbf{87.73\%} & \textbf{82.12\%} & \textbf{75.11\%} & \textbf{55.80\%} & \textbf{82.85\%} & \textbf{77.75\%} & \textbf{67.87\%} \\
\hline
\hline
& & $\epsilon=0.1$ & $\epsilon=0.2$ & $\epsilon=0.4$ & $\epsilon=0.1$ & $\epsilon=0.2$ & $\epsilon=0.4$ \\
\hline
MGDA       & 99.29\% & 98.21\% & 95.39\% & 83.43\% & 98.53\% & 97.03\% & 92.41\% \\
RSGDA      & 99.26\% & 98.11\% & 95.83\% & 84.76\% & 93.05\% & 91.26\% & 90.81\% \\
\textbf{\sm{}}  & \textbf{99.45\%} & \textbf{98.60\%} & \textbf{96.69\%} & \textbf{87.42\%} & \textbf{98.69\%} & \textbf{97.62\%} & \textbf{92.91\%} \\
\hline
\end{tabular}
\label{tab:attack_accuracy}
\end{table}

\section{Conclusion}
\vspace{-0.2cm}
We propose a retraction-free smoothed manifold gradient descent-ascent (\sm{}) method for solving minimax problems over compact submanifolds through solving an exact penalized problem with an additional norm constraint. \rv{We establish convergence guarantees under nonsmooth PL condition for weakly convex functions, and more specifically under both merely and strongly concave settings,} and show that the algorithm asymptotically recovers stationarity of \rv{the original manifold minimax problem.} 
%under mild coercivity conditions. 
Numerical results on robust training and risk-sensitive learning tasks demonstrate the practical advantages of our approach over existing methods.

Our work motivates several interesting directions. First, a natural next step is to extend our algorithm and analysis to general submanifolds defined by smooth equality constraints $\{x : c(x) = 0\}$. %Another direction is to derive a cleaner and more principled formulation of the norm constraint set $C$. 
Finally, we also plan to study the stochastic version of \sm{}, which is essential for solving large-scale manifold minimax problems.

\bibliographystyle{plain}
\bibliography{references}

\newpage
\section{Proofs}
\subsection{Proof of Lemma~\ref{lem:lip-f}}
\label{sec:Lipschitz-proof}
%\rv{\bf Note: We define Lipschitz constants based on Frobenius norm. Say $\cM$ is the Stiefel manifold (d x r), then $\sup\{\norm{x}: x \in \cM\} = \sqrt{r}$. See lemma 1, there $l_{yx} = \cO(C^2 L_{yx}) = \cO(r L_{yx})$. On the other hand, if we could use $\norm{\cdot}_2$ for defining $C$, then $C=\cO(1)$. Please let me know if the analysis would go through if we use $\sup\{\norm{x}_2: x \in \cM\}$ in the definition of $C$ rather than $\norm{\cdot}$ and obtain the Lipschitz constants corresponding to $\tilde f$ accordingly.}

\sa{Note that for any $(x,y)\in X\times Y$,} we have $\nabla_y \tilde{f}(x,y) = \nabla_y f(A(x), y)$ and 
\[ \nabla_x \tilde{f}(x,y) = \nabla A(x)[\sa{\nabla_x} f(A(x), y)] + \rho x(x^\top x - I_r),   \]
where $\sa{\nabla} A(x) [u] := u (\frac{3}{2}I_r - \frac{1}{2}x^\top x) - x {\rm sym}(x^\top u)$.
% \nsa{should not it be $-x {\rm sym}(x^\top u)$?}\nsa{What norm is $\norm{x}$?}
\sa{Since $\|x\| \leq C$, 
%and that $\nabla_x f(x, y)$ and $\nabla_y f(x, y)$ are Lipschitz continuous with constants $L_{xx}$, $L_{xy}$, $L_{yx}$, and $L_{yy}$. It holds that
\cref{assum:lip-f} implies that}
\[ \|\nabla A(x)\|_{\rm op}:= \max_{u \ne 0}\frac{\|\nabla A(x) [u]\|}{\|u\|}  \leq \frac{3}{2} + \frac{1}{2}\|x^\top x\| + %\frac{1}{2}
\sa{\|x\|^2} \leq \frac{3}{2} + \frac{3}{2} \|x\|^2 \leq \frac{3}{2} + \frac{3}{2}C^2 \]
\sa{for all $x\in X$,} where we use $\|xy\|\leq \|x\|_2\|y\| \leq \|x\|\|y\|.$

\textit{Lipschitz Continuity of $\nabla_y \tilde{f}$.}  
It directly follows from the expression of \sa{$\nabla_y \tilde{f}$ that we have
    $l_{yy} = L_{yy}$.} 
    
    \rv{Next, we consider $l_{yx}$.} \sa{%Moreover, 
    Given $x_1,x_2\in X$ and $y\in Y$, we have 
    \begin{align*}
        \MoveEqLeft \norm{\nabla_y\tilde f(x_1,y)-\nabla_y\tilde f(x_2,y)}\\
        &=\norm{\nabla_y f(A(x_1),y)-\nabla_y f(A(x_2),y)}\leq L_{yx}\norm{A(x_1)-A(x_2)}\leq \frac{3}{2}L_{yx}(1+C^2)\norm{x_1-x_2},
    \end{align*}
    where we used the uniform bound on $\norm{\nabla A(\cdot)}_{\rm op}$. Thus, $l_{yx} = \frac{3}{2}(1+C^2) L_{yx}$.} 
    % Differentiating with respect to $x$ gives 
    % $
    % \nabla_x \nabla_y \tilde{f}(x, y) = \nabla_x A(x) [\nabla^2_{xy} f(A(x), y)]
    % $.
    % Thus, 
    % \[
    % L_{yx}^{\tilde{f}} = \frac{3}{2}(1 +C^2) L_{yx}.
    % \]

\textit{Lipschitz Continuity of $\nabla_x \tilde{f}$.}  
% \nsa{Let's prove this without using twice differentiability as I argued above.}
From the expression of \rv{$\nabla_x \tilde{f}$, for any $x\in X$ and $y_1,y_2\in Y$,} we have
    \[ \begin{aligned}
       \| \nabla_x \tilde{f}(x,y_1) - \nabla_x \tilde{f}(x,y_2)\| = & \| \nabla A(x) [ \nabla_x f(A(x), y_1) - \nabla_x f(A(x), y_2)] \| \\
       \leq &  \|\nabla A(x)\|_{\rm op} L_{xy} \|y_1 - y_2\| \leq \frac{3}{2}(1+C^2)L_{xy} \|y_1 -y_2\|,  
    \end{aligned}
\]
which implies that $l_{xy} = \frac{3}{2}(1+C^2) L_{xy}$.

\sa{Next, we consider $l_{xx}$.} 
Define $\sa{s(x)} := x(x^\top x - I_r)$. Then $\nabla s(x)[u] = u (x^\top x - I_r) + x(u^\top x + x^\top u)$, which implies that
$ \|\nabla s(x) \|_{\rm op} = \max_{u \ne 0} \frac{\|\nabla s(x)[u]\|}{\|u\|} \leq C^2 +1  + 2C^2 = 3C^2 + 1$ for all $x \in X$. \rv{Moreover, let $\bar l_x:= \max_{x \in X, y \in Y} \|\nabla_x f(A(x), y)\|$. Since $A(x) := x(\frac{3}{2}I_r - \frac{1}{2}x^\top x)$, we also have $\norm{A(x)}\leq C\frac{1}{2}(3+C^2)$ for all $x\in X$. Therefore, $\bar l_x\leq \max\{\norm{\nabla_x f(x,y)}:\  \|x\| \leq C(\frac{3}{2} + \frac{1}{2}C^2),\ y \in Y\}$. %Noting that 
For any $x_1,x_2\in X$ and $y\in Y$,} we have
\[ \begin{aligned}
    & \| \nabla_x \tilde{f}(x_1,y) - \nabla_x \tilde{f}(x_2, y) \| \\
    = & \|\nabla A(x_1) [\nabla_x f(A(x_1), y)] - \nabla A(x_2) [\nabla_x f(A(x_2), y)] \| + \rho \|s(x_1) - s(x_2)\| \\
    \leq & \|\Big(\nabla A(x_1)- \nabla A(x_2)\Big) [\nabla_x f(A(x_1), y)] %- \nabla A(x_2)[\nabla_x f(A(x_1), y)]
    \| + \|\nabla A(x_2)[\nabla_x f(A(x_1), y) - \nabla_x f(A(x_2), y)]  \| \\
    & + 
    % \sa{\rho \|\nabla s(x')\|_{\rm op}} \|x_2 - x_1\| 
    \rho \max_{x \in X} \|\nabla s(x)\|_{\rm op} \|x_2 - x_1\|
    \\ 
    \leq & \rv{3}C \bar l_x \|x_2 -x_1\| + \left(\frac{9}{4}(1+C^2)^2 \right) L_{xx} \|x_2 - x_1\| + \rho(3C^2 + 1)\|x_2 - x_1\| \\
    \leq &\left( %5
    \rv{3}C \bar l_x + \frac{9}{4}(1+C^2)^2L_{xx} + \rho(3C^2 +1)\right) \|x_2 - x_1\|, 
\end{aligned}
\] 
%where 
% \sa{$x'$ is some point belonging to line segment $[x_1,x_2]$,} 
%the last inequality uses 
\rv{where we used the definition of $\bar l_x$ and}
\[ \begin{aligned}
   \max_{\|x_1\| \leq C, \|x_2\| \leq C} \| \Big(\nabla A(x_1)-\nabla A(x_2)\Big)[u] %- \nabla A(y)[u] 
   \| 
   \leq & %\frac{3}{2} \|xx^\top - yy^\top \| \|u\| 
   \rv{\frac{1}{2} \|x_1^\top x_1 - x_2^\top x_2 \|} \|u\| + (\|x_1\| + \|x_2\|)\|x_1 - x_2\|\|u\| 
   \\
   = & \rv{\frac{1}{2} \|(x_1 -x_2)^\top x_1 + x_2^\top (x_1-x_2) \|} \| u\|  %\frac{3}{2} \|(x -y)x^\top + y(x-y)^\top \|\| u\| 
   + (\|x_1\| + \|x_2\|)\|x_1 - x_2\|\|u\|  \\
   \leq & \rv{3C \|x_1 -x_2 \|}\|u\|.
\end{aligned}
 \]
Thus, we have $l_{xx} = \rv{3}C \bar l_x + \frac{9}{4}(1+C^2)^2L_{xx} + \rho(3C^2 +1)$.
\qed
%\end{proof}
\subsection{Proof of Lemma \ref{lem:equiv-stationary}}
\label{sec:eps-stationary-connection-proof}
%\begin{proof}[Proof of Lemma \ref{lem:equiv-stationary}]
    %Noting that 
    \rv{Let $x_\epsilon = usv^\top$ denote the compact SVD of $\xe$ and $\reals^r_+\ni \sigma={\rm diag}(s)$. Since ${\rm dist}(x_\epsilon, \Mcal) \leq \frac{1}{2}$, we have $\frac{1}{2}\leq \sigma_i\leq \frac{3}{2}$ for $i=1,\ldots, r$, which also implies that $\frac{1}{2}\leq \norm{x_\epsilon}_2\leq \frac{3}{2}$.}
    \rv{Since} $\Mcal$ is the Stiefel manifold and $A(x) = x (\frac{3}{2} I_r - \frac{1}{2} x^\top x)$ \sa{for $x\in X$}, %we have from 
    \rv{\citep[Proposition 2.8]{xiao2024solving2} implies that} 
    \[ \nabla_x \tilde{f}(x, y) = G(x,y)
    %\nabla_x f(A(x),y) \Big(\frac{3}{2} I_r - \frac{1}{2} x^\top x\Big) - x~{\rm sym}(x^\top \nabla_x f(A(x),y)) 
    + \rho x(x^\top x - I_r),\quad\forall~(x,y)\in X\times Y,\]
    \rv{where} $G(x,y)\rv{:=}  \nabla_x f(A(x),y) (\frac{3}{2} I_r - \frac{1}{2}x^\top x) - x{\rm sym}(x^\top \nabla_x f(A(x),y))$. 
    %and taking squared norm on both sides yield
    \rv{Then, for $(\xe,\ye)$ %it trivially holds that
    we get}
    \begin{equation}
    \label{eq:grad_x-lower-bound}
        \begin{aligned}
        \|\rv{\nabla_x} \tilde{f}(\xe,\ye)\|^2 = & \| G(\xe,\ye) \| ^2 + \rho^2 \|\xe(\xe^\top \xe - I_r) \|^2  + \rv{2\rho} \iprod{G(\xe,\ye)}{\xe(\xe^\top \xe - I_r)} \\ 
        = & \| G(\xe,\ye) \| ^2 + \rho^2 \|\xe(\xe^\top \xe - I_r) \|^2  - %\frac{3\rho}{2}
        \rv{3\rho}\iprod{{\rm sym}(\xe^\top\nabla_x f(A(\xe),\ye))}{(\xe^\top \xe - I_r)^2} \\
        \geq & \| G(\xe,\ye) \| ^2 + \rho^2 \|\xe(\xe^\top \xe - I_r) \|^2  - %\frac{3\rho}{2}
        \rv{3\rho}\|\xe^\top\nabla_x f(A(\xe),\ye)\| \|\xe^\top \xe - I_r \|^2 \\
        % \geq & \| G(x,y) \| ^2 + \beta^2 \|x(x^\top x - I_r) \|^2  - \frac{3\beta}{2}\|x^\top\nabla_x f(x,y)\| \|x^\top x - I_r \|^2 \\
        \geq & \| G(\xe,\ye) \| ^2 + \rho^2 \|\xe(\xe^\top \xe - I_r) \|^2  - \mod{3\rho \cdot} \frac{3}{2} \rv{\tilde L_x} \cdot \mod{4} \|\xe(\xe^\top \xe - I_r)\|^2 \\
        = & \| G(\xe,\ye) \| ^2 + \rho(\rho - {18} \rv{\tilde l_x}) \|\xe(\xe^\top \xe - I_r) \|^2 \\
        \geq & \| G(\xe,\ye) \| ^2 + \frac{\rho^2}{2}  \|\xe(\xe^\top \xe - I_r) \|^2, 
    \end{aligned} 
    \end{equation} 
    where \rv{$\tilde L_x:=L_x(\ye) = \max\{\|\nabla_x f(x,\ye)\|:\ \|x\|_2 \leq 1\}$.} Here, the second %line
    \rv{equation} follows from the definition of $G(\cdot,\cdot)$ \rv{and the fact that $\fprod{A,B}=\fprod{{\rm sym}(A),B}$ whenever $B$ is symmetric, the first inequality is due to $\norm{{\rm sym}(A)}\leq \norm{A}$ and $\fprod{A,B}\leq\norm{A}\norm{B}$. Next, we argue for the the second inequality. Note that $$\|A(\xe)\|_2 =\frac{1}{2} \max_{i=1,\ldots,r} \sigma_i(3-\sigma_i^2)\leq \frac{1}{2}\max\Big\{t(3-t^2):\ t\in \Big[\frac{1}{2},\frac{3}{2}\Big]\Big\}=1.$$ 
    %and $\norm{\sigma}_\infty\leq 3/2$ implies that $3I_r-s^2$ is a psd matrix. Therefore, $\|3I_r-s^2\|_2\leq 3$, which implies $\|A(x)\|_2 \leq \frac{3}{2}\|x\|_2 \leq \frac{9}{4}$. 
    Thus, $\|\nabla_x f(A(x_\epsilon),\ye)\|\leq \tilde L_x$. Consequently, the second inequality follows from $\|\xe(\xe^\top \xe - I_r)\| \geq \frac{1}{2} \|\xe^\top \xe - I_r\|$ where we use $\norm{AB}\geq \sigma_{\rm min}(A) \cdot \norm{B}$, and $\sigma_{\rm min}(A)\geq 0$ denotes the minimum singular value of $A$}. Finally, the last line uses $\rho \geq {36} \rv{\tilde L_x}$.  
    
    In addition, it holds $\Pcal_{\Mcal}(\xe) = uv^\top$, and \rv{using $\|\xe(\xe^\top \xe - I_r)\| \geq \frac{1}{2} \|\xe^\top \xe - I_r\|$, we also get}
    \[  \|\xe - \Pcal_{\Mcal}(\xe)\| = \|usv^\top - uv^\top\| = \|s - I_r\| \leq \|s^2 - I_r\| = \|\xe^\top \xe - I_r\| \leq 2 \| \xe (\xe^\top \xe - I_r) \|.   \]
    Thus, we get $\|\xe - \Pcal_{\Mcal}(\xe)\| \leq \frac{3}{\rho} \epsilon$, which follows from \rv{using \eqref{eq:grad_x-lower-bound} together with}
    \begin{equation}
    \label{eq:distance-control}
    \|\xe- \Pcal_{\Mcal}(\xe)\|^2  \leq 4\|\xe(\xe^\top \xe- I_r)\|^2 \leq \frac{8}{\rho^2} \|\rv{\nabla_x} \tilde{f}(\xe,\ye)\|^2\leq \rv{\frac{9}{\rho^2}\epsilon^2}.
    \end{equation}
    
   \rv{Next, we provide a bound on $\norm{\grad_x f(\Pcal_{\Mcal}(\xe), \ye)}$. Using the definition of $G(\xe,\ye)$ and the expression for the Riemannian gradient given in \eqref{eq:Riemannian_grad}, we get}
    \begin{align}
    \MoveEqLeft \| \rv{\grad_x} f(\Pcal_{\Mcal}(\xe), \ye) - G(\xe,\ye)\| \nonumber\\
        =&  \|\nabla_x f(\Pcal_{\Mcal}(\xe),\ye) - \Pcal_{\Mcal}(\xe){\rm sym}(\Pcal_{\Mcal}(\xe)^\top \nabla_x f(\Pcal_{\Mcal}(\xe) ,\ye)) \nonumber\\
        & -\nabla_x f(A(\xe),\ye)\Big(\frac{3}{2}I_r - \frac{1}{2} \xe^\top \xe\Big) - \xe{\rm sym}(\xe^\top \nabla_x f(A(\xe),\ye)) \| \nonumber\\
        \leq & \frac{1}{2}\|\nabla_x f(A(\xe),\ye)\| \|\xe^\top \xe - I_r\| + L_{xx} \| \Pcal_{\Mcal}(\xe) - A(\xe)\| \nonumber\\
        & +(\Pcal_{\Mcal}(\xe) - \xe) {\rm sym}\Big(\Pcal_{\Mcal}(\xe)^\top \nabla_x f(\Pcal_{\Mcal}(\xe) ,\ye)\Big)\nonumber\\ 
        &+ \xe {\rm sym}\Big(\Pcal_{\Mcal}(\xe)^\top \nabla_x f(\Pcal_{\Mcal}(\xe) ,\ye) - \xe^\top \nabla_x f(A(\xe),\ye)\Big) \nonumber\\
        \leq & \frac{1}{2}\|\nabla_x f(A(\xe),\ye)\| \|\xe^\top \xe - I_r\| + L_{xx} \| \Pcal_{\Mcal}(\xe) - A(\xe)\| + \rv{\tilde L_x} \| \Pcal_{\Mcal}(\xe) - \xe \| \nonumber\\
        & +\rv{\|\xe\|_2} \Big(\| \Pcal_{\Mcal}(\xe) - \xe \| \| \nabla_x f(\Pcal_{\Mcal}(\xe) ,\ye)) \| + \rv{\|\xe\|_2} \| \nabla_x f(\Pcal_{\Mcal}(\xe) ,\ye) - \nabla_x f(A(\xe),\ye)\|\Big) \nonumber\\
        \leq & \frac{1}{2}\|\nabla_x f(A(\xe),\ye)\| \|\xe^\top \xe - I_r\| + L_{xx} \| \Pcal_{\Mcal}(\xe) - A(\xe)\| + \rv{\tilde L_x} \| \Pcal_{\Mcal}(\xe) - \xe \| \nonumber\\
        & +\frac{3}{2}\rv{\tilde L_x} \| \Pcal_{\Mcal}(\xe) - \xe \| + \frac{9}{4} L_{xx} \| \Pcal_{\Mcal}(\xe) -A(\xe) \|. \label{eq:gradx-G-bound}
    \end{align}
    \rv{Furthermore, since $x_\epsilon = usv^\top$, we also have} $A(\xe) = \rv{\frac{1}{2}}us(\rv{3}I_r - s^2)v^\top$; \rv{hence,}  
    \be \label{eq:bound-ax}
    \begin{aligned}
    \MoveEqLeft \| \Pcal_{\Mcal}(\xe) - A(\xe)\| = \|u\Big(I_r-\frac{1}{2}s(3I_r - s^2)\Big)v^\top \| \\
    & = \|(s - I_r)\Big(\frac{1}{2}(s^2 + s)-I_r\Big) \| \leq \|s- I_r\| = \| \Pcal_{\Mcal}(\xe) - \xe\|,
    \end{aligned}
    \ee
    \rv{where the inequality follows from $\norm{\frac{1}{2}(s^2 + s)-I_r}_2\leq 1$, which is implied by $\frac{1}{2}\leq \sigma_i\leq \frac{3}{2}$ for $i=1,\ldots, r$. Moreover, also note that $\|\xe(\xe^\top \xe - I_r)\| \geq \frac{1}{2} \|\xe^\top \xe - I_r\|$ together with \eqref{eq:distance-control} gives us $\norm{\xe^\top\xe-I_r}\leq \frac{3}{\rho}$. Finally, \eqref{eq:grad_x-lower-bound} implies that $\norm{G(\xe,\ye)}\leq\epsilon$. Therefore, using \eqref{eq:gradx-G-bound}, we get}
    \begin{align*}
        \|\grad_x f(\Pcal_{\Mcal}(\xe), \ye) \| 
        &\leq \| G(\xe,\ye) \| + \| \grad_x f(\Pcal_{\Mcal}(\xe), \ye) - G(\xe,\ye)\| \\
        &\leq \epsilon + \frac{3}{\rho}\Big(\frac{7}{2} \rv{\tilde L_x}+ \frac{13}{4} L_{xx}\Big) \epsilon. 
    \end{align*}
    \rv{Finally, the desired bound for ${\rm dist}\Big(0,-\nabla_y f(\Pcal_{\Mcal}(\xe), \ye)+\partial h(\ye)\Big)$
    %the gradient mapping in $y$, 
     directly follows from 
     %the non-expansivity of ${\rm prox}_{\tau_2 h}(\cdot)$ for $\tau_2\in (0,\frac{1}{\zeta})$, and using 
     the Lipschitz continuity of $\nabla_{y} f$ together with \eqref{eq:distance-control} and \eqref{eq:bound-ax}.} \qed  
%\end{proof}

\section{Proof of \cref{thm}}
\label{sec:pl-proof}
\sa{In the rest $\cX$ and $\cY$ denote finite dimensional Euclidean vector spaces.} Let us start with some necessary notations. Let
%\nsa{Think about the domains of the following functions if we were to assume Assumption 2 holds only on $X$.}
\begin{itemize}
\item \rv{$\tilde f_r(x,y):=\tilde f(x,y)- h(y)$ with $\tilde{f}(x,y):= f(A(x),y) + \frac{\rho}{4}\|c(x)\|^2$ for $x\in X$ and $y\in\cY$;}
\item \sa{\rv{$\Phi^*:=\min_{x\in X}\Phi(x)$ with} $\Phi(x):=\max_{y\in {\cY}}\tilde f_r(x,y)$;}
\item $\hat{f}_r(x, y ; z):=\hat{f}(x, y ; z) - h(y)$ with $\hat{f}(x, y ; z):=\tilde{f}(x, y)+\frac{p}{2}\|x - z\|^2$ for any $(x,y)\in X\times \rv{Y}$ and $z\in\cX$. \rv{Consider the auxiliary %minimax 
problem $\min_{x\in X}\max_{y\in\cY}\hat f_r(x,y;z)$.} \sa{For any $y\in \rv{Y}$, due to \cref{lem:lip-f}, $\tilde f(\cdot,y)$ is $l$-weakly convex \rv{over $X$}; hence, $\hat f(\cdot,y;z)$ is $(p-l)$-strongly convex on $X$ for any fixed $z\in\cX$ and $y\in\rv{Y}$.}
 \item \rv{$\Phi(x ; z):=\max_{y \in \cY} \hat{f}_r(x, y ; z)$ for $x\in X$ and $z\in\cX$,} denotes the primal function of the auxiliary problem \rv{--note that $\Phi(x;z)=\Phi(x)+\frac{p}{2}\norm{x-z}^2$ for any $x\in X$ and $z\in \cX$;}
 \item \rv{$\Psi_r(y ; z):=\Psi(y ; z)-h(y)$} with $\Psi(y ; z):=\min_{x\in \sa{X}} \hat{f}(x, y ; z)$ for $y\in \rv{Y}$ and $z\in\cX$, %- \hjn{h(y)}$ 
  denotes the dual function of the auxiliary problem;
 \item $P(z):=\min_{x\in \sa{X}} \max_{y\in \hjn{\cY}} \rv{\hat{f}_r(x, y ; z)}$ is the optimal value for the auxiliary primal problem, i.e., $\rv{P(z)=}\min_{x\in \sa{X}}\Phi(x;z)$ for any fixed $z\in\cX$;
 \item $x^*(y, z):=\argmin_{x\in \sa{X}} \hat{f}_r(x, y ; z) =\argmin_{x\in \sa{X}} \hat{f}(x, y ; z)
 $ for any %fixed 
 $y\in \rv{Y}$ and $z\in\cX$;
 \item $x^*(z):=\argmin_{x \in \sa{X}} \Phi(x ; z)$ is the optimal solution \rv{to the auxiliary primal problem} for any fixed $z\in\cX$;
\item $\hjn{Y}^*(z):=\argmax_{y \in \hjn{\cY}} \rv{\Psi_r(y ; z)}$ is \rv{the set of optimal solutions} to the auxiliary dual problem for any fixed $z\in\cX$;
\item $\hjn{y^{+}(z_t):= {\rm prox}_{\tau_2 h}(\rv{y_t+\tau_2 {\nabla_y \tilde f\left(x^*(y_t, z_t), y_t\right)}})}$ for any $y_t\in Y$ and $z_t\in\cX$, and it denotes \rv{a prox-gradient update in $y$ 
%in the direction of the gradient of
corresponding to
the dual function $\Psi_r(\cdot;z)$ since $\Psi_r(\cdot;z)=\Psi(\cdot;z)-h(\cdot)$ and $\nabla\Psi(y;z)=\nabla_y \tilde f(x^*(y,z),y)$;}
\item \rv{$V(x,y;z):=\rv{\hat{f}_r}\left(x, y ; z\right) -2 \rv{\Psi_r}\left(y ; z\right)+2 P\left(z\right)$ for $x\in X$, $y\in Y$ and $z\in\cX$ denotes the potential function.}
%\item $x^{+}(y, z)=x-\tau_1 \grad_x \hat{f}(x, y ; z): x$ after one step of gradient descent with gradient at current point;\nsa{I am not sure if we need this.}
%\nhj{Yes. We may move this and proof to appendix.}
%\nsa{I meant that $x^+(y,z)$ may not have been used within the proof.}
%\item $\hat{G}_x(x, y ; z)=G_x(x, y)+p(x -z)$ : the stochastic gradient for regularized auxiliary function.
\end{itemize}%

Based on Assumptions \ref{lem:lip-f} and \ref{assum:sc}, we first list some helpful results from \citep{zhang2020single,yang2022faster}.
%\nsa{Add here minmax switch lemma, i.e., extension of Lemma 2.1 in \citep{yang2020global}}
%\nsa{Lemmas should be stated more rigorously in terms of assumptions. See my edits below.}
\begin{lemma}
\label{lem:PLPL-switch}
\rv{Let $\varphi:\reals^n\times\reals^m\to \reals\cup\{\pm\infty\}$ with domain $\dom \varphi:=\{(x,y)\in\reals^n\times\reals^m:\ -\infty<\varphi(x,y)<\infty\}$. Let $\mathcal{D}_x=\{x:\ \exists~y\ {\rm s.t.}\ (x,y)\in\dom \varphi\}$ and $\mathcal{D}_y=\{y:\ \exists~x\ {\rm s.t.}\ (x,y)\in\dom \varphi\}$. Suppose that 
\begin{enumerate}
    \item[(i)] $\varphi(\cdot,y)$ is a proper, closed $\zeta_x$-weakly convex function that is $\mu_x$-PL uniformly for all $y\in \mathcal{D}_y$;
    \item[(ii)] $-\varphi(x,\cdot)$ is a proper, closed $\zeta_y$-weakly convex function that is $\mu_y$-PL uniformly for all $x\in \mathcal{D}_x$.
\end{enumerate}}%

\rv{If there exists $(x^*,y^*)\in\dom \varphi$ such that $0\in\partial_x \varphi(x,y^*)\mid_{x=x^*}$ and $0\in\partial_y -\varphi(x^*,y)\mid_{y=y^*}$, then $(x^*,y^*)$ is a saddle point, i.e., $\varphi(x^*,y)\leq \varphi(x^*,y^*)\leq \varphi(x,y^*)$ for all $(x,y)\in \reals^n\times\reals^m$; hence, the minimax switch holds: $\min_{x}\max_y \varphi(x,y) = \max_y \min_x \varphi(x,y)$.}
\end{lemma}
\begin{proof}
\rv{The proof proceeds along the lines of \citep[Lemma~2.1]{yang2020global}. In particular, suppose $(x^*,y^*)$ is a stationary point, i.e., $0\in\partial_x \varphi(x,y^*)\mid_{x=x^*}$ and $0\in\partial_y -\varphi(x^*,y)\mid_{y=y^*}$. From \cref{assum:sc}, it follows that
\begin{align*}
    \varphi(x^*,y^*)-\min_x \varphi(x,y^*)\leq \frac{1}{2\mu_x}{\rm dist}^2(0, \partial_x p(x^*,y^*))=0,\\
    \max_y \varphi(x^*,y) - \varphi(x^*,y^*) \leq \frac{1}{2\mu_y}{\rm dist}^2(0, \partial_y -\varphi(x^*, y^*))=0;
\end{align*}
hence, we can conclude that $\max_y \varphi(x^*,y)=\varphi(x^*,y^*)\leq \min_x \varphi(x,y^*)$.}
%we replace every occurrence of the gradient with the Fr\'echet subdifferential $\partial$, interpret gradient norms as $\dist(0,\partial\,\cdot)$, and then invoke the subgradient PL inequality stated in \cref{assum:sc}.
\end{proof}
%\nsa{Add here Danskin-type theorem I proved.}
\begin{lemma}
    \label{lem:Danskin}
    Under Assumptions~\ref{assum:Y},~\ref{assum:lip-f} and~\ref{assum:sc}, let $\Phi:X\to \reals$ such that $\Phi(x):=\max_{y\in {\cY}}\tilde f_r(x,y)$ for $x\in X$. Then $\Phi(\cdot)$ is differentiable on $X$; indeed, for all $x\in X$,
    %The following Danskin's theorem holds:
    \[ \nabla \Phi(x) = \nabla_x \tilde f_r(x, r^*(x)), \]
    for any $r^*(x) \in \argmax_{y \in \cY} \; \tilde f_r(x, y)$, and $\nabla \Phi(\cdot)$ is $l(1+2\kappa)$-Lipschitz on $X$.

    \rv{Moreover, for any fixed $z\in \cX$, let $\Psi(y;z)=\min_{x\in X}\hat f(x,y;z)$ for $y\in Y$. Then, $\Psi(\cdot;z)$ is differentiable\footnote{\rv{The classic Danskin's theorem is stated for the case $\hat f(x,\cdot;z)$ is concave. In our case, as we only assume $\hat f(x,y;z)$ is smooth in $y$, for completeness we state the result for our setting.}} on $Y$; indeed, for all $y\in Y$,  $\nabla\Psi(y;z)=\nabla_y\tilde f(x^*(y,z),y)$ where $x^*(y,z)=\argmin_{x\in X} \hat f(x,y;z)$.}
\end{lemma}
\begin{proof}
%From the definition of 
Since we assume that for any  $x \in X$, $-f_r(x, \cdot)$ is $\mu$-PL, it holds that
\[ 2 \mu \Big(\Phi(x)-f_r(x,y)\Big) \leq {\rm dist}^2\Big(0, -\partial_y f_r(x,y)\Big),  \]
using the notational convention that $-\partial_y f_r(x,y):=\partial_y (-f_r(x,y))$.
By \citep[Theorem 3.1]{liao2024error}, any $\mu$-PL function also satisfies quadratic growth condition, namely,
\[ \frac{\mu}{4}{\rm dist}^2(y, R^*(x)) \leq \Phi(x)-f_r(x,y), \]
where $R^*(x):= \argmax_{y \in \cY}\; f_r(x,y)$ is a closed set. For given $x_1,x_2\in X$, it holds for any $y_1 \in R^*(x_1)$ that $0 \in -\nabla_y f(x_1, y_1) + \partial h(y_1)$; hence,
\be \label{eq:subgrad-danskin}\nabla_y f(x_1, y_1) -  \nabla_y f(x_2,y_1) \in - \nabla_y f(x_2, y_1) + \partial h(y_1). \ee
Consider $\Delta := \Phi(x_2) - f_r(x_2,y_1)\geq 0$. By the $\mu$-PL property of $-f_r(x_2, \cdot)$, it holds that
\[ \Delta \leq \frac{1}{2\mu} {\rm dist}^2 \Big(0, - \nabla_y f(x_2, y_1) + \partial h(y_1)\Big) \leq \frac{l^2}{2\mu}\|x_2-x_1\|^2, \]
where the second inequality uses \cref{eq:subgrad-danskin} and the Lipschitz continuity of $\nabla_y f(\cdot,y_1)$. Furthermore, let $y_2 = \argmin_{y \in R^*(x_2)} \; \|y- y_1\|$. Since 
${\rm dist}(y_1, R^*(x_2))=\norm{y_1-y_2}$, the quadratic growth property implies that $\Delta \geq \frac{\mu}{4} \|y_1 - y_2\|^2$; hence, combining both inequalities on $\Delta$ leads to
\be \label{eq:lip-y} \|y_1 - y_2\|^2 \leq 2\frac{l^2}{\mu^2} \|x_1 - x_2\|^2. \ee
Now given $x\in X$ and $d \in \cX$, let $y^* \in R^*(x)$. From \cref{eq:lip-y}, there exists $y^*(\tau) \in \argmax_{y \in \cY} \; f_r(x+ \tau d, y)$ such that 
$\| y^*(\tau) - y^*\|^2 \leq 2\tau^2\frac{l^2}{\mu^2}\|d\|^2$. Moreover, since $h$ is $\zeta$-weakly convex and $\nabla_y f(x, y^*) \in \partial h(y^*)$, we have
\be \label{eq:weak-cvx-danskin} h(y^*) - h(y^*(\tau)) + \iprod{\nabla_y f(x,y^*) }{y^*(\tau) - y^*} \leq \frac{\zeta}{2} \|y^*(\tau) - y^*\|^2.\ee
Then, we can bound the primal value difference, $\Phi(x+ \tau d) - \Phi(x)$, from above as follows:
\[ \begin{aligned}
    \Phi(x+ \tau d) - \Phi(x) = & f_r(x+ \tau d, y^*(\tau)) - f_r(x, y^* ) \\
    = & f(x+ \tau d, y^*(\tau)) - f(x, y^* ) + h(y^*) - h(y^*(\tau)) \\
    = & \tau \nabla_x f(x,y^*)^\top d + \nabla_y f(x,y^*)^\top (y^*(\tau) - y^*) + h(y^*) - h(y^*(\tau))+ o(\tau) \\
    \leq & \tau \nabla_x f(x,y^*)^\top d + \frac{\zeta}{2}\|y^*(\tau) - y^*\|^2 + o(\tau).
\end{aligned} \]
This implies
\[ \limsup_{\tau \downarrow 0} \frac{\Phi(x+ \tau d) -\Phi(x)}{\tau} \leq \nabla_x f(x,y^*)^\top d + \lim_{\tau \downarrow 0} \Big(\frac{o(\tau)}{\tau} + \zeta \frac{l^2}{\mu^2} \|d\|^2 \tau\Big) = \nabla_x f(x,y^*)^\top d. \]
On the other hand, we also have
\[ \Phi(x+ \tau d) \geq f_r(x + \tau d, y^*) = f_r(x,y^*) + \tau \nabla_x f(x,y^*)^\top d + o(\tau), \]
which implies that
\[  \liminf_{\tau \downarrow 0} \frac{\Phi(x+ \tau d) - \Phi(x)}{\tau} \geq \nabla_x f(x,y^*)^\top d. \]
Thus, 
\be \label{eq:limit-danskin} \lim_{\tau \downarrow 0} \frac{\Phi(x+ \tau d) - \Phi(x)}{\tau} = \nabla_x f(x,y^*)^\top d. \ee
Since \eqref{eq:limit-danskin} holds for all $d\in\cX$, we have $\nabla \Phi(x) = \nabla_x f(x,y^*)$ for any $y^* \in R^*(x)$. 

Next, to argue for the smoothness of $\Phi(\cdot)$, let $x_1, x_2\in X$ be some arbitrary points, consider any $y_1 \in R^*(x_1)$ and $y_2 \in \argmin_{y \in R^*(x_2)} \|y - y_1\|$. Then, we complete the proof by using \eqref{eq:lip-y} to reach at the conclusion:
\[ \begin{aligned}
    \| \nabla \Phi(x_2) - \nabla \Phi(x_1) \| = \|\nabla_x f(x_2,y_2) - \nabla_x f(x_1,y_1)\| \leq l(\|x_2 - x_1\| + \|y_2 -y_1\|) \leq l \left( 1+ 2\frac{l}{\mu} \right) \|x_2 - x_1\|. 
\end{aligned}\]
Next, for any fixed $z\in \cX$, consider $\Psi(y;z)=\min_{x\in X}\hat f(x,y;z)$ for $y\in Y$. Note that $\hat f(\cdot,y;z)$ is $(p-l)$-strongly convex for any $y\in Y$; hence, $\hat f(\cdot,y;z)$ is $\mu$-PL and let $x^*(y,z)=\argmin_{x\in X}\hat f(x,y;z)$ denote the unique optimal solution. Observing that $\Psi(y;z)=-\max_{x\in X} -\hat f(x,y;z)$, using the result from the first part, we can conclude that $\nabla \Psi(y;z)=\nabla_y \hat f(x^*(y,z),y;z)$ for all $y\in Y$, which gives us the desired result.  
%\nsa{Be careful about this extra 2 here in the constant.}
%We complete the proof.
\end{proof}
% We will list examine the Lipschitz continuity of $\Psi(\cdot,z)$, $x^*(y,z)$, and $x^*(z)$.
\begin{lemma}
\label{lem:minmax}
    \rv{For any given $z\in\cX$, consider $\hat f_r(\cdot,\cdot;z)$ defined above. Under Assumptions~\ref{assum:Y},~\ref{assum:lip-f} and~\ref{assum:sc}, it holds that $\min_{x\in X}\max_{y\in\cY}\hat f_r(x,y;z)=\max_{y\in\cY}\min_{x\in X}\hat f_r(x,y;z)$; hence, $P(z)=\min_{x\in X}\Phi(x;z)=\max_{y\in\cY}\Psi_r(y;z)$.}
\end{lemma}
\begin{proof}
    \rv{Fix an arbitrary $z\in \cX$, and define $\varphi(x,y):= \hat f_r(x,y;z)+\delta_X(x)$, where $\delta_X(\cdot)$ denotes the indicator function of the set $X$. For any $y\in Y$, $\varphi(\cdot,y)$ is strongly convex with modulus $p-l$, and according to \citep[Theorem 3.1]{liao2024error}, the non-smooth function $\varphi(\cdot,y)$ is $\mu_x$-PL with $\mu_x=p-l$. On the other hand, according to \cref{assum:sc}, $\varphi(x,\cdot)$ is $\mu_y$-PL with $\mu_y=\mu$.} 
    
    \rv{Since $\hat f_r(\cdot,y;z)$ is strongly convex with modulus $p-l$ for any $y\in Y$, %we also have 
    $\Phi(\cdot;z)$ is also strongly convex with modulus $p-l$ as it is the pointwise maximum of strongly convex functions. Therefore, $x^*(z)=\argmin_{x\in X}\Phi(x;z)$ exists. Let $y^*(z)\in\argmax_{y\in\cY}\hat f_r(x^*(z),y;z)$ --note that $Y=\dom h$ is compact and $x^*(z)\in X$ imply that $\hat f_r(x^*(z),\cdot;z)$ is continuous on $Y$; therefore, $y^*(z)$ exists. Thus, first-order optimality conditions and Lemma~\ref{lem:Danskin} imply that
    \begin{align*}
        &0\in\nabla \Phi(x^*(z);z)+\partial \delta_X(x^*(z))=\nabla_x \hat f_r(x^*(z),y^*(z);z)+\partial \delta_X(x^*(z))=\partial_x \varphi(x^*(z),y^*(z)),\\
        &0\in -\partial_y \hat f_r(x^*(z),y^*(z);z) = -\partial_y \varphi(x^*(z),y^*(z)); 
    \end{align*}
    hence, $(x^*(z),y^*(z))$ is a stationary point of $\min_{x\in X}\max_{y\in\cY}\varphi(x,y)$. Therefore, using \cref{lem:PLPL-switch} we can conclude that $(x^*(z),y^*(z))$ is a saddle point, which implies that $y^*(z)\in Y^*(z)$. This completes the proof.}%
\end{proof}
\begin{lemma} \label{lem:lip}
    Suppose that Assumptions \ref{lem:lip-f} and \ref{assum:sc} hold, and $p$ is chosen such that $p>l$. Then, we have
\begin{subequations}
\begin{align}
& \left\|x^*(y, z)-x^*\left(y, z^{\prime}\right)\right\| \leq \gamma_1\left\|z-z^{\prime}\right\|,\quad \sa{\forall~y\in \rv{Y},\ \forall~z,z'\in \rv{X},} \label{eq:lip-z-yz}\\
& \left\|x^*(z)-x^*\left(z^{\prime}\right) \right\| \leq \gamma_1 \|z-z^{\prime} \|, \quad \sa{\forall~z,z'\in \rv{X},} \label{eq:lip-z}\\
& \left\|x^*(y, z)-x^*\left(y^{\prime}, z\right)\right\| \leq \gamma_2\left\|y-y^{\prime}\right\|,\quad \sa{\forall~z\in \rv{X},\ \forall~y,y'\in \rv{Y},} \label{eq:lip-y-yz}\\
& \sa{\left\|x_{t+1}-x^*\left(y_t, z_t\right)\right\| \leq \gamma_3   \left \| x_{t+1} - x_t\right\|,} \label{eq:delta_x-bound}
\end{align}
\end{subequations}
where $\gamma_1:=\frac{p}{p-l}, \gamma_2:=\frac{p+l}{p-l}$, and \sa{$\gamma_3:=1+\frac{1}{\tau_1(p-l)}$}.
%\nsa{The proof for \eqref{eq:delta_x-bound} can be found in Appendix C of \url{https://arxiv.org/pdf/1812.10229}. See writeup ``Strong Convexity ineq2" and compare it with what appendix C claims: $\frac{1}{\mu t}\norm{x^+-\bar x}\geq \norm{\bar x - x^*}$.} 
\end{lemma}
\begin{proof}
    \rv{%This result is shown in~\citep[Lemma B.2]{zhang2020single}. 
    For the proof of \eqref{eq:lip-z-yz} and \eqref{eq:lip-z}, see \citep[Lemma~21]{laguel2024high} and \citep[Lemma~23]{laguel2024high}, respectively -- these proofs simply extends to the constrained optimization setting here with $x\in X$. The inequality in \eqref{eq:lip-y-yz} follows from the proof of \citep[Lemma B.2(c)]{lin2020near} which can also handle $x\in X$ constraint. The inequality in~\eqref{eq:delta_x-bound} is provided in \citep[Lemma B.2]{zhang2020single} of which proof can be found in Appendix C at \url{https://arxiv.org/pdf/1812.10229}.}
\end{proof}

We first establish the following lemma. 
{
\begin{lemma} \label{lem:bound-new}
Under \rv{Assumptions \ref{assum:Y},} \ref{assum:lip-f} and \ref{assum:sc}, \rv{for all $t\in\mathbb{Z}_+$, it holds that}
\be \label{eq:lip-station-y-new}
\begin{aligned}
& \left\|x^*\left(y^+(z_t), z_t\right)-x^*\left(z_t\right)\right\|^2 
\leq  \frac{3}{\mu(p-l) \tau_2^2} (1 + \tau_2^2 l^2 + \gamma_2^2 \tau_2^2 l^2 )\left\| y_t^+(z_t) - y_t \right\|^2 .
\end{aligned}
\ee
\end{lemma}
\begin{proof}
\rv{Since %$\hat f(\cdot,y;z)-h(y)$ 
$\hat f_r(\cdot,y;z)$ is $(p-l)$ strongly convex for any fixed $y\in Y$ and $z\in\cX$, and $\Phi(\cdot,z)$ is the pointwise supremum over $y\in \cY$, we can conclude that $\Phi(\cdot, z)$ is also $(p-l)$-strongly convex for any fixed $z\in\cX$. Thus, for any $g_t \in \partial h(y^+(z_t))$, we get}
\be \label{eq:bound-1-new} \begin{aligned}
    \MoveEqLeft \frac{p-l}{2}\|x^*(z_t) - x^*(y^+(z_t), z_t)\|^2 \leq  \Phi(x^*(y^+(z_t), z_t); z_t) - \Phi(x^*(z_t); z_t)  \\
    \rv{=} & \Phi(x^*(y^+(z_t), z_t); z_t) 
    - \hat{f}_r(x^*(y^+(z_t), z_t), y^+(z_t); z_t)
    %- \hat{f}(x^*(y^+(z_t), z_t), y^+(z_t); z_t) + h(y^+(z_t)) 
    \\
    & + \hat{f}_r(x^*(y^+(z_t), z_t), y^+(z_t); z_t)
    %+ \hat{f}(x^*(y^+(z_t), z_t), y^+(z_t); z_t) - h(y^+(z_t))
    - \Phi(x^*(z_t); z_t) \\
    \leq & \frac{1}{2\mu}\| \rv{\nabla_y \tilde{f}(x^*(y^+(z_t), z_t), y^+(z_t))} - g_t \|^2,
\end{aligned} \ee
where \rv{in the last inequality we used the fact that of 
%$h(\cdot)-\hat{f}(x, \cdot; z)$ 
$-\hat{f}_r(x, \cdot; z)$ is $\mu$-PL for any $x\in X$, $z\in\cX$ and that $
%\hat{f}(x^*(y^+(z_t), z_t), y^+(z_t), z_t) - h(y^+(z_t)) 
\hat{f}_r(x^*(y^+(z_t), z_t), y^+(z_t), z_t)=\Psi_r(y^+(z_t); z_t)\leq \Psi_r(y^*(z_t); z_t) = \Phi(x^*(z_t), z_t)$,} where $y^*(z_t) \in Y^*(z_t)$ and the equality follows from Lemma~\ref{lem:minmax}.
%\nsa{Verify/cite min-max switch for PL-PL case.} 
On the other hand, $y^+(z_t) = {\rm prox}_{\tau_2 h}(y_t + \tau_2 \rv{\nabla_y \tilde{f}(x^*(y_t, z_t), y_t)})$ is \rv{well defined for $\tau_2\in(0,\frac{1}{\zeta})$,
%\nsa{I used $\tau_2\in(0,\frac{1}{L_{yy}+\rho})$ before and Jiang changed the bound to $\frac{1}{\rho}$.} 
and we have $0\in\partial h(y^+(z_t))+(y^+(z_t)-y_t)/\tau_2-\nabla_y \tilde{f}(x^*(y_t, z_t), y_t)$ ; hence, for $g_t\in\partial h(y^+(z_t))$ such that $g_t=\nabla_y \tilde{f}(x^*(y_t, z_t), y_t)-(y^+(z_t)-y_t)/\tau_2$, we get}
%we have 
\be \label{eq:bound-2-new} 
\begin{aligned}
    & \| \nabla_y \tilde{f}(x^*(y^+(z_t), z_t), y^+(z_t)) - g_t \|^2 \\
    \leq & 3\| \nabla_y \tilde{f}(x^*(y_t, z_t), y_t) - g_t \|^2 + 3\| \nabla_y \tilde{f}(x^*(y_t, z_t), y_t) -  \nabla_y \tilde{f}(x^*(y_t, z_t), y^+(z_t))  \|^2 \\
    & + 3 \| \nabla_y \tilde{f}(x^*(y_t, z_t), y^+(z_t)) -  \nabla_y \tilde{f}(x^*(y^+(z_t), z_t), y^+(z_t))  \|^2 \\
    \leq & 3\Big(\frac{1}{\tau_2^2} + l^2 + \gamma_2^2 l^2\Big)\|y_t - y^+(z_t) \|^2 ,
\end{aligned}
\ee
where we use the %triangle inequality 
\rv{identity $\norm{a+b+c}^2\leq 3(\norm{a}_2+\norm{b}_2+\norm{c}_2)$} in the first inequality, and the second inequality is from the definition of ${\rm prox}_{\tau_2 h}$, Lipschitz continuity of \rv{$\nabla_y f$}, and Lemma \ref{lem:lip}. 
Combining \eqref{eq:bound-1-new} and \eqref{eq:bound-2-new} gives the desired inequality.
\end{proof}
}

\rv{Next we state a classic result on projected gradient updates, which will be useful later in our analysis. Due to space limitations, we omit its proof.}
\begin{lemma}
\label{lem:grad-map-1}
    Let $f$ be a convex function and $X$ be a convex set. Given $\bar x\in X$ and some arbitrary $t>0$, let $x^+=\cP_X(\bar x-t\nabla f(\bar x))$. Then $\fprod{g_t(\bar x)-\nabla f(\bar x),~x^+-x}\geq 0$ for all $x\in X$, where $g_t(\bar x)=(\bar x - x^+)/t$. In particular, $\fprod{g_t(\bar x)-\nabla f(\bar x),~x^+-\bar x}\geq 0$.
\end{lemma}

Now, we proceed with the proof of Theorem \ref{thm}, \rv{which adopts the same potential function framework} in~\citep{zhang2020single,yang2022faster}. Following the analysis in \citep{yang2022faster}, we separate our proof into several parts: we first present three descent results, then we show the descent property for the potential function, later we discuss the relation between our \rv{stationarity measure} and the potential function, and lastly we put all things together.

\begin{proof}[Proof of Theorem \ref{thm}] \rv{In the first part of the proof, we begin with showing primal descent by providing a lower bound on $\hat{f}\left(x_t, y_t ; z_t\right) - \hat{f}\left(x_{t+1}, y_{t+1}; z_{t+1}\right)$; next, we show dual ascent by bounding $\Psi\left(y_{t+1} ; z_{t+1}\right)-\Psi\left(y_t ; z_t\right)$ from below, and finally lower bound the change in the Moreau envelope $P(z_t)-P(z_{t+1})$.}

\textbf{Primal descent}: By the $(p+l)$-smoothness of $\hat{f}\left(\cdot, y_t ; z_t\right)$, \hj{we have}
\begin{equation}
\label{eq:p-descent-1}
    \begin{aligned}
\MoveEqLeft \hat{f}\left(x_{t+1}, y_t ; z_t\right) -\hat{f}\left(x_t, y_t; z_t\right)\\
\leq &  \left\langle\nabla_x \hat{f}\left(x_t, y_t ; z_t\right), x_{t+1} - x_t \right\rangle+\frac{p+l}{2}\left\| x_{t+1} - x_t \right\|^2 \\
= & \left\langle\nabla_x \hat{f}\left(x_t, y_t ; z_t\right), \pcal_X\Big(x_t - \tau_1 \nabla_x \hat{f}\left(x_t, y_t ; z_t\right)\Big) -x_t \right\rangle +\frac{p+l}{2} \left\|x_{t+1} - x_t\right\|^2.
\end{aligned}
\end{equation}
\sa{Note that it follows from \cref{lem:grad-map-1} that}  
\begin{equation}
    \label{eq:p-descent-2}
    \iprod{\nabla_x \hat{f}(x_t,y_t;z_t)}{\pcal_X(x_t - \tau_1 \nabla_x \hat{f}\left(x_t, y_t ; z_t\right)) -x_t} \leq \sa{-\frac{1}{\tau_1}\|x_{t+1} - x_t\|^2.} 
\end{equation}
\sa{Thus, for the choice of $\tau_1 \leq \frac{1}{p+l}$, we get}
\be \label{eq:de-x}
 \hat{f}\left(x_t, y_t ; z_t\right)- \hat{f}\left(x_{t+1}, y_t ; z_t\right) \geq \sa{\frac{1}{2\tau_1}} \left\|x_{t+1} - x_t\right\|^2.
\ee
Moreover, since $\hat{f}\left(x_{t+1}, \cdot ; z_t\right)$ is \sa{$l$-smooth}, %and \sa{$\tau_2 \leq \frac{1}{l}$},
\be \label{eq:de-y}
\begin{aligned}
\hat{f}\left(x_{t+1}, y_t ; z_t\right)-\hat{f}\left(x_{t+1}, y_{t+1} ; z_t\right) \geq & \rv{\left\langle\nabla_y \hat{f}\left(x_{t+1}, y_t ; z_t\right), y_t-y_{t+1}\right\rangle-\frac{l}{2}\left\|y_t-y_{t+1}\right\|^2.}
%= & \rv{\left\langle\nabla_y \hat{f}\left(x_{t+1}, y_t ; z_t\right),~ y_t-\Pcal_{Y}(y_t +\tau_2 \nabla_y \hat{f}(x_{t+1}, y_t; z_t) ) \right\rangle-\frac{l}{2} \left\|y_{t+1} - y_t \right\|^2} \\
%\sa{\geq} & \sa{-\left(\frac{1}{\tau_2}+\frac{l}{2}\right)}{\tau_2^2}\| {\nabla_y \tilde{f}(x_{t+1}, y_t) }\|^2,
\end{aligned}
\ee 
\sa{Furthermore, from the definition of $\hat{f}$ and $z_{t+1}-z_t=\beta(x_{t+1}-z_t)$, and using $0<\beta\leq 1$, we get}
\be \label{eq:de-z}
\begin{aligned}
\MoveEqLeft\hat{f}\left(x_{t+1}, y_{t+1} ; z_t\right)-\hat{f}\left(x_{t+1}, y_{t+1} ; z_{t+1}\right)  \\
&=\frac{p}{2}\left[\left\|x_{t+1}-z_t \right\|^2-\left\|x_{t+1}-z_{t+1}  \right\|^2\right] = \frac{p}{2}\iprod{z_t - z_{t+1}}{z_{t+1} + z_{t} - 2 x_{t+1}} \\
&= \frac{p}{2} \iprod{z_t - z_{t+1}}{z_{t+1} - z_t \rv{+} \frac{2}{\beta} (z_t - z_{t+1})} = \frac{p}{2} \Big(\frac{2}{\beta} - 1\Big) \|z_t - z_{t+1}\|^2  \\
&\geq \frac{p}{2\beta} \|z_t - z_{t+1}\|^2. 
% = &\frac{p}{2}\left[\frac{1}{\rho^2}\left\|\log_{z_t} z_{t+1}\right\|^2-\left\|(1-\rho)\log_{z_t} x_{t+1} \right\|^2\right] \\
% = & \frac{p}{2}\left[\frac{1}{\rho^2}\left\|\log_{z_t} z_{t+1}\right\|^2-\frac{(1-\rho)^2}{\rho^2}\left\|\log_{z_t} z_{t+1}\right\|^2\right] \geq \frac{p}{2 \rho}\left\|\log_{z_t} z_{t+1}\right\|^2 .
\end{aligned}
\ee 

Combining \eqref{eq:de-x}, \eqref{eq:de-y} and \eqref{eq:de-z}, we get
$$
% \begin{aligned}
% \hat{f}\left(x_t, y_t ; z_t\right)- \hat{f}\left(x_{t+1}, y_{t+1}; z_{t+1}\right) \geq & \sa{\frac{1}{2\tau_1}} \|x_{t+1} - x_t\|^2
% %-\sa{\Big(\frac{1}{\tau_2}+\frac{l}{2}\Big)} {\tau_2^2}\| \sa{\nabla_y}{\tilde{f}(x_{t+1}, y_t) }\|^2 \\
% +\frac{p}{2 \beta} \left\| z_{t+1} - z_t\right\|^2\\
% &+\rv{\left\langle\nabla_y \tilde{f}\left(x_{t+1}, y_t\right), y_t-y_{t+1}\right\rangle-\frac{l}{2}\left\|y_t-y_{t+1}\right\|^2.}
% \end{aligned}
% $$
\begin{aligned}
\MoveEqLeft \hat{f}\left(x_t, y_t ; z_t\right) - \hat{f}\left(x_{t+1}, y_{t+1}; z_{t+1}\right) %+\hjn{h(y_{t+1})- h(y_t)} 
\\
\geq & \sa{\frac{1}{2\tau_1}} \|x_{t+1} - x_t\|^2
%-\sa{\Big(\frac{1}{\tau_2}+\frac{l}{2}\Big)} {\tau_2^2}\| \sa{\nabla_y}{\tilde{f}(x_{t+1}, y_t) }\|^2 \\
+\frac{p}{2 \beta} \left\| z_{t+1} - z_t\right\|^2 + \hjn{\iprod{\rv{\nabla_y \tilde{f}(x_{t+1}, y_t)}}{y_t - y_{t+1}}} 
%- h(y_t)+ h(y_{t+1})
\hjn{-\frac{l}{2} \left\|y_t-y_{t+1}\right\|^2.}
\end{aligned}
$$

\textbf{Dual Descent}: %\hjn{Need to consider $h(y)$} 
\rv{Recall that the dual function of the auxiliary minimax problem, $\Psi_r(\cdot;z)$, has a composite form, i.e., $\Psi_r(\cdot;z)=\Psi(\cdot;z)-h(\cdot)$.} \rv{For any $z\in \cX$ and $y\in Y$, Lemma~\ref{lem:Danskin} implies that $\nabla \Psi(y;z)=\nabla_y\tilde f(x^*(y,z),y)$; therefore, for any $y_1,y_2\in Y$, we have $\norm{\nabla \Psi(y_1;z)-\nabla \Psi(y_2;z)}\leq \norm{\nabla_y\tilde f(x^*(y_1,z),y_1)-\nabla_y\tilde f(x^*(y_2,z),y_2)}\leq l(\norm{x^*(y_1,z)-x^*(y_2,z)}+\norm{y_1-y_2})\leq l(1+\gamma_2)\norm{y_1-y_2}$ which follows from \eqref{eq:lip-y-yz}.} 
%\rv{\bf The result below is correct but we cannot cite \citep{zhang2020single} because it assumes concavity; hence, it calls for the classic Danskin's theorem. In our case, we need to argue through directional derivatives as we only assume $\hat f(x,y;z)$ is smooth in $y$. Add this result to Lemma 6, i.e., we show a result for both $\nabla\Phi$ and $\nabla \Psi$ there.}
%\sa{It follows from \citep[Lemma B.3]{zhang2020single} that} 
%the dual function \sa{$\Psi(\cdot ; z)$ \hjn{$+h(\cdot)$}} 

Thus, $\Psi(\cdot ; z)$ is $L_{\Psi}$ smooth \sa{for all $z\in\cX$} with $L_{\Psi}\sa{:=}l(1+ \gamma_2)$, \sa{where $\gamma_2:=\frac{p+l}{p-l}$. Using the fact that $\nabla_y\Psi(y_t;z_t )
%=\nabla_y\hat f(x^*(y_t,z_t),y_t;z_t)
\rv{=\nabla_y\tilde f(x^*(y_t,z_t),y_t)}$, we get}
\be \label{eq:dual-de-y}
\begin{aligned}
& \Psi\left(y_{t+1} ; z_t\right)-\Psi\left(y_t ; z_t\right)  %\hjn{+h(y_{t+1})-h(y_t)} 
\geq \left\langle \rv{\nabla_y\tilde f(x^*(y_t,z_t),y_t)}, y_{t+1}-y_t\right\rangle-\frac{L_{\Psi}}{2}\left\|y_{t+1}-y_t\right\|^2. 
% \\
% = & \rv{\left\langle\nabla_y \hat{f}\left(x^*\left(y_t, z_t\right), y_t ; z_t\right), y_{t+1}-y_t\right\rangle-\frac{L_{\Psi}}{2}\left\|y_{t+1}-y_t\right\|^2.} 
% \\
% = & \tau_2 \left\langle\nabla_y \tilde{f}\left(x^*\left(y_t, z_t\right), y_t \right), 
% % \nabla_y \tilde{f}(x_{t+1}, y_t) 
% \revise{G_t^y}
% \right\rangle  -\frac{L_{\Psi} \tau_2^2 }{2}  
% % \left\| \nabla_y \tilde{f} \left(x_{t+1}, y_t\right)  \right\|^2.
% \revise{\|G_t^y\|^2}.
\end{aligned}
\ee
\sa{Moreover, from the definition of $\Psi$, we also have}
\begin{equation}
\label{eq:dual-descent}
\begin{aligned}
\MoveEqLeft\Psi\left(y_{t+1} ; z_{t+1}\right) -\Psi\left(y_{t+1} ; z_t\right) \\
= & \hat{f}\left(x^*\left(y_{t+1}, z_{t+1}\right), y_{t+1} ; z_{t+1}\right)-\hat{f}\left(x^*\left(y_{t+1}, z_t\right), y_{t+1} ; z_t\right) \\
\geq & \hat{f}\left(x^*\left(y_{t+1}, z_{t+1}\right), y_{t+1} ; z_{t+1}\right)-\hat{f}\left(x^*\left(y_{t+1}, z_{t+1}\right), y_{t+1} ; z_t\right) \\ = & \frac{p}{2}\left[\left\| x^*\left(y_{t+1}, z_{t+1}\right) - z_{t+1}\right\|^2-\left\|  x^*\left(y_{t+1}, z_{t+1}\right) - z_t \right\|^2\right] \\
= & {\frac{p}{2}\iprod{z_{t+1}-z_t} {~z_{t+1}+z_t-2 x^*\left(y_{t+1}, z_{t+1}\right)}},
\end{aligned}
\end{equation}
where the %second 
inequality  uses the optimality of $x^*(y_{t+1}, z_t)$, i.e., $\hat{f}(x^*(y_{t+1}, z_t), y_{t+1}; z_t) \leq \hat{f}(x^*(y_{t+1}, z_{t+1}), y_{t+1}; z_t)$.  
Combining %\sa{the last inequality} 
\eqref{eq:dual-descent} with \eqref{eq:dual-de-y}, we have
$$
\begin{aligned}
\MoveEqLeft\Psi\left(y_{t+1} ; z_{t+1}\right)-\Psi\left(y_t ; z_t\right) %\hjn{+h(y_{t+1})-h(y_t)} 
\\
\geq &  \left\langle
\rv{\nabla_y \tilde{f}\left(x^*\left(y_t, z_t\right), y_t\right)},
%\nabla_y \Psi\left(y_t ; z_t\right), 
y_{t+1}-y_t
\right\rangle-\frac{L_{\Psi}}{2} \left\|
y_{t+1}-y_t
\right\|^2 \\
&+\frac{p}{2} \iprod{z_{t+1}-z_t}{~z_{t+1}+z_t-2 x^*\left(y_{t+1}, z_{t+1}\right)}.
\end{aligned}
$$
\textbf{Proximal Descent:} Let \hjn{$y^*\left(z_{t+1}\right) \in Y^*(z_{t+1})$ and $y^*\left(z_t\right) \in Y^*(z_t)$}. Then, we have
$$
\begin{aligned}
\MoveEqLeft P\left(z_t\right)-P\left(z_{t+1}\right) \\ 
= & \min_{x\in X}\max_{y\in\cY} \rv{\hat f_r(x,y;z_{t})}-\min_{x\in X}\max_{y\in\cY} \rv{\hat f_r(x,y;z_{t+1})}\\
= & \rv{\max_{y\in\cY}\min_{x\in X} \hat f_r(x,y;z_{t})-\max_{y\in\cY}\min_{x\in X} \hat f_r(x,y;z_{t+1})}\\
= & \rv{\Psi_r}\left(y^*\left(z_t\right) ; z_t\right) -\rv{\Psi_r}\left(y^*\left(z_{t+1}\right); z_{t+1}\right)\\
\geq &\rv{\Psi_r}\left(y^*\left(z_{t+1}\right) ; z_t\right)-\rv{\Psi_r}\left(y^*\left(z_{t+1}\right) ; z_{t+1}\right) \\ 
= & \hat{f}\left(x^*\left(y^*\left(z_{t+1}\right), z_t\right), y^*\left(z_{t+1}\right) ; z_t\right)-\hat{f}\left(x^*\left(y^*\left(z_{t+1}\right), z_{t+1}\right), y^*\left(z_{t+1}\right) ; z_{t+1}\right) \\
\geq & \hat{f}\left(x^*\left(y^*\left(z_{t+1}\right), z_t\right), y^*(z_{t+1}) ; z_t\right)-\hat{f}\left(x^*\left(y^*\left(z_{t+1}\right), z_t\right), y^*\left(z_{t+1}\right) ; z_{t+1}\right) \\
= & \frac{p}{2} \Big[\|  x^*\left(y^*\left(z_{t+1}\right); z_t\right) - z_t \|^{\sa{2}}-\| x^*\left(y^*\left(z_{t+1}\right); z_t\right) -z_{t+1}  \|^{\sa{2}}\Big] \\
= & -\frac{p}{2}\iprod{z_{t+1}-z_t}{~z_{t+1}+z_t-2 x^*\left(y^*\left(z_{t+1}\right), z_t\right)}, 
\end{aligned}
$$
where \rv{the second equality follows from Lemma~\ref{lem:PLPL-switch} since $\hat f_r(x,y;z)$ is strongly convex in $x$ and PL in $y$ for every $z\in\cX$,} the first inequality uses the optimality of $y^*(z_t)$ for \rv{$\max_y\Psi_r(\cdot, z_t)$} and the second inequality is from the optimality of $x^*(y^*(z_{t+1}), z_{t+1})$ \rv{for $\min_{x\in X} \hat f(x,y^*(z_{t+1}, z_{t+1})$.} 

\textbf{Potential Function:} We adopt the potential function as in \citep{zhang2020single,yang2022faster}, 
\begin{equation}
\label{eq:Vt}
    V_t:=V\left(x_t, y_t, z_t\right)=\rv{\hat{f}_r}\left(x_t, y_t ; z_t\right) -2 \rv{\Psi_r}\left(y_t ; z_t\right)+2 P\left(z_t\right),\qquad \forall~t\geq 0.
\end{equation} 
\rv{%Under Assumption~\ref{assum:coercive}, 
Note that we have $P(z_t)\geq \rv{\Phi^*}=\min_{x\in X}\Phi(x)$, and from weak duality, we also have $\rv{\hat{f}_r}\left(x_t, y_t ; z_t\right)\geq P(z_t)\geq  \rv{\Psi_r}\left(y_t ; z_t\right)$; thus, we can conclude that $V_t\geq P(z_t)\geq \rv{\Phi^*}$ for all $t\in\mathbb{Z}_+$.}

\rv{Recall that $y_{t+1}=\argmin_{y\in\cY} q_t(y)$ where $q_t(y):=\tau_2 h(y)+\frac{1}{2}\norm{y-\Big(y_t+\tau_2\nabla_y\tilde f(x_{t+1},y_t;z_t)\Big)}^2$. Since $h(\cdot)$ is $\zeta$-weakly convex, it follows that $q_t(\cdot)$ is strongly convex with modulus $(1- \tau_2 \zeta)$. Therefore, we get} 
\[\begin{aligned}
   \MoveEqLeft \tau_2 h(y_{t}) + \frac{1}{2}\|\tau_2 \nabla_y \hat{f}\left(x_{t+1}, y_t, z_t\right) \|^2 \\
   \geq & \tau_2 h(y_{t+1}) + \frac{1}{2}\|y_{t+1} -( y_t + \tau_2 \nabla_y \hat{f}\left(x_{t+1}, y_t, z_t\right)) \|^2 + \frac{1-\tau_2 \zeta}{2} \|y_{t+1} - y_t\|^2, 
\end{aligned}
  \]
which implies that
\begin{equation}
\label{eq:h-relation}
    \left\langle\nabla_y \tilde{f}\left(x_{t+1}, y_t\right), y_{t+1}-y_{t}\right\rangle + h(y_t) - h(y_{t+1})\geq  \frac{2 - \tau_2\zeta}{2 \tau_2} \|y_{t+1} - y_t\|^2. 
\end{equation}
%where we use $\tilde{f}\left(x_{t+1}, y_t\right) = \hat{f}\left(x_{t+1}, y_t;z_t\right)$. 
Since $V_t=\rv{\hat{f}\left(x_t, y_t ; z_t\right) -2\Psi\left(y_t ; z_t\right)+2 P\left(z_t\right)+h(y_t)}$, \rv{combining \eqref{eq:h-relation} with} the above %three
descent \sa{results}, %analyses, 
we get
\begingroup
\allowdisplaybreaks
\be \label{eq:de-V}
\begin{aligned}
\MoveEqLeft V_t- V_{t+1} \\
 \geq & \rv{\frac{1}{2\tau_1} \left\|x_{t+1}-x_t\right\|^2 -  
 % \frac{l\tau_2^2}{2}
 % \left\|G_t^y \right\|^2
 \hjn{\frac{l}{2} \|y_{t+1} - y_t\|^2}
 }+ \frac{p}{2  \beta} \left\| z_t - z_{t+1}\right\|^2 \\
& \rv{+\fprod{\nabla_y\tilde f(x_{t+1},y_t),~y_{t+1}-y_t}} \hjn{+ h(y_t) - h(y_{t+1})}\\
& \rv{+2\left\langle\nabla_y \tilde{f}\left(x^*\left(y_t, z_t\right), y_t \right)-\nabla_y\tilde f(x_{t+1},y_t),~y_{t+1}-y_t \right\rangle}
-L_{\Psi}
% \tau_2^2  
% \rv{\left\|G_t^y \right\|^2} 
\hjn{\|y_{t+1} - y_t\|^2}
\\
& + p \iprod{z_{t+1}-z_t}{z_{t+1}+z_t-2 x^*\left(y_{t+1}, z_{t+1}\right)} \\
& -p \iprod{z_{t+1}-z_t} {z_{t+1}+z_t-2 x^*\left(y^*\left(z_{t+1}\right), z_t\right)} \\
\rv{\geq} %\hjn{=} 
& \rv{\frac{1}{2\tau_1} \left\|x_{t+1}-x_t\right\|^2 +
% \left( 1 - \frac{l\tau_2}{2}\sa{-}L_{\Psi} \tau_2\right)\tau_2 \left\| G_t^y\right\|^2
\hjn{\Big(\frac{1}{\tau_2} - \frac{l + \zeta}{2} - L_{\Psi}\Big) \|y_{t+1} - y_t\|^2}
+\frac{p}{2 \beta} \left\|z_t-z_{t+1}\right\|^2}\\
& \rv{+2\left\langle\nabla_y \tilde{f}\left(x^*\left(y_t, z_t\right), y_t\right)-\nabla_y \tilde{f}\left(x_{t+1}, y_t\right),~y_{t+1}-y_t \right\rangle} \\
& +2p \iprod{z_{t+1}-z_t}{x^*\left(y^*\left(z_{t+1}\right), z_t\right)-x^*\left(y_{t+1}, z_{t+1}\right)}\\
\geq & \rv{\frac{1}{2\tau_1} \left\|x_{t+1}-x_t\right\|^2+}
% \frac{\tau_2}{2} \left\|G_t^y \right\|^2
\hjn{\frac{1}{2\tau_2}\|y_{t+1} - y_t\|^2}
+\frac{p}{2 \beta} \left\|z_t-z_{t+1}\right\|^2 \\
& \rv{+2 \left\langle\nabla_y \tilde{f}\left(x^*\left(y_t, z_t\right), y_t\right)-\nabla_y \tilde{f}\left(x_{t+1}, y_t\right),~y_{t+1}-y_t \right\rangle} \\
& +2 p \iprod{z_{t+1}-z_t}{x^*\left(y^*\left(z_{t+1}\right), z_t\right)-x^*\left(y_{t+1}, z_{t+1}\right)},
\end{aligned}
\ee
\endgroup
where 
%\rv{the first inequality used the definition of $G_t^x$,}
% and $G_t^y$,
% the second inequality follows from the property of projection, i.e., $\fprod{w_p-\bar w,~w-w_p}\geq 0$ for all $w\in W$ where $W$ is a closed convex set and $w_p=\cP_W(\bar w)$ for some $\bar w$ --in particular, since $y_t\in Y$, this inequality implies that $\fprod{y_{t+1}-y_t-\tau_2\nabla_y \tilde f(x_{t+1},y_t),~y_t-y_{t+1}}\geq 0$, i.e., $\fprod{\nabla_y \tilde f(x_{t+1},y_t),~y_{t+1}-y_t}\geq\frac{1}{\tau_2}\norm{y_{t+1}-y_t}^2$; finally,}
%\hjn{and} 
the last inequality uses \rv{$\frac{1}{\tau_2}\geq l+\zeta+2L_\Psi=9l+\zeta$}
%$1 -\frac{(l + \zeta)\tau_2}{2} -L_{\Psi} \tau_2 \geq \frac{1}{2}$ 
\rv{which follows from \sa{$p=2l$}, implying $L_{\Psi}=4 l$, and from our choices of {$\tau_2=\frac{1}{16}(\frac{3}{\tau_1}+\zeta)^{-1}$ for $\tau_1\leq \frac{1}{p+l}$}.} 

Consider $A:=2 \left\langle\nabla_y \tilde{f}\left(x^*\left(y_t, z_t\right), y_t\right)-\nabla_y \tilde{f}\left(x_{t+1}, y_t\right),~\rv{y_{t+1}-y_t} \right\rangle$. \sa{Note that}
$$
\begin{aligned}
A & \geq-2 \|\nabla_y \tilde{f}\left(x^*\left(y_t, z_t\right), y_t\right)-\nabla_y \tilde{f}\left(x_{t+1}, y_t\right)\| \| \rv{y_{t+1}-y_t} \| \\
& \geq-2 l\left\|x_{t+1}-x^*\left(y_t, z_t\right)\right\| \| \rv{y_{t+1}-y_t} \| \\
& \geq 
% \rv{-l\tau_2 \nu\left\|y_{t+1}-y_t  \right\|^2-\frac{l}{ \nu\tau_2}\left\|x_{t+1}-x^*\left(y_t, z_t\right)\right\|^2,}
\hjn{-l \nu\left\|y_{t+1}-y_t  \right\|^2-\frac{l}{ \nu}\left\|x_{t+1}-x^*\left(y_t, z_t\right)\right\|^2,}
\end{aligned}
$$
where the second inequality uses \sa{the $l$-Lipschitz continuity of $\nabla_{y} \tilde{f}(\cdot,y)$ for any $y\in Y$}, and the last inequality is due to \sa{the Young's inequality} with $\nu >0$ specified later. Moreover, using Lemma \ref{lem:lip} to bound the term $\left\|x_{t+1}-x^*\left(y_t, z_t\right)\right\|^2$ yields 
\be \label{eq:A}
A \geq 
% \rv{- l \nu \tau_2  \left\|y_{t+1}-y_t \right\|^2-\frac{l}{\nu\tau_2} \gamma_3^2 \left\|x_{t+1} - x_t\right\|^2.}
\hjn{- l \nu  \left\|y_{t+1}-y_t \right\|^2-\frac{l}{\nu} \gamma_3^2 \left\|x_{t+1} - x_t\right\|^2.}
\ee
Next, consider $B:=2 p\iprod{z_{t+1}-z_t}{x^*\left(y^*\left(z_{t+1}\right), z_t\right)-x^*\left(y_{t+1}, z_{t+1}\right)}$. Note that
\be \label{eq:B}
\begin{aligned}
B = &2 p\iprod{z_{t+1}-z_t}{x^*\left(y^*\left(z_{t+1}\right), z_t\right)-x^*\left(y^*\left(z_{t+1}\right), z_{t+1}\right)} \\
& +2 p\iprod{z_{t+1}-z_t}{x^*\left(y^*\left(z_{t+1}\right), z_{t+1}\right)-x^*\left(y_{t+1}, z_{t+1}\right)} \\
 \geq & -2 p \gamma_1\left\|z_{t+1}-z_t\right\|^2+2 p\iprod{z_{t+1}-z_t}{x^*\left(y^*\left(z_{t+1}\right), z_{t+1}\right)-x^*\left(y_{t+1}, z_{t+1}\right)} \\
\geq & -\left(2 p \gamma_1+\frac{p}{6 \beta}\right)\left\|z_{t+1}-z_t\right\|^2-6 p \beta\left\|x^*\left(y^*\left(z_{t+1}\right), z_{t+1}\right)-x^*\left(y_{t+1}, z_{t+1}\right)\right\|^2,
\end{aligned}
\ee
where in the first inequality we use Lemma \ref{lem:lip} and the Cauchy-Schwarz
inequality, and for the second inequality we used Young's inequality for \rv{$\beta>0$ specified later}, \rv{i.e., $2\iprod{a}{b}\leq \frac{1}{\theta} \norm{a}^2+\theta\norm{b}^2$ for any $\theta>0$.}

Plugging \eqref{eq:A} and \eqref{eq:B} into \eqref{eq:de-V} leads to
\be \label{eq:de-V2}
\begin{aligned}
V_t- V_{t+1} \geq & \left(\frac{1}{2\tau_1}-\frac{l}{\nu} \gamma_3^2\right) \left\|x_{t+1} - x_t \right\|^2+\left( \hjn{\frac{1}{2\tau_2} - l \nu } \right) \hj{\left\| y_{t+1} - y_t 
 \right\|^2} \\
& +\left(\frac{p}{2 \beta}-2 p \gamma_1-\frac{p}{6 \beta}\right) \left\|z_t-z_{t+1}\right\|^2-6 p \beta \left\|x^*\left(y^*\left(z_{t+1}\right), z_{t+1}\right)-x^*\left(y_{t+1}, z_{t+1}\right)\right\|^2.
\end{aligned}
\ee
\rv{Next, we bound 
% $\left\|\nabla_y \tilde{f} \left(x_{t+1}, y_t\right) \right\|^2$
$\|y_{t+1} - y_t\|$ by bounding $\left\| y^+(z_t)  - y_{t+1} \right\|$. First, recall that $y_{t+1}$ and $y^+(z_t)$ have the form:
 $y_{t+1}=\prox_{\tau_2 h}\Big(y_t + \tau_2 \rv{\nabla_y} \tilde{f}(x_{t+1}, y_{t})\Big)$ and $y^+(z_t)=\prox_{\tau_2 h}\Big(y_t+\tau_2 \nabla_y \tilde{f}(x^*(y_t,z_t), y_{t})\Big)$; hence, 
 \begin{equation}
\label{eq:y+_y}
\begin{aligned}
\left\| y^+(z_t)  - y_{t+1} \right\|
&\leq \frac{\tau_2}{1 - \tau_2 \zeta} \norm{\nabla_y \tilde{f}(x_{t+1}, y_t) - \nabla_y \tilde{f}(x^*(y_t, z_t), y_t)}\\
&\leq \frac{\tau_2 l}{1 - \tau_2 \zeta} \norm{x_{t+1}-x^*(y_t, z_t)}\leq \frac{\tau_2 l \gamma_3}{1 - \tau_2 \zeta} \norm{x_{t+1}-x_t},
\end{aligned}
\end{equation}
where in the first inequality we used the fact that $\prox_{\tau_2 h}(\cdot)$ is $\frac{1}{1-\tau_2 \zeta}$-Lipschitz continuous for $0<\tau_2<\frac{1}{\zeta}$, the second inequality and 
the last one follow from  the Lipschitz continuity of $\nabla_{y} \tilde f(\cdot,y_t)$ and Lemma \ref{lem:lip}, respectively. Thus,}
% \begin{equation}
% \label{eq:grad_y_lower_bound}
% \begin{aligned} \left\|\nabla_y \tilde{f} \left(x_{t+1}, y_t\right)  \right\|^2 
% = & \sa{\left\| \nabla_y \tilde{f}(x_{t+1}, y_t) - \nabla_y \tilde{f}(x^*(y_t, z_t), y_t) + \nabla_y \tilde{f}(x^*(y_t, z_t), y_t) \right \|^2}  \\
% \geq & \sa{\frac{1}{2}\left\|\nabla_y \tilde{f}(x^*(y_t, z_t), y_t) \right\|^2 - \|\nabla_y \tilde{f}(x_{t+1}, y_t) - \nabla_y \tilde{f}(x^*(y_t, z_t), y_t)\|^2}\\
% \geq & \frac{1}{2}\left\|\nabla_y \tilde{f}(x^*(y_t, z_t), y_t) \right\|^2  - l^2\left\|x_{t+1}-x^*\left(y_t, z_t\right)\right\|^2\\
% \geq & \sa{\frac{1}{2}\left\|\nabla_y \tilde{f}(x^*(y_t, z_t), y_t) \right\|^2  - l^2\gamma_3^2\left\|x_{t+1}-x_t\right\|^2,}
% \end{aligned}
% \end{equation}
\hjn{\begin{equation}
\label{eq:grad_y_lower_bound}
\begin{aligned} \left\| y_{t+1} - y_t  \right\|^2 
\geq & \rv{\frac{1}{2}\left\|\rv{y^+(z_t)}  - y_t\right\|^2- \left\| y^+(z_t)  - y_{t+1} \right \|^2}  \\
%\geq & \frac{1}{2}\left\|\rv{y^+(z_t)}  - y_t\right\|^2 - \frac{\tau_2^2}{(1 - \tau_2 \zeta)^2} \|\nabla_y \tilde{f}(x_{t+1}, y_t) - \nabla_y \tilde{f}(x^*(y_t, z_t), y_t)\|^2\\
%\geq & \frac{1}{2}\left\|\rv{y^+(z_t)} - y_t\right\|^2  - \frac{\tau_2^2}{(1 - \tau_2 \zeta)^2} l^2\left\|x_{t+1}-x^*\left(y_t, z_t\right)\right\|^2\\
\geq & \frac{1}{2}\left\|\rv{y^+(z_t)} - y_t \right\|^2  - \frac{\tau_2^2l^2}{(1 - \tau_2 \zeta)^2} \gamma_3^2\left\|x_{t+1}-x_t\right\|^2,
\end{aligned}
\end{equation}}%
\rv{where in the first inequality we use $\norm{a+b}^2\leq 2\norm{a}^2+2\norm{b}^2$, and the second one follows from \eqref{eq:y+_y}.}
% Applying Lemma \ref{lem:lip} 
% \be \label{eq:grad-y}
% \left\|\nabla_y f\left(x_{t+1}, y_t\right)\right\|^2 \geq \left\|\nabla_y \hat{f}\left(x^*\left(y_t, z_t\right), y_t ; z_t\right)\right\|^2 / 2-l^2 \gamma_3^2 \left\|x_{t+1} - x_t\right\|^2.
% \ee

\sa{On the other hand, since} $x^*\left(y^*\left(z_{t+1}\right), z_{t+1}\right)=x^*\left(z_{t+1}\right)$, 
we have
\be \label{eq:x-solu}
\begin{aligned}
\MoveEqLeft \left\|x^*\left(y^*\left(z_{t+1}\right), z_{t+1}\right)-x^*\left(y_{t+1}, z_{t+1}\right)\right\|^2\\
=&\left\|x^*\left(z_{t+1}\right)-x^*\left(y_{t+1}, z_{t+1}\right)\right\|^2 \\
\leq & 4\left\|x^*\left(z_{t+1}\right)-x^*\left(z_t\right)\right\|^2+4\left\|x^*\left(z_t\right)-x^*\left(y^{+}\left(z_t\right), z_t\right)\right\|^2 \\
& +4\left\|x^*\left(y^{+}\left(z_t\right), z_t\right)-x^*\left(y_{t+1}, z_t\right)\right\|^2+4\left\|x^*\left(y_{t+1}, z_t\right)-x^*\left(y_{t+1}, z_{t+1}\right)\right\|^2 \\
\leq & 4\left\|x^*\left(z_t\right)-x^*\left(y_t^{+}\left(z_t\right), z_t\right)\right\|^2+4 \gamma_2^2\left\|y_t^{+}\left(z_t\right)-y_{t+1}\right\|^2+8 \gamma_1^2\left\|z_t-z_{t+1}\right\|^2 \\
%\leq &  4\left\|x^*\left(z_t\right)-x^*\left(y_t^{+}\left(z_t\right), z_t\right)\right\|^2+ \hjn{\frac{\sa{4} \gamma_2^2 \tau_2^2}{(1 - \tau_2 \zeta)^2}} \| \nabla_y \tilde{f}\left(x^*\left(y_t, z_t\right), y_t\right)-\nabla_y \tilde{f}\left(x_{t+1}, y_t\right) \|^2 \\
%& +8 \gamma_1^2\left\|z_t-z_{t+1}\right\|^2 \\
\leq & 4\left\|x^*\left(z_t\right)-x^*\left(y_t^{+}\left(z_t\right), z_t\right)\right\|^2+ \hjn{\frac{\sa{4} \gamma_2^2 \tau_2^2 l^2\gamma_3^2}{(1 -\tau_2 \zeta)^2}}\left\|x_{t+1}-x_t\right\|^2 +8 \gamma_1^2\left\|z_t-z_{t+1}\right\|^2,
\end{aligned}
\ee
\rv{which follows from Lemma \ref{lem:lip} and \eqref{eq:y+_y}.}
% \sa{Invoking Lemma \ref{lem:lip} one more time} yields
% \be \label{eq:x-solu}
% \begin{aligned}
% & 
% %\left\|x^*\left(z_{t+1}\right)-x^*\left(y_{t+1}, z_{t+1}\right)\right\|^2 
% \sa{\left\|x^*\left(y^*\left(z_{t+1}\right), z_{t+1}\right)-x^*\left(y_{t+1}, z_{t+1}\right)\right\|^2} \\
% \leq & 8 \gamma_1^2 \left\|z_t-z_{t+1}\right\|^2+4 \left\|x^*\left(z_t\right)-x^*\left(y_t^{+}\left(z_t\right), z_t\right)\right\|^2+ \hjn{\frac{\sa{4} \gamma_2^2 \tau_2^2 l^2 \gamma_3^2}{(1 -\tau_2 \zeta)^2}}  \left\|x_{t+1} - x_t \right\|^2.
% \end{aligned}
% \ee
\rv{Finally, plugging \eqref{eq:grad_y_lower_bound} and \eqref{eq:x-solu}} into \eqref{eq:de-V2}, we get
\be \label{eq:v-de-1}
\begin{aligned}
V_t- V_{t+1}
\geq & \left[\frac{1}{2\tau_1}-\frac{l}{\nu} \gamma_3^2-\left(\frac{1}{2\tau_2} - l \nu\right)\frac{\tau_2^2 l^2}{\hjn{(1-\tau_2 \zeta)^2}} \gamma_3^2 - \frac{\sa{24} p \beta \gamma_2^2 \tau_2^2 l^2}{\hjn{(1 -\tau_2\zeta)^2}}\gamma_3^2 \right] \left\|x_{t+1} - x_t\right\|^2 \\
& -24 p \beta \left\|x^*\left(z_t\right)-x^*\left(y_t^{+}\left(z_t\right), z_t\right)\right\|^2 
% +\left(\frac{1}{4\tau_2}- \frac{l \nu}{2}\right) \sa{\tau_2^2\left\|\nabla_y \tilde{f}\left(x^*\left(y_t, z_t\right), y_t\right)\right\|^2}
+\frac{1}{2}\left(\frac{1}{2\tau_2}- l \nu\right) \left\| y^+(z_t) - y_t\right\|^2
\\
& +\left[\frac{p}{2 \beta}-2 p \gamma_1-\frac{p}{6 \beta}-48 p \beta \gamma_1^2\right] \left\|z_t-z_{t+1}\right\|^2 \\
\geq & \rv{\frac{3}{10\tau_1}} \left\|x_{t+1} - x_t\right\|^2
% +\sa{\frac{\tau_2}{8}} \left\| \nabla_y \tilde{f}\left(x^*\left(y_t, z_t\right), y_t\right)  \right\|^2
+ \frac{1}{8\tau_2} \left\| y^+(z_t) - y_t\right\|^2
+\frac{p}{4 \beta} \left\|z_t-z_{t+1}\right\|^2 \\
& -24 p \beta \left\|x^*\left(z_t\right)-x^*\left(y_t^{+}\left(z_t\right), z_t\right)\right\|^2,
\end{aligned}
\ee
where \rv{the last inequality holds due to our choice of \sm{} parameters, i.e., $\tau_1\in (0, \frac{1}{p+l}]$, $\tau_2=\frac{1}{16}(\frac{3}{\tau_1}+\zeta)^{-1}$, $\beta=\alpha\min\{\mu, l\}\tau_2$ for any $\alpha\in(0,\frac{1}{406}]$ and $p=2l$,}
%$\tau_1\in(0,\frac{1}{3\hjn{(l+\zeta)} }]$, $\tau_2=\tau_1/48$, $\beta=\alpha\mu\tau_2$ for any $\alpha\in(0,\frac{1}{406}]$ and $p=2l$,
\rv{whenever $0<\nu\leq\frac{1}{4 l \tau_2}$}. \rv{Indeed, our choice implies that} $\gamma_1=2$, $\gamma_2=3$ and $\gamma_3^2 \sa{\leq} \frac{2}{\tau_1^2 l^2}+2$; moreover, choosing $\nu=\frac{1}{4 l \tau_2}\rv{\geq}\frac{\revise{12}}{l \tau_1}$, we have $\frac{1}{2\tau_2}-\sa{l \nu}=\frac{1}{4\tau_2}$ \sa{and this particular selection of parameters implies that}
\begingroup
\allowdisplaybreaks
\begin{align*}
\MoveEqLeft \frac{l}{\nu} \gamma_3^2+\left(\frac{1}{2\tau_2}- l \nu\right) \hjn{\frac{\tau_2^2l^2}{(1-\tau_2 \zeta)^2}}\gamma_3^2+ \sa{24} p \beta \gamma_2^2\frac{\tau_2^2 l^2 }{\hjn{(1-\tau_2\zeta)^2}}\gamma_3^2 \\
= & \left[\nu^{-1}\left(l \tau_1 \gamma_3^2\right)+\rv{\frac{\tau_2}{4\tau_1} \frac{1}{(1-\tau_2\zeta)^2}}\left(l^2 \tau_1^2 \gamma_3^2\right)+\rv{48\frac{\tau_2}{\tau_1}\cdot 9l\beta \frac{\tau_2}{(1-\tau_2\zeta)^2}}\left(l^2 \tau_1^2 \gamma_3^2\right)\right] \sa{\frac{1}{\tau_1}} \\
\leq &\rv{\left[\frac{2 \nu^{-1}}{\tau_1 l}+\frac{1}{96}\frac{1}{(1-\tau_2\zeta)^2}+  18 l^2 \cdot \frac{\min\{\mu,l\}}{l}\frac{\tau_2^2}{(1-\tau_2\zeta)^2}\cdot\alpha\right]} \sa{\frac{1+\tau_1^2 l^2}{\tau_1}} \\
\leq &\rv{\left[\frac{2 \nu^{-1}}{\tau_1 l}+\frac{1}{96}\frac{1}{(1-\tau_2\zeta)^2}+ \frac{\alpha}{12}\cdot\frac{1}{96}\frac{1}{(1-\tau_2\zeta)^2}\right]} \sa{\frac{1+\tau_1^2 l^2}{\tau_1}} \\
\leq & \left[\frac{1}{6}+\rv{\frac{1}{84}+\frac{\alpha}{12}\cdot\frac{1}{84}} \right] \frac{10}{9}\cdot \frac{1}{\tau_1} \leq \sa{\frac{1}{5\tau_1}}
\end{align*}
\endgroup
\rv{holds for any $\alpha\in(0,\rv{\frac{1}{2}}]$, where in the first inequality we used $\gamma_3^2 \sa{\leq} \frac{2}{\tau_1^2 l^2}+2$ and $\frac{\tau_2}{\tau_1}\leq\frac{1}{48}$, the second inequality follows from $\tau_2^2\leq\frac{\tau_1^2}{48^2}$ and $\tau_1\leq \frac{1}{3l}$, finally, for the last inequality we use the bounds $\nu^{-1}\leq \frac{l\tau_1}{12}$ and $1+\tau_1^2l^2\leq \frac{10}{9}$ together with $\frac{1}{(1-\tau_2\zeta)^2}\leq (\frac{16}{15})^2$; hence, $\frac{1}{96}\cdot\frac{1}{(1-\tau_2\zeta)^2}\leq \frac{1}{84}$, which follows from $\tau_2\leq\frac{1}{16\zeta}$. On the other hand, we also have}
$$
\frac{p}{2 \beta}-2 p \gamma_1-\frac{p}{6 \beta}-48 p \beta \gamma_1^2 %\geq
\rv{=}\left[\frac{1}{\sa{3}}-4 \beta-192 \beta^2\right] \frac{p}{\beta} \geq \frac{p}{4 \beta}
$$
\sa{holds for all $\beta\in(0,\frac{1}{78}]$, and since \rv{$\beta=\alpha\min\{\mu,l\}\tau_2\leq \alpha \min\{\mu,l\} \frac{\tau_1}{48}\leq \frac{\alpha}{144} \frac{\min\{\mu,l\}}{l} \leq\frac{\alpha}{144}$,} we conclude that the inequality holds for all $\alpha\in(0,1]$.}

\sa{Next, we bound the last term on the right hand side of \eqref{eq:v-de-1} using Lemma \ref{lem:bound-new} as follows:}
\begin{equation}
    \label{eq:xsxs}
\begin{aligned}
\MoveEqLeft 24 p \beta\left\|x^*(\sa{z_t})-x^*\left(y^{+}(\sa{z_t}), \sa{z_t}\right)\right\|^2 \\
\leq & 
{\frac{72p\beta}{\mu(p-l) \tau_2^2} (1 + \tau_2^2 l^2 + \gamma_2^2 \tau_2^2 l^2 )  \| y^+(z_t) - y_t\|^2 } 
\leq {\frac{1}{16\tau_2} \|y^+(z_t) - y_t\|^2,}
\end{aligned}
\end{equation}
where \sa{in the second inequality we use $p=2l$, \rv{$\gamma_2=3$,} $\tau_2 l \leq 1/144$, \rv{$\beta=\alpha\min\{\mu, l\}\tau_2$} for %\redsout{$\alpha\in(0,\frac{1}{816}]$} 
\hjn{$\alpha\in(0,\frac{1}{2306}]$} implying that}
% $$
% \redsout{\frac{24 p \beta}{(p-l) \mu}\left(1+\tau_2 l\frac{2p}{p-l}\right)^2=\frac{48 \beta}{\mu}\left(1+4\tau_2 l\right)^2 \leq %\frac{96 \beta}{\mu} 
% \frac{\sa{51} \beta}{\mu}\leq %\frac{1}{16} 
% \sa{\frac{1}{16}} \tau_2.}
% $$
{$$\frac{72 p \beta}{\mu(p-l)\rv{\tau_2^2}} (1 + \tau_2^2 l^2 + \gamma_2^2 \tau_2^2 l^2 )    \leq \frac{144 \beta}{\mu \tau_2^2} \Big(1 + \frac{10}{144^2}\Big)\rv{\leq \frac{\alpha}{\tau_2} \Big(144 + \frac{10}{144}\Big)\leq \frac{1}{16\tau_2}.} 
$$}%
\sa{Plugging the bound \rv{in~\eqref{eq:xsxs}} into \eqref{eq:v-de-1} and using the fact $\norm{z_{t+1}-z_t}=\beta\norm{x_{t+1}-z_t}$, we get}
\be \label{eq:v-descent}
\ V_t- V_{t+1} \geq \sa{\frac{3}{10\tau_1}} \left\|x_{t+1}-x_t\right\|^2+{\frac{1}{16\tau_2} \left\|y^+(z_t) - y_t \right\|^2}+\frac{p \beta}{4} \left\|x_{t+1}-z_t\right\|^2.
\ee
\textbf{\rv{Stationarity} Measure:}
% First we note that
% $$
% \begin{aligned}
% \left\|\grad_x \tilde{f}\left(x_t, y_t\right)\right\| \leq & \left\|\grad_x \hat{f}\left(x_t, y_t ; z_t\right)\right\|+p\left\|x_t-z_t\right\| \\
% \leq & \left\|\grad_x \hat{f}\left(x_t, y_t ; z_t\right)\right\|+p\left\|x_t-x_{t+1}\right\|+p\left\|x_{t+1}-z_t\right\| \\
% \leq & \left\|\grad_x \hat{f}\left(x_t, y_t ; z_t\right)\right\|+p \tau_1\left\|\hat{G}_x\left(x_t, y_t, \xi_1^t ; z_t\right)\right\|+p\left\|x_{t+1}-z_t\right\| .
% \end{aligned}
% $$
% Taking square and expectation
% $$
% \begin{aligned}
% \mathbb{E}\left\|\grad_x f\left(x_t, y_t\right)\right\|^2 \leq & 6 \mathbb{E}\left\|\grad_x \hat{f}\left(x_t, y_t ; z_t\right)\right\|^2+6 p^2 \tau_1^2 \mathbb{E}\left\|\grad_x \hat{f}\left(x_t, y_t ; z_t\right)\right\|^2\\
% & +6 p^2 \mathbb{E}\left\|x_{t+1}-z_t\right\|^2+6 p^2 \tau_1^2 \sigma^2 \\
% = & 6\left(1+p^2 \tau_1^2\right) \mathbb{E}\left\|\grad_x \hat{f}\left(x_t, y_t ; z_t\right)\right\|^2+6 p^2 \mathbb{E}\left\|x_{t+1}-z_t\right\|^2+6 p^2 \tau_1^2 \sigma^2 .
% \end{aligned}
% $$
\rv{From the first-order optimality conditions, for all $t\geq 1$, we have $G_t^x\in\nabla_x\tilde f(x_t,y_t)+\partial \delta_X(x_t)$ and $G_t^y\in\nabla_y\tilde f(x_t,y_t)-\partial h(y_t)$, where $G_t^x$ and $G_t^y$ are defined in \eqref{eq:subgradient-G}.} Note that given $a,b,c\in\cX$ and a closed convex set $X\subset\cX$, it holds that $\norm{\cP_X(a+b)-c}\leq \norm{\cP_X(a)-c}+\norm{b}$; hence, 
% using it \rv{together with $\nabla_x \hat{f}(x_t,y_t; z_t)=\nabla_x \tilde{f}(x_t,y_t)+p(x_t-z_t)$} implies
% \[ 
% \begin{aligned}
%     \hjn{\tau_1 \|G_t^x \|} = & \|\Pcal_{X}( x_t - \tau_1 \nabla_x \tilde{f}(x_t,y_t)) - x_t \| \\
%     \leq & \| \Pcal_{X}( x_t - \tau_1 \nabla_x \hat{f}(x_t,y_t; z_t)) - x_t  \| + \tau_1 p \|x_t - z_t\|  \\
%     \leq &  \| \Pcal_{X}( x_t - \tau_1 \nabla_x \hat{f}(x_t,y_t; z_t)) - x_t  \| + \tau_1 p \|x_{t+1} - x_t\| + \tau_1 p\|x_{t+1} - z_t\|. 
% \end{aligned}
% \] 
\[ 
\begin{aligned}
    \rv{\tau_1 \|G_{t+1}^x \|} \leq & \|x_{t+1} - x_t \| + \tau_1 \| \nabla_x \tilde f(x_{t+1}, y_{t+1}) - \nabla_x \tilde f(x_t,y_t) \| + \tau_1 p\|z_t - x_t\| \\
    \leq & \| x_{t+1}- x_t  \| + \tau_1l(\|x_{t+1} - x_t\| + \|y_{t+1} - y_t\|)+ \tau_1 p \|x_t - z_t\|  \\
    \leq &  (1 + \tau_1 p+\tau_1 l) \|x_{t+1} - x_t\| + \tau_1 p\|x_{t+1} - z_t\| + \tau_1 \rv{l} \|y_{t+1} - y_t\|. 
\end{aligned}
\] 
Then, 
% \sa{using %the definition of 
% \rv{$x_{t+1}=\Pcal_{X}( x_t - \tau_1 \nabla_x \hat{f}(x_t,y_t; z_t))$} and 
\mod{using} the inequality $ (a+b+c)^2 \leq 3(a^2 +b^2 +c^2)$, we get
\be \label{eq:grad-x} 
\begin{aligned}
\rv{\|G_{t+1}^x\|^2 }
\leq  3\frac{(1 +\tau_1 p + \tau_1 l)^2}{\tau_1^2} \|x_{t+1} - x_t\|^2 + 3 p^2\|x_{t+1} - z_t\|^2 + {3 \rv{l^2}\|y_{t+1} - y_t\|^2}. 
\end{aligned}\ee
In addition, \rv{for $0<\tau_2<\frac{1}{\zeta}$, $\prox_{\tau_2 h}(\cdot)$ is $\frac{1}{1-\tau_2 \zeta}$-Lipschitz continuous, and for any given $a,b,c\in\cY$, it also holds that $\norm{\prox_{\tau_2 h}(a+b)-c}\leq \norm{\prox_{\tau_2 h}(a)-c}+\frac{1}{1-\tau_2 \zeta}\norm{b}$; hence, together with the definition of $y^+(z_t)=\prox_{\tau_2 h}\Big(y_t+\tau_2 \nabla_y \tilde{f}(x^*(y_t,z_t), y_{t})\Big)$ and using Lemma ~\ref{lem:lip},} it also holds that
\begin{equation}
\label{eq:y-dif-bound}
\begin{aligned}
\| y_{t+1} - y_t \| 
= & \left\| \prox_{\tau_2 h}(y_t + \tau_2 \nabla_y \tilde f(\mod{x_{t+1}},y_t))   - y_t \right \|  \\
\leq & \norm{y^+(z_t) - y_t} +
\frac{\tau_2}{1 -\tau_2 \zeta} \left\|\nabla_y \tilde{f}\left( \mod{x_{t+1}}, y_t\right)-\nabla_y \tilde{f} \left(x^*(y_t, z_t), y_t\right)\right\| \\
\leq & \rv{\norm{y^+(z_t) - y_t}} +\frac{\tau_2 l }{1 -\tau_2 \zeta}  \left\|x_{t+1}-x^*\left(y_t, z_t\right)\right\|,\\
\leq & \rv{\norm{y^+(z_t) - y_t}} +\frac{\tau_2 l }{1 -\tau_2 \zeta} \gamma_3 \left\|x_{t+1}-x_t\right\|.
\end{aligned}
\end{equation}
Note that $\tau_2 \left\|  \rv{G_{t+1}^y} \right\| 
\leq \| y_{t+1} - y_t \| {+ \tau_2 l\| y_{t+1} - y_t \|}$; therefore, we get
\begin{equation}
\label{eq:grad-y-eqv}
\begin{aligned}
\left\|  \rv{G_{t+1}^y} \right\|^2 
\leq  \rv{\frac{2(1 +\tau_2 l)^2}{\tau_2^2}\left(\norm{y^+(z_t) - y_t}^2+\frac{\tau_2^2  l^2}{(1 -\tau_2 \zeta)^2}  \gamma_3^2\left\|x_{t+1} - x_t\right\|^2\right)}.
\end{aligned}
\end{equation}
Then, \rv{let $\bar\kappa:=\max\{\kappa,1\}$, using \eqref{eq:grad-x}, \eqref{eq:y-dif-bound} and \eqref{eq:grad-y-eqv}}, we have
\be \label{eq:grad-comb}
\begin{aligned}
\MoveEqLeft \|G_{t+1}^x\|^2+ \rv{\bar\kappa}\left\| G_{t+1}^y \right\|^2 \sa{-3 p^2 \|x_{t+1} - z_t \|^2}\\
\leq & \rv{\frac{3}{\tau_1^2}\left( (1 + \tau_1 p + \tau_1 l)^2 +  \Big(l^2+\frac{\bar\kappa(1+\tau_2 l)^2}{3\tau_2^2}\Big)\frac{ 2\tau_1^2\tau_2^2 l^2 }{(1 -\tau_2 \zeta)^2 } \gamma_3^2\right)} \left\|x_{t+1} - x_t\right\|^2 \\
& + \rv{\frac{2}{\tau_2^2}\Big(\bar\kappa(1 +\tau_2 l)^2+3\tau_2^2l^2\Big)}\left\| y^+(z_t) - y_t   \right\|^2,\\
\leq &\rv{\frac{172}{10}\frac{\bar\kappa}{\tau_1^2}\norm{x_{t+1}-x_t}^2+\frac{33}{16}\frac{\bar\kappa}{\tau_2^2}\norm{y^+(z_t)-y_t}^2,}
\end{aligned}
\ee
\rv{which follows from the following bounds: $2\bar\kappa(1+\tau_2 l)^2+6\tau_2^2 l^2\leq 2\bar\kappa[(1+\frac{1}{144})^2+\frac{3}{144^2}]\leq \frac{33}{16}\bar\kappa$, $3(1+ \tau_1 p + \tau_1 l)^2=(1+3 l\tau_1)^2\leq 12$ and
$$
\begin{aligned}
3\Big(l^2+\frac{\bar\kappa(1+\tau_2 l)^2}{3\tau_2^2}\Big)\frac{ 2\tau_1^2\tau_2^2 l^2 }{(1 -\tau_2 \zeta)^2 } \gamma_3^2
& \leq \Big(\tau_2^2l^2+\frac{\bar\kappa}{3}(1+\tau_2 l)^2\Big)\cdot\frac{12}{(1 -\tau_2 \zeta)^2}\cdot(1+\tau_1^2 l^2)\\
& \leq 12\Big(\frac{1}{144^2}+\frac{\bar\kappa}{3}\Big(1+\frac{1}{144}\Big)^2\Big)\Big(\frac{16}{15}\Big)^2\Big(1+\frac{1}{9}\Big)\leq\frac{52}{10}\bar\kappa,
\end{aligned}
$$}%
where we used \rv{$\tau_1 l\leq \frac{1}{3}$ and $\tau_2 l \leq \frac{1}{144}$ for $p=2l$, $\gamma_3^2\leq 2(1+\frac{1}{\tau_1^2 l^2})$ and $\frac{1}{(1-\tau_2\zeta)^2}\leq (\frac{16}{15})^2$ together with $\bar\kappa\geq 1$.}

\textbf{Putting pieces together:} 
\sa{Combining \eqref{eq:v-descent} and \eqref{eq:grad-comb}, we get}
% \be \label{eq:decen-sum}
% \begin{aligned}
% & \frac{1}{\tau_1^2}\left\|\Pcal_{X}(x_t - \tau_1 \sa{\nabla_x\tilde{f}(x_t,y_t)}) - x_t\right\|^2+ \kappa  \left\| \nabla_y \tilde{f}\left(x_t, y_t\right) \right\|^2 \\
% \leq & \frac{12\kappa}{\tau_1^2} \left\|x_{t+1} -x_t\right\|^2+ 3 \kappa \left\| \nabla_y \tilde{f} \left(x^*(y_t, z_t), y_t\right) \right\|   +3\sa{p^2} \left\|z_t-x_{t+1}\right\|^2\\
% \leq & \max \left\{\frac{%48
% \sa{40}\kappa}{\tau_1}, \frac{%96
% \sa{48}\kappa}{\tau_2}, \frac{12 p}{\beta}\right\}\left[ V_t- V_{t+1}\right] \\
% \leq & \frac{O(1) \kappa}{\tau_2}\left[ V_t- V_{t+1}\right],
% \end{aligned}
% \ee
{\be \label{eq:decen-sum}
\begin{aligned}
\MoveEqLeft\left \| G_{t+1}^x \right\|^2+ \rv{\bar\kappa}  \left\|G_{t+1}^y \right\|^2 \\
\leq & \rv{\frac{172}{10}}\frac{{\bar\kappa}}{\tau_1^2} \left\|x_{t+1} -x_t\right\|^2+ \rv{\frac{33}{16}}\frac{{\bar\kappa}}{\rv{\tau_2^2}} \left\| y^+(z_t) - y_t \right\|^2   +3\sa{p^2} \left\|x_{t+1}-z_t\right\|^2\\
\leq & \max \left\{\frac{%48
\rv{58}\rv{\bar\kappa}}{\tau_1}, \frac{%96
\rv{33}\rv{\bar\kappa}}{\tau_2}, \frac{12 p}{\beta}\right\}\left( V_t- V_{t+1}\right) \\
\leq & \frac{O(1) \rv{\bar\kappa}}{\tau_2}\left( V_t- V_{t+1}\right),
\end{aligned}
\ee}%
where in the third inequality we use \rv{$\frac{1}{\tau_1}\leq \frac{1}{48 \tau_2}=\cO(\frac{1}{\tau_2})$ and $\frac{p}{\beta}=\frac{2}{\alpha}\frac{l}{\min\{\mu, l\}}\cdot\frac{1}{\tau_2}=\cO\Big(\frac{\rv{\bar\kappa}}{\tau_2}\Big)$ since $\frac{2}{\alpha}\leq 4612$ and $\frac{l}{\min\{\mu, l\}}=\bar\kappa$.} Thus, \eqref{eq:decen-sum} directly implies that
$$
\begin{aligned}
    & 
    % \frac{1}{T} \sum_{t=0}^{T-1} 
    % \frac{1}{\tau_1^2}\left\|\Pcal_{X}(x_t - \tau_1 \sa{\nabla_x\tilde{f}(x_t,y_t)}) - x_t \right\|^2+\kappa \left\| \nabla_y \tilde{f}\left(x_t, y_t\right) \right\|^2 
    \frac{1}{T} \sum_{t=0}^{T-1} \Big(
    \left\|G_{t+1}^x \right\|^2+\rv{\bar\kappa} \left\| G_{t+1}^y \right\|^2\Big)
    \leq  \frac{O(1) \rv{\bar\kappa}}{T}\rv{\Big(\frac{1}{\tau_1}+\zeta\Big)}\left[V_0-\min_{x\in X, y\in Y, z\in\cX} V(x, y; z)\right].
\end{aligned}
$$
\rv{%Note that due to strong duality, 
For any $x\in X$, $y\in Y$ and $z\in \cX$, let
\begin{equation}
\label{eq:Delta}
    \Delta(x,y;z):=\hat{f}_r(x, y ; z)-\Psi_r(y ; z)+P(z)-\Psi_r(y ; z).
\end{equation}
Fix an arbitrary $z\in\cX$, weak duality implies that $\Delta(x,y;z)\geq 0$ for all $x\in X$ and $y\in Y$; moreover, according to Lemma~\ref{lem:minmax}, we can find $x\in X$ and $y\in Y$ such that \rv{$\Delta(x,y;z)=0$}. Therefore, for any $z\in \cX$, we have $\min_{x\in X, y\in Y} \Delta(x,y;z)=0$, which implies that $\min_{x\in X, y\in\cY, z\in\cX} V(x, y; z)=\min_{z\in\cX} \{P(z)+\min_{x\in X, y\in Y}\Delta(x,y;z)\}=\min_{z\in\cX}P(z)$. This observation leads to the following identity: 
$$
\begin{aligned}
\MoveEqLeft V_0-\min_{{x\in X, y\in\cY, z\in\cX}} V(x, y; z) \\
= & P\left(z_0\right)+ \Delta(x_0,y_0;z_0)
%\rv{\hat{f}_r}\left(x_0, y_0 ; z_0\right)-\rv{\Psi_r}\left(y_0 ;z_0\right)+P\left(z_0\right)-\rv{\Psi_r}\left(y_0 ; z_0\right)\\
-\min_{x\in X, y\in\cY, z\in\cX}\Big\{P(z)+ \Delta(x,y;z)
%\rv{\hat{f}_r}(x, y ; z)-\rv{\Psi_r}(y ; z)+P(z)-\rv{\Psi_r}(y ; z)
\Big\} \\
%\leq
\rv{=} & P\left(z_0\right)-\min_{z\in\cX} P(z)+\Delta(x_0,y_0;z_0).
%\rv{\hat{f}_r}\left(x_0, y_0 ; z_0\right)-\rv{\Psi_r}\left(y_0 ; z_0\right)+P\left(z_0\right)-\rv{\Psi_r}\left(y_0 ; z_0\right).
\end{aligned}
$$}%
\sa{%Note that since $p=2l$, 
\rv{Furthermore, for any $z\in X$,} we have}
\begin{equation}
\label{eq:P-Phi-relation}
   P(z)=\sa{\min_{x\in X} \max_{y\in Y} \rv{\tilde f_r}(x, y)+\rv{\frac{p}{2}}\|x-z\|^2}=\min_{x\in X} \sa{\Phi}(x)+\rv{\frac{p}{2}}\|x-z\|^2\leq \sa{\Phi(z)}, 
\end{equation}
%and since $P(\cdot)$ is the Moreau envelope of $\Phi$, 
and we also have $P(z)\geq \min_{x\in X} \Phi(x)$ for all $z\in X$; therefore, $\min_{z\in\cX} P(z)=\min_{x\in X} \Phi(x)$. Hence, \sa{for any $x_0,z_0\in X$ and $y_0\in Y$, we have}
$$
V_0-\sa{\min_{x\in X, y\in Y, z\in\cX} V(x, y, z) = P\left(z_0\right)-\min_{z\in X} \Phi(z)+\rv{\Delta_0},}
$$
where $\rv{\Delta_0}:=
%\hat{f}_r\left(x_0, y_0 ; z_0\right)-\Psi_r\left(y_0 ; z_0\right)+P\left(z_0\right)-\Psi_r\left(y_0 ; z_0\right)
\Delta(x_0,y_0;z_0)$. Thus, when we initialize $z_0\in \cM$, for the final complexity bound we get
$$
\frac{1}{T} \sum_{t=0}^{T-1} \Big(\left\| G_{t+1}^x \right\|^2 +\rv{\bar\kappa} \left\| G_{t+1}^y \right\|^2\Big) \leq \frac{O(1) \rv{\bar\kappa}}{T}\rv{\Big(\frac{1}{\tau_1}+\zeta\Big)\Big(P(z_0)-\bar F+\Delta_0\Big)},
$$
%\sa{with $\Delta:=F\left(z_0\right)-%\Phi^*
%\rv{\bar F}$, 
where we used the fact that %$F(z)=\Phi(z)$ for $z\in\cM$ and 
$\min_{z\in X}\Phi(z)\geq\min_{z\in %\cX
\rv{X}}F(\rv{A(z)})= %\Phi^*
\rv{\bar F}$.
%\rv{Note that $\frac{x_t-x_{t+1}}{\tau_1}\in \nabla_x \tilde{f}(x_t, y_t) + p(x_t - z_t) + \partial \delta_{X}(x_{t+1})$ and $\frac{y_t - y_{t+1}}{\tau_2} \in - \nabla_y \tilde f(x_{t+1}, y_t) + \partial h(y_{t+1})$}, 
% for $x_t = \Pcal_{X}(x_t -\tau_1 \nabla_x \tilde f(x_t,y_t))$}, 
\rv{Moreover, since we have $G_t^x\in\nabla_x\tilde f(x_t,y_t)+\partial \delta_X(x_t)$ and $G_t^y\in\nabla_y\tilde f(x_t,y_t)-\partial h(y_t)$ for all $t\geq 1$, it follows that}
% ${\rm dist}(x_t, \nabla \tilde{f}(x_t,y_t) + \partial \delta_X(x_t) ) \leq \frac{1}{\tau_1} \| \Pcal_{X}(x_t - \tau_1 \sa{\nabla_x\tilde{f}(x_t,y_t)}) - x_t \| \hjn{= \| G_t^x\|}
\[ \begin{aligned}
    & {\rm dist}\Big(0, \nabla_x \tilde{f}(x_{t},y_{t}) + \partial \delta_X(x_{t})\Big)\leq \norm{G^x_t},\quad {\rm dist}\Big(0, - \nabla_y \tilde{f}(x_{t},y_{t}) + \partial h (y_{t}) \Big) \leq \norm{G^y_t}.
\end{aligned}
\] 
\rv{Therefore, for any $T\in\mathbb{Z}_+$, we can conclude that $\min_{1 \leq t \leq T} D_t =\mathcal{O}\Big(\frac{\bar \kappa}{T}\Big)$, where
\[ 
\begin{aligned}
    D_t:= {\rm dist}^2\Big(0, \nabla_x \tilde{f}(x_{t},y_{t}) + \partial \delta_{X}(x_{t}) \Big)+\bar\kappa~{\rm dist}^2\Big(0, - \nabla_y \tilde{f}(x_{t},y_{t}) + \partial h(y_{t}) \Big),\quad\forall~t\in\mathbb{Z}_+.
\end{aligned} 
\]}% 
\end{proof}

\section{Proof of Theorem~\ref{thm:mu-pl}}
\label{sec:PL-limit-proof}
\begin{proof}
    \sa{Note that $z_0\in X$ and $\{x_t\}\subset X$ by the construction; hence, through induction one can argue that $\{z_t\}\subset X$.} Moreover, \rv{from \eqref{eq:v-descent} in the proof of \cref{thm},} we know that $\{V_t\}_{t \geq 0}$ is non-increasing sequence, where $V_t \triangleq \rv{\hat{f}_r(x_t, y_t; z_t) + 2 P(z_t)- 2\Psi_r(y_t; z_t)}$. We clearly have %$\hat{f}_r(x_t, y_t; z_t) - \Psi_r(y_t; z_t) \geq 0$ and 
    $P(z_t) - \rv{\Psi_r(y_t; z_t)} \geq 0$ \sa{from weak duality}; \rv{therefore,}
    % and we also have $P(z_t) - \tilde\Phi^* \geq 0$ since $\tilde \Phi^*=\min_{z\in\cX}P(z)$ as $P(\cdot)$ is a Moreau envelope of $\tilde \Phi(\cdot)$}; hence, $V_0-\tilde\Phi^*\geq V_t-\tilde\Phi^*\geq 0$ for $t\geq 0$. 
%     Therefore, we can conclude that
%     \begin{equation}
% \label{eq:potential-bound}
%  0 \leq P(z_t) - \Psi(y_t; z_t) \leq V_t- \tilde\Phi^* \leq V_0 - \tilde\Phi^*, \quad \forall t \geq 0.\end{equation}
% By the initialization of $x_0=x^*(z_0)$ and $y_0\in Y^*(z_0)$ with $z_0 \in \Mcal$, we have $V_0 = P(z_0) \leq \max_{y \in \cY} \tilde{f}(z_0, y) - h(y)= \max_{y \in \cY} f(z_0, y) - h(y) = \Phi(z_0)$, which is independent of $\rho$.}
% Recall that $ \Psi(y; z) \triangleq \min_{x \in X} \hat{f}(x, y) + \frac{\hj{p}}{2} \|x - z\|^2$,
% \sa{
    % \hjn{The boundedness of $C$ and $Y$} implies that $\{x_t,y_t,z_t\}$ stays in a bounded set; therefore, it has at least one limit point $(x^*,y^*,z^*)$.} 
    % Consider the potential function $V_t= \hat{f}(x_t, y_t; z_t) - 2\Psi(y_t,z_t) + 2P(z_t)$ for $t\geq 0$. 
    %Furthermore, 
    it holds for all $t\geq 0$, $\rv{\hat{f}_r(x_t,y_t;z_t)} %\hjn{-h(y_t)}
    \leq V_t\leq V_0$. %where we used $P(z_t)\geq \rv{\Psi_r(y_t,z_t)}$.

    \rv{For any $x\in X$, define $Y\supset R^*(x) \triangleq \arg \max_{y\in\cY} \tilde{f}_r(x, y) =\arg \max_{y\in\cY} f(A(x), y)-h(y)$.}
    \rv{Define $\bar l_y:=\max_{x \in X, y \in Y} \|\nabla_y f(A(x),y) \|<\infty$, and 
    %Moreover, 
    let $r^*(x_t)\in R^*(x_t)$ be an arbitrary maximizer of $\tilde f_r(x_t,\cdot)$. Since $\| \nabla_y\hat f(x_t,y;z_t)\| \leq \bar l_y$ for all $y\in Y$, mean-value theorem implies that
    \begin{align*}
        \hat f(x_t,r^*(x_t);z_t)-\hat f(x_t,y_t;z_t)\leq \bar l_y \norm{y_t-r^*(x_t)};
    \end{align*}
    moreover, $h$ being Lipschitz on $Y$ implies that $h(y_t)-h(r^*(x_t))\leq l_h \norm{y_t-r^*(x_t)}$. Thus, summing the two inequalities, we get}
    \begin{equation} \label{eq:low-boun-lip}
        \Phi(x_t)+\frac{p}{2}\norm{x_t-z_t}^2-{(\bar l_y + l_h)\|y_t - r^*(x_t)\|}
        % \frac{L_{yy} + \hjn{\zeta}}{2}\norm{y_t-r^*(x_t)}^2
        \leq \hat{f}_r(x_t,y_t;z_t), %\hjn{- h(y_t)}, 
    \end{equation}
    where we used \rv{$\Phi(x_t)+\frac{p}{2}\norm{x_t-z_t}^2=\Phi(x_t;z_t)=\max_{y\in\cY}\hat f_r(x_t,y;z_t)= \hat f_r(x_t,r^*(x_t);z_t)=\hat f(x_t,r^*(x_t);z_t)-h(r^*(x_t))$}.
    %\nsa{You wrote "we use $\hjn{0 \in \nabla_y \tilde f(x_t,r^*(x_t)) - \partial h(r^*(x_t))}$ due to optimality" Where do you use it?} 
    %\nhj{We now don't need it any more.}
    Therefore, for all $t\geq 0$, %it holds that 
    $\Phi(x_t)\leq V_0+{(\bar l_y + l_h)} {D_Y}$ with $D_Y:= \max_{y_1,y_2 \in Y}\|y_1 - y_2\|$. %$\leq V_0+\frac{L_{yy}}{2}\bar\delta^2$ where $\bar\delta^2\in\R_+$ such that $\delta_t^2\leq \bar\delta^2$ for $t\geq 0$ 
    Thus, {by definition of $\Phi$}, we get $\rv{F}(A(x_t))+\frac{\rho}{4}\norm{c(x_t)}^2=\Phi(x_t)\leq V_0+{(\bar l_y +l_h)D_Y}$ for $t\geq 0$. Note that $\rv{\rv{F}(A(x_t))\geq} \min_{x\in %\cX
    \rv{X}}F(A(x))
    %\geq\Phi^*:=\min_{x\in \cX}\Phi(x)
    \rv{=\bar F}$; therefore, %$\Phi(A(x_t))\geq \Phi^*$ and 
    we get
    \begin{align}
        \norm{c(x_t)}^2\leq \frac{4}{\rho}\left(V_0-%\Phi^*
        \rv{\bar F}+ {(\bar l_y + l_h)D_Y} \right),\quad\forall~t\geq 0.
    \end{align}
    % Since we initialize $x_0=x^*(z_0)$ and $y_0=y^*(z_0)$ for some arbitrary $z_0\in\cM$, from the discussion in the proof of \cref{lem:bounded-sequence}, we have $V_0=P(z_0)\leq \Phi(z_0)$, which is independent of the parameter $\rho>0$. 
    \rv{Recall the definition of the gap function $\Delta(\cdot,\cdot;\cdot)$ given in \eqref{eq:Delta}. If we initialize $x_0=x^*(z_0)$ and $y_0\in Y^*(z_0)$ for some arbitrary $z_0\in\cM$, we have $\Delta(x_0,y_0;z_0)=0$; hence, together with \eqref{eq:P-Phi-relation} and $z_0\in\cM$, it implies that $V_0 = P(z_0) \leq \max_{y \in \cY} \tilde{f}_r(z_0, y)= \max_{y \in \cY} f_r(z_0, y)= \rv{F}(z_0)$, which is independent of $\rho$ --the first equality follows from the fact that $z_0\in\cM$ implies $A(z_0)=z_0$ and $c(z_0)=0$. On the other hand, if we initialize $x_0=z_0$ and $y_0\in Y^*(z_0)$ for some arbitrary $z_0\in\cM$, then $\Delta_0=\hat f_r(z_0,y_0;z_0)-P(z_0)$ since $P(z_0)=\Psi_r(y_0;z_0)$ whenever $y_0\in Y^*(z_0)$. Moreover, since $V_0=P(z_0)+\Delta_0$, we have $V_0=\hat f_r(z_0,y_0;z_0)=\tilde f_r(z_0,y_0)=f_r(z_0,y_0)\leq F(z_0)$, which is independent of $\rho$ as well.}
    Therefore, for {$\rho\geq
    \rv{16\Big(F(z_0)-%\Phi^*
    \rv{\bar F}+ (\bar l_y+l_h)D_Y\Big)}$}, 
    % Recall that $\delta_t\to 0$ as $t\to\infty$; therefore, there exists $\bar T_\rho\in\mathbb{Z}_+$ such that $\delta_t\leq \frac{\rho}{8 L_{yy}}$ for all $t\geq \bar T_\rho$. 
    we can conclude that $\norm{c(x_t)}\leq \frac{1}{2}$ for all $t\geq 0$. 
    
     Let $x_t = u_ts_tv_t^\top$ be the compact singular value decomposition of $x_t$. Then, $\norm{c(x_t)}=\|x_t^\top x_t - I_r\| = \| v_t (s_t^2 - I_r) v_t^\top \| = \| s_t^2 - I_r\|$. Note that \sa{whenever ${\rm dist}(x_t, \Mcal)\leq\frac{1}{2}$, projection on to $\cM$ is well-defined\hj{, i.e., single-valued and Lipschitz continuous;} \rv{moreover,}} ${\rm dist}(x_t, \Mcal) = \|x_t - \Pcal_{\Mcal}(x_t)\| = \|u_ts_tv_t^\top - u_tv_t^\top\| = \|s_t - I_r\|$ and $\|s_t^2-I_r\| \geq\|s_t - I_r\|$, i.e., $\norm{c(x_t)}\geq {\rm dist}(x_t, \Mcal)$. Hence, we can conclude that for any \rv{$\rho\geq 16\Big(F(z_0)-%\Phi^*
     \rv{\bar F+}(\bar l_y+l_h)D_Y\Big)$}, we have  ${\rm dist}(x_t, \Mcal)\leq \frac{1}{2}$ for all $t\geq 0$. 
    %Hence, fixing $\rho>0$ arbitrarily large, possibly depending on given $C>0$, we can ensure that whole sequence $\{x_t\}$ would be arbitrarily close to the manifold. 
    On the other hand, since $\cM$ is compact, for any $C>0$ such that $C>\frac{1}{2}+\sup_{x\in\cM}\norm{x}$, it must hold that $\|x_t\|<C$ for all $t\geq 0$. Thus, \cref{thm} implies that $\sum_{\rv{t=1}}^\infty\norm{\nabla_x\tilde f(x_t,y_t)}^2+ {\bar \kappa~{\rm dist}( 0, -\nabla_y\tilde f(x_t,y_t) + \partial h(y_t))^2}=\cO( \bar \kappa)$, and we have $\min\{\norm{\nabla_x\tilde f(x_t,y_t)}^2+{\bar \kappa{\rm dist}( 0, -\nabla_y\tilde f(x_t,y_t) + \partial h(y_t))^2}:\ t=1,\ldots,T\}=\cO(\bar \kappa/T)$ for all $T\geq 1$. 
%     Let $\{t_n\}_{n\geq 0}\subset\mathbb{Z}_+$ be a convergent subsequence, i.e., $(x_{t_n}, y_{t_n}, z_{t_n})\to(x^*,y^*,z^*)$ as $n\to\infty$. From the proof of \cref{thm}, it follows that
%     \be
% \ V_t- V_{t+1} \geq \sa{\frac{3}{10\tau_1}} \left\|x_{t+1}-x_t\right\|^2+\frac{\tau_2}{16} \left\|\nabla_y \tilde{f} \left(x^*(y_t, z_t), y_t\right) \right\|^2+\frac{p \beta}{4} \left\|z_t-x_{t+1}\right\|^2.
% \ee
    Therefore, for any $\epsilon>0$ given, the algorithm can generate $(x_\epsilon,y_\epsilon)\rv{\in X\times Y}$ such that $\norm{\nabla_x \tilde{f}(x_\epsilon,y_\epsilon)}\leq \epsilon$, ${\bar \kappa{\rm dist}( 0, -\nabla_y\tilde f(x_\epsilon,y_\epsilon) + \partial h(y_\epsilon)) \leq \epsilon}$, and  ${\rm dist}(x_\epsilon, \Mcal)\leq \frac{1}{2}$ within $\cO(\frac{\rv{\bar\kappa}}{\epsilon^2})$ iterations of \sm{}.
    % \sa{Moreover, since $\Phi^*\leq\tilde\Phi^*$ and $\{z_t\}\subset X$, it follows from \cref{lem:bounded-sequence} that $y_t\in Y_c:=\{y\in\cY: \norm{y-y_0}\leq \sqrt{\frac{2}{\mu}(\Phi(z_0)-\Phi^*)}+4\kappa_{yx}C\}$ for all $t\geq 0$ and note that the bound does not depend on $\rho$ parameter. Hence, setting $\rho\geq \max\{32(\Phi(z_0)-\Phi^*), 18\bar L_g\}$ where $\bar L_g:=\sup_{y\in Y_c}L_g(y)<\infty$, and}
    Hence, \rv{setting $\rho\geq \max\{16\Big(F(z_0)-\rv{\bar F}+ (\bar l_y+l_h)D_Y\Big), \rv{36}\bar L_x\}$ where $\bar L_x:=\sup_{y\in Y}L_x(y)<\infty$ with $L_x(y):=\max\{\norm{\nabla_x f(x,y)}:\ \norm{x}_2\leq 1\}$ for $y\in Y$, and} invoking Lemma \ref{lem:equiv-stationary}  implies that $\| x_\epsilon - \Pcal_{\Mcal}(x_\epsilon)\| \leq \frac{3}{\rho} \epsilon$, $\|\sa{\grad_x} f(\Pcal_{\Mcal}(x_\epsilon), y_\epsilon)\| = \mathcal{O}(\epsilon)$, and ${\rm dist}(0, -\nabla_y f(\Pcal_{\Mcal}(x_\epsilon), y_\epsilon) + \partial h(y_\epsilon)) = \mathcal{O}(\epsilon)$. 
    %\nsa{State the extra condition on $\rho$ so that Lemma 2 applies.} \nhj{Right.}
    
    \rv{Next, we argue for the asymptotic stationarity. According to the proof of \cref{thm}, for all $t\geq 1$, we have $G_t^x\in\nabla_x\tilde f(x_t,y_t)+\partial \delta_X(x_t)$ and $G_t^y\in -\nabla_y\tilde f(x_t,y_t)+\partial h(y_t)$. Since $\norm{x_t}<C$ for all $t\geq 1$, we also get $G_t^x=\nabla_x\tilde f(x_t,y_t)$ for all $t\geq 1$. Furthermore, \cref{thm} guarantees that $\sum_{t=1}^{+\infty}\norm{G_t^x}^2+\bar\kappa\norm{G_t^y}^2=\cO(\bar \kappa)$, which implies that $\nabla_x\tilde f(x_t,y_t)\to 0$ and $G^y_t\to 0$ as $t\to\infty$. 
    %Consider $\{\bar x_t,\bar y_t\}_{t\geq 1}$ such that for any given $t\geq 1$, $(\bar x_t,\bar y_t)=(x_{T^*(t)}, y_{T^*(t)})$ for $T^*(t):=\argmin\{\norm{G_\ell^x}^2+\bar\kappa\norm{G_\ell^y}^2:\ \ell=1,\ldots,t\}$ for $t\geq 1$. Let $\bar G_t^y:=G^y_{T^*(t)}$. Thus, from the definition of $(\bar x_t,\bar y_t)$, we get $\norm{\nabla_x\tilde f(\bar x_t,\bar y_t)}=\cO\Big(\sqrt{\frac{\bar\kappa}{t}}\Big)$ and $\norm{\bar G_t^y}=\cO(1/\sqrt{t})$ for some $\bar G_t^y\in -\nabla_y\tilde f(\bar x_t,\bar y_t)+\partial h(\bar y_t)$ for all $t\geq 1$. 
    Moreover, for all $t\geq 1$, since ${\rm dist}(\bar x_t, \Mcal)\leq \frac{1}{2}$, Lemma~\ref{lem:equiv-stationary} implies that $\| x_t - \Pcal_{\Mcal}(x_t)\| \to 0$ and ${\grad_x} f(\Pcal_{\Mcal}( x_t), y_t) \to 0$. Since \sm{} sequence $\{(x_t,y_t)\}_{t\geq 0}\subset X\times Y$ is bounded, it has at least one limit point. Let $(x^*,y^*)$ be an arbitrary limit point of $\{(x_t, y_t)\}_{t\geq 1}$ and let $\{(x_{t_k},y_{t_k})\}_{k\geq 1}$ be a subsequence such that $(x_{t_k}, y_{t_k})\to(x^*,y^*)$ as $k\to\infty$. One can conclude that $x^*=\Pcal_{\Mcal}(x^*)\in\cM$ and ${\grad_x} f(x^*, y^*)=0$ since $\| x_{t_k} - \Pcal_{\Mcal}(x_{t_k})\|\to 0$ and ${\grad_x} f(\Pcal_{\Mcal}(x_{t_k}), y_{t_k}) \to 0$ as $k\to\infty$; furthermore, since $G_{t_k}^y\in -\nabla_y\tilde f(x_{t_k}, y_{t_k})+\partial h(y_{t_k})$ for all $k\geq 0$, $\nabla_y\tilde f(x_{t_k},y_{t_k})\to \nabla_y f(x^*,y^*)$ (this is due to $A(x^*)=x^*$ as $x^*\in\cM$) and $G_{t_k}^y\to 0$, it follows from \citep[Theorem 24.4]{rockafellar-1970a} that $0 \in - \nabla_y f(x^*, y^*) + \partial h(y^*)$.}

\end{proof}

\section{Proof of Theorem \ref{coro}}
\label{sec:proof-convergence}
\subsection{Proof of Lemma~\ref{lem:bounded-sequence}}
\label{sec:proof-bounded-sequence}
% Let $\phi^* = \min_{x \in \mathcal{X}} \max_y \tilde{f}(x,y)$. From the proof of Theorem 1, $\{V_t\}_{t \ge 0}$ is a non-increasing sequence, where
% \[
% V_t \triangleq \hat{f}(x_t)
% \]\\
\rv{Since we select $z_0\in X$ and $\{x_t\}\subset X$ by the construction, through induction one can argue that $\{z_t\}\subset X$.} Moreover, \rv{from \eqref{eq:v-descent} in the proof of \cref{thm},} we know that $\{V_t\}_{t \geq 0}$ is non-increasing sequence where $V_t \triangleq \rv{\hat{f}_r}(x_t, y_t; z_t) - 2\rv{\Psi_r}(y_t; z_t) + 2 P(z_t)$. We clearly have $\rv{\hat{f}_r}(x_t, y_t; z_t) - \rv{\Psi_r}(y_t; z_t) \geq 0$ and $P(z_t) - \rv{\Psi_r}(y_t; z_t) \geq 0$ \sa{from weak duality, and we also have $P(z_t) - \rv{\Phi^*} \geq 0$ since $\Phi^*=\min_{z\in\cX}P(z)$ as $P(\cdot)$ is a Moreau envelope of $\Phi(\cdot)$}; hence, $V_0-\rv{\Phi^*}\geq V_t-\rv{\Phi^*}\geq 0$ for $t\geq 0$. Therefore, we can conclude that
\begin{equation}
\label{eq:potential-bound}
 0 \leq P(z_t) - \rv{\Psi_r}(y_t; z_t) \leq V_t- \rv{\Phi^*} \leq V_0 - \rv{\Phi^*}, \quad \forall t \geq 0.
\end{equation}
\rv{For any $z\in \cX$}, recall that $\rv{\Psi_r}(y; z) \triangleq \min_{x \in X} \rv{\tilde{f}_r}(x, y) + \frac{{p}}{2} \|x - z\|^2$, and that $\rv{\tilde{f}_r}(x, \cdot)$ is $\mu$-concave over $\cY$ for all $x\in X$; thus, it follows that $\rv{\Psi_r(\cdot; z)}$ is $\mu$-concave. Indeed, for any $z\in\cX$ and $y\in \cY$,
\[
\rv{\Psi_r}(y; z) + \frac{\mu}{2} \|y\|^2 = \min_{x \in X} 
H(x,y;z), \quad H(x,y;z):=\rv{\tilde{f}_r}(x, y) + \frac{\mu}{2} \|y\|^2 + \frac{\hj{p}}{2} \|x - z\|^2,
\]
and $H(x,\cdot;z)$ is concave for all $x\in X$ and $z\in \cX$. Since the pointwise infimum of concave functions is concave, $\rv{\Psi_r}(y; z) + \frac{\mu}{2} \|y\|^2$ is concave in $y$; hence, $\rv{\Psi_r}(\cdot; z)$ is $\mu$-concave for all $z\in \cX$. Note that for any $t \ge 0$, we have 
\begin{equation}
\label{eq:dual-sol-subopt}
\begin{aligned}
&P(z_t) = \max_{y\in\cY} \rv{\Psi_r}(y; z_t) = \rv{\Psi_r}(y^*(z_t); z_t)
\ \Rightarrow\ \frac{\mu}{2} \| y_t - y^*(z_t) \|^2 
\leq P(z_t) - \rv{\Psi_r}(y_t; z_t) 
\leq V_0 - \rv{\Phi^*}.
\end{aligned}
\end{equation}
Moreover, it follows from
$
x^*(z) \triangleq \arg \min_{x \in X} \Phi(x; z) 
= \arg \min_{x \in X} \max_{y\in\cY} \rv{\hat{f}_r}(x, y;z)$ and the definition of \rv{$r^*(\cdot)$} that $x^*(z)=\arg\min_{x\in X}\rv{\hat{f}_r}(x, r^*(x);z)$, which leads to
$P(z) = \rv{\Phi(x^*(z); z)} = \rv{\hat{f}_r}(x^*(z), r^*(x^*(z));z)$. Finally, since $(x^*(z), y^*(z))$  is the unique saddle point of  $\rv{\hat{f}_r}(x, y; z)$, we can conclude that $y^*(z) = r^*(x^*(z))$ for all $z\in \cX$. 

It is known that $r^*(\cdot)$ is $\kappa_{yx}$-Lipschitz with $\sa{\kappa_{yx} := \frac{l_{yx}}{\mu}}$, e.g., see \citep[Lemma A.4]{zhang2024agda+}. Consequently, we have
\[
\|y^*(z_t) - y^*(z_0)\| = \|r^*(x^*(z_t)) - r^*(x^*(z_0))\| \leq \sa{\kappa_{yx}} \|x^*(z_t) - x^*(z_0)\|.
\]
Furthermore, using \cref{lem:lip}, we get 
$
\|x^*(z) - x^*(z')\| \leq \frac{p}{p -l} \|z - z'\|$, which implies that $\|y^*(z_t) - y^*(z_0)\| \leq \sa{\kappa_{yx}}\frac{p}{p - l} \|z_t - z_0\|$. Combining this bound with \eqref{eq:dual-sol-subopt} and using triangular inequality leads to $\norm{y_t-y^*(z_0)}\leq \sqrt{\frac{2}{\mu}(V_0-\rv{\Phi^*})}+2\kappa_{yx}\norm{z_t-z_0}$ for all $t\geq 0$. 
%\rv{Now, suppose we start with some arbitrary $z_0\in\cM$ and $x_0=x^*(z_0)$, $y_0=y^*(z_0)$.} By definition of $x^*(z_0)$ and $y^*(z_0)$, we get $V_0=P(z_0)\leq \max_{y\in \cY} \rv{\hat{f}_r}(z_0, y ; z_0) =\max_{y\in \cY} \rv{f_r}(z_0, y) =\rv{F}(z_0)$, where we used the fact that $z_0\in\cM$ implies $A(z_0)=z_0$ and $c(z_0)=0$.
Since \sm{} is initialized from $(x_0,y_0,z_0)$ as in Theorem~\ref{thm:mu-pl}, according to the proof of Theorem~\ref{thm:mu-pl}, one has $V_0\leq F(z_0)$, which completes the proof of Lemma~\ref{lem:bounded-sequence}.
\subsection{Proof of Lemma~\ref{lem:delta}}
\label{sec:proof-delta}
    \sa{The first result directly follows from 
    %\citep[Lemma 4.4]{boct2023alternating}
    \citep[Lemma E.4]{zhang2024agda+}. Moreover, \cref{thm} implies that $\sum_{t=0}^{\infty}\norm{x_{t+1}-x_t}^2=\cO(\kappa\tau_1^2)$, and $\tau_2<\frac{1}{l}\leq\frac{1}{\mu}$ implies that $1-\tau_2\mu/2\in (0,1)$, which implies that $\sum_{t=0}^{\infty}\delta_t<\infty$; hence, $\delta_t\to 0$.}

\rv{Now we are ready to prove Theorem \ref{coro}.}
\begin{proof}[Proof of Theorem \ref{coro}]
    \sa{\cref{lem:bounded-sequence} implies that $\{x_t,y_t,z_t\}$ stays in a bounded set; therefore, it has at least one limit point $(x^*,y^*,z^*)$.} Consider the potential function values at $(x_t,y_t,z_t)$, i.e., $V_t= \rv{\hat{f}_r}(x_t, y_t; z_t) - 2\rv{\Psi_r}(y_t;z_t) + 2P(z_t)$ for $t\geq 0$. It follows from \eqref{eq:potential-bound} that for all $t\geq 0$, we have $\rv{\hat{f}_r}(x_t,y_t;z_t)  \leq V_t\leq V_0$ where we used $P(z_t)\geq \rv{\Psi_r(y_t;z_t)}$. 
    
    %Moreover, noting that $r^*(\cdot)$ is $\kappa_{yx}$-Lipschitz with respect to $x$ and $\delta_t \leq \|y_t - r^*(x_t)\|^2$ is summable, thus $\{y_t\}$ lies in a compact set $\tilde{Y}$.
    \rv{Let $\bar Y\subset\cY$ denote a compact set for which $\{y_t\}_{t\geq 0}\subset \bar Y$. Define $\bar l_y:= \max_{x\in X, y \in \bar{Y}} \| \nabla_y f(A(x),y) \|$ and let $l_h$ be the Lipschitz constant of $h$ over $\dom h\cap \bar{Y}$ --Assumption~\ref{assum:Y}.\textbf{(i)} implies that such $l_h$ exists.} 
    Then, \rv{from the same arguments we used for \eqref{eq:low-boun-lip},} we get
    \begin{equation} \label{eq:low-boun-lip-2}
        \Phi(x_t)+\frac{p}{2}\norm{x_t-z_t}^2-{(\bar l_y + l_h)}\norm{y_t-r^*(x_t)}\leq \hat{f}_r(x_t,y_t;z_t). 
    \end{equation}
    %where we used $\Phi(x_t)+\frac{p}{2}\norm{x_t-z_t}^2=\max_{y\in\cY}\hat f(x_t,y;z_t)=\max_{y\in\cY}\hat f(x_t,r^*(x_t);z_t)$ and $\nabla_y \tilde f(x_t,r^*(x_t))=0$ due to optimality. 
    Therefore, for all $t\geq 0$, it holds that $\Phi(x_t)\leq V_0+{(\bar l_y + l_h)}\sqrt{\delta_t}$. %$\leq V_0+\frac{L_{yy}}{2}\bar\delta^2$ where $\bar\delta^2\in\R_+$ such that $\delta_t^2\leq \bar\delta^2$ for $t\geq 0$ 
    Thus, from \cref{rem:Phi-connection}, we get $\rv{F}(A(x_t))+\frac{\rho}{4}\norm{c(x_t)}^2\leq V_0+{(\bar l_y + l_h)}\sqrt{\delta_t}$ for $t\geq 0$. Note that %$\min_{x\in \cX}\Phi(A(x))\geq\Phi^*:=\min_{x\in \cX}\Phi(x)$; therefore, 
    $\rv{F}(A(x_t))\geq %\Phi^*
    \rv{\bar F}$ and we get
    \begin{align}
        \norm{c(x_t)}^2\leq \frac{4}{\rho}\left(V_0-%\Phi^*
        \rv{\bar F}+(\bar l_y + l_h)\sqrt{\delta_t}\right),\quad\forall~t\geq 0.
    \end{align}
    %Since we initialize $x_0=x^*(z_0)$ and $y_0=y^*(z_0)$ for some arbitrary $z_0\in\cM$, from the discussion in the proof of \cref{lem:bounded-sequence}, 
    \rv{Since \sm{} is initialized from $(x_0,y_0,z_0)$ as in Theorem~\ref{thm:mu-pl}, we have $V_0=P(z_0)\leq \rv{F}(z_0)$, which is independent of the parameter $\rho>0$.} Suppose we fix $\rho>32(\rv{F}(z_0)-%\Phi^*
    \rv{\bar F})$. Recall that $\delta_t\to 0$ as $t\to\infty$; therefore, there exists $\bar T_\rho\in\mathbb{Z}_+$ such that $\delta_t\leq \frac{\rho^2}{1024(\bar l_y+l_h)^2}$ for all $t\geq \bar T_\rho$. Thus, for all $t\geq \bar T_\rho$, we can conclude that $\norm{c(x_t)}\leq \frac{1}{2}$. 
    
     For any fixed $t\geq \bar T_\rho$, let $x_t = u_ts_tv_t^\top$ be the compact singular value decomposition of $x_t$. Then, $\norm{c(x_t)}=\|x_t^\top x_t - I_r\| = \| v_t (s_t^2 - I_r) v_t^\top \| = \| s_t^2 - I_r\|$. Note that \sa{whenever ${\rm dist}(x_t, \Mcal)\leq\frac{1}{2}$, projection on to $\cM$ is well-defined\hj{, i.e., single-valued and Lipschitz continuous,} and} ${\rm dist}(x_t, \Mcal) = \|x_t - \Pcal_{\Mcal}(x_t)\| = \|u_ts_tv_t^\top - u_tv_t^\top\| = \|s_t - I_r\|$ and $\|s_t^2-I_r\| \geq\|s_t - I_r\|$, i.e., $\norm{c(x_t)}\geq {\rm dist}(x_t, \Mcal)$. Hence, we can conclude that for any $\rho\geq 32(\rv{F}(z_0)-%\Phi^*
     \rv{\bar F})$, we have  ${\rm dist}(x_t, \Mcal)\leq \frac{1}{2}$ for all $t\geq \bar T_\rho$.  
    On the other hand, since $\cM$ is compact, for any $C>0$ such that $C>\frac{1}{2}+\sup_{x\in\cM}\norm{x}$, it must hold that $\norm{x_t}<C$ for all $t\geq \bar T_\rho$. Thus, \cref{thm} implies that $\sum_{t=\bar T_\rho}^\infty\norm{\nabla_x\tilde f(x_t,y_t)}^2+\mod{\bar \kappa{\rm dist}( 0, -\nabla_y\tilde f(x_t,y_t) + \partial h(y_t))^2} =\cO(\bar \kappa)$, and we have $\min\{\norm{\nabla_x\tilde f(x_t,y_t)}^2+\mod{\bar \kappa{\rm dist}( 0, -\nabla_y\tilde f(x_t,y_t) + \partial h(y_t))^2}:\ t=\bar T_\rho,\ldots,\bar T_\rho+T-1\}=\cO(\bar \kappa/T)$ for all $T\geq 1$. 
    Therefore, for any $\epsilon>0$ given, the algorithm can generate $(x_\epsilon,y_\epsilon)$ such that $\norm{\nabla_x \tilde{f}(x_\epsilon,y_\epsilon)}\leq \epsilon$, $\mod{\bar \kappa{\rm dist}( 0, -\nabla_y\tilde f(x_\epsilon,y_\epsilon) + \partial h(y_\epsilon)) \leq \epsilon}$, and  ${\rm dist}(x_\epsilon, \Mcal)\leq \frac{1}{2}$ within $\bar T_\rho+\cO(\frac{1}{\epsilon^2})$ iterations of \sm{}. \sa{Moreover, since $%\Phi^*
    \rv{\bar F}\leq\rv{\Phi^*}$ and $\{z_t\}\subset X$, it follows from \cref{lem:bounded-sequence} that $y_t\in Y_c:=\{y\in\cY: \norm{y-y_0}\leq \sqrt{\frac{2}{\mu}(\rv{F}(z_0)-%\Phi^*
    \rv{\bar F})}+4\kappa_{yx}C\}$ for all $t\geq 0$ and note that the bound does not depend on $\rho$ parameter \rv{since $l_{yx}$ and so is $\kappa_{yx}=l_{yx}/\mu$ independent of $\rho$.} Hence, setting $\rho\geq \max\{32(\rv{F}(z_0)-%\Phi^*
    \rv{\bar F}), \rv{36}\bar l_x\}$ where
    %\nsa{say $\kappa_{yx}$ does not depend on $\rho$!} 
    $\bar l_x:=\sup_{y\in Y_c}l_x(y)<\infty$, and} invoking Lemma \ref{lem:equiv-stationary} implies that $\| x_\epsilon - \Pcal_{\Mcal}(x_\epsilon)\| \leq \frac{3}{\rho} \epsilon$, $\|\sa{\grad_x} f(\Pcal_{\Mcal}(x_\epsilon), y_\epsilon)\| = \mathcal{O}(\epsilon)$, and $\mod{{\rm dist}( 0, -\nabla_y\tilde f(x_\epsilon,y_\epsilon) + \partial h(y_\epsilon))} = \mathcal{O}(\epsilon)$. Moreover, for any limit point $(x^*,y^*)$ of the \sm{} sequence $(x_t,y_t)$, it holds that $x^*\in\cM$, $\grad_x f(x^*, y^*) =0$ and $0 \in - \nabla_y f(x^*, y^*) + \partial h(y^*)$.
\end{proof}

\section{Proof of Theorem \ref{thm:concave}}
\label{sec:proof-concave}
%\begin{proof}[Proof of Theorem \ref{thm:concave}]
Given arbitrary $x_0 \in \Mcal$ and $y_0\in Y$. %define $y_{-1} := y_0$. 
We use \rv{a slightly modified version\footnote{\rv{The potential function we use in this paper also involves the closed convex function $h(\cdot)$.}} of} the potential function defined in \citep[Lemma 3.4]{xu2023unified}, \rv{i.e., $V_0:=f_r(x_0,y_0)$ and for all $t\geq 0$, }
\begin{align}
\label{eq:V-function}
V_{t+1} & :=\rv{\tilde{f}_r}\left(x_{t+1}, y_{t+1}\right) + \frac{1}{2 \tau_2}\left( \frac{16}{\tau_2 \theta_{t+1}} -15 \right) \left\|y_{t+1}-y_t\right\|^2 + \left( \frac{8}{\tau_2}\left(1-\frac{\theta_t}{\theta_{t+1}}\right)- \frac{\theta_t}{2}\right) \left\|y_{t+1}\right\|^2.
\end{align}
\rv{For $t\geq 0$, recall that} $y_{t+1}
= \prox_{\tau_2 h} \Big(
    y_t + \tau_2 \nabla_y \tilde f (x_{t+1},y_t) - \tau_2 \theta_t y_t \Big)$; hence,  it holds that
$0 \in \frac{1}{\tau_2}(y_{t+1}-y_t)
   - (\nabla_y \tilde f(x_{t+1},y_t) - \theta_t y_t)
   + \partial h(y_{t+1})$, which also implies 
   \[ \rv{\cG_{t+1}^y}:=\frac{1}{\tau_2}(y_{t+1}-y_{t})+
   \nabla_y \tilde f(x_{t+1},y_{t+1}) - \nabla_y \tilde f(x_{t+1},y_t) + \theta_t y_t \in  \nabla_y \tilde f(x_{t+1},y_{t+1})
   - \partial h(y_{t+1}). \]
Then, $D^y_{t+1}:=  {\rm dist}\Big(0, - \nabla_y \tilde f(x_{t+1},y_{t+1})
   + \partial h(y_{t+1})\Big) \leq \rv{\norm{\cG_{t+1}^y}}\leq (\frac{1}{\tau_2}+l) \| y_{t+1}-y_{t}\| + \theta_t \|y_t\|$ for all $t\geq 0$. 
Similarly, for $t\geq 0$, since $x_{t+1} = \Pcal_X\Big(x_t - \tau_{1,t} \nabla_x \tilde f(x_t,y_{t})\Big)$, it holds
$ 0 \in \frac{1}{\tau_{1,t}} (x_{t+1} - x_t) + \nabla_x \tilde f(x_t,y_t) + \partial \delta_X(x_{t+1})$; therefore, we get
\[ \rv{\cG_{t+1}^x}:=\frac{1}{\tau_{1,t}}(x_{t}- x_{t+1}) + \nabla_x \tilde f(x_{t+1}, y_{t+1}) - \nabla_x \tilde f(x_{t}, y_{t}) \in \nabla_x \tilde f(x_{t+1}, y_{t+1}) + \partial \delta_X(x_{t+1}).   \]
Then, $D_{t+1}^x := {\rm dist}\Big(0, \nabla_x \tilde f(x_{t+1}, y_{t+1}) + \partial \delta_X(x_{t+1})\Big)\leq \norm{\cG_{t+1}^x}\leq (\frac{1}{\tau_{1,t}}+l)\|x_{t+1} - x_t\| + l \|y_{t+1} - y_t\|$ for all $t\geq 0$. Thus, combining the two inequality above, we get for all $t\geq 0$ that
\begin{equation}
\label{eq:Dt}
    \begin{aligned}
    D_{t+1}^2 
    &:= (D_{t+1}^x)^2 + (D_{t+1}^y)^2 \leq \rv{\norm{\cG_{t+1}^x}^2 + \norm{\cG_{t+1}^y}^2}\\ 
    & \leq
    2\Big(\frac{1}{\tau_{1,t}} + l\Big)^2 \|x_{t+1} - x_t\|^2 + 2\Big(l^2 + \Big(\frac{1}{\tau_2} + l\Big)^2 \Big)\|y_{t+1} - y_t\|^2 + 2\theta_t^2 \|y_t\|^2.
    \end{aligned} 
\end{equation}
\rv{Given an arbitrary $\tilde b >\max\{ \frac{1}{16}\frac{19^2}{20^2}\cdot (\frac{2}{\tau_2 l}-1),~2\}$, and for all $t\geq 0$ let $\alpha_t := \frac{8(\tilde b - 2)l^2}{\tau_2 \theta_t^2}$ and $\beta_t:= \tau_2 l^2 + \frac{16 l^2}{\tau_2 \theta_t^2} \tilde b - l$. Note that since $\theta_t = \frac{19}{20}\cdot\frac{1}{\tau_2}\cdot
    \rv{\frac{1}{(t+1)^{1/4}}}$, the definition of $\tilde b$ implies that $\beta_t>l$ for $t\geq 0$.
Using these definitions, we get the following identities: 
\begin{equation}
\label{eq:alpha-identity}
    \tau_{1,t}=(\beta_t+l/2)^{-1},\quad (\beta_t+l)/2=\alpha_t+\tau_2 l^2/2+\frac{16 l^2}{\tau_2\theta_t^2},\quad\forall~t\geq 0.
\end{equation} 
Moreover, let  
$d_1 := \frac{8 \tilde{b}^2}{(\tilde b  -2 )^2} + \frac{1}{32}\frac{19^4}{20^4}\cdot\frac{(\tau_2 - \frac{1}{2l})^2 + 1 }{ \tau_2^2 (\tilde b -2)^2 l^2}$.} Then, it follows from \cite[Equation (3.52)]{xu2023unified} that
\begin{equation}
\label{eq:d1}
    \frac{2(\frac{1}{\tau_{1,t}} + l)^2}{\alpha_t^2} \leq \frac{(2\beta_t + l)^2 + 4l^2}{\alpha_t^2} \leq 4 d_1, \quad\rv{\forall~t\geq 0}.
\end{equation}
\rv{Therefore, \eqref{eq:Dt} and \eqref{eq:d1} together imply that}
\[ D_{t+1}^2 \leq \rv{\norm{\cG_{t+1}^x}^2 + \norm{\cG_{t+1}^y}^2}\leq 4 d_1 \alpha_t^2 \|x_{t+1} - x_t\|^2 + \Big(\frac{4}{\tau_2^2} + 6l^2\Big)\|y_{t+1} - y_t\|^2 + 2\theta_t^2 \|y_t\|^2, \quad\rv{\forall~t\geq 0}. \]	
In addition, from \cite[Equation (3.48)]{xu2023unified}, \rv{for all $t\geq 1$,} we have 
\begin{equation}
\label{eq:potential-recursion-1}
\begin{aligned}
    &\alpha_t \|x_{t+1} - x_t\|^2 + \frac{9}{10\tau_2} \|y_{t+1} - y_t\|^2 \leq V_t - V_{t+1} + B_t,\\ 
    &B_t:=\frac{8}{\tau_2}\Big(\frac{\theta_{t-1}}{\theta_t} - \frac{\theta_t}{\theta_{t+1}} \Big) \|y_{t+1}\|^2 + \frac{\rv{\theta_{t-1}} - \theta_t}{2} \|y_{t+1}\|^2.
    \end{aligned}
\end{equation}
\rv{Next, we upper bound $V_1-V_0$; the analysis in \cite{xu2023unified} does not require this bound; however, to establish asymptotic stationarity we need this bound as $V_0=f_r(x_0,y_0)$ does not depend on $\rho$ whenever $x_0\in\cM$ and this observation is essential for our analysis. Indeed, we first bound $\tilde f(x_1,y_1)-\tilde f(x_0,y_0)$ using the following two inequalities:
    \begin{align}
        &\tilde f(x_1,y_0)-\tilde f(x_0,y_0) \leq \Big(\frac{l}{2}-\frac{1}{\tau_{1,0}}\Big)\norm{x_2-x_1}^2\leq -\frac{\beta_0+l}{2} \norm{x_1-x_0}^2, \label{eq:x0-step}\\
        &\tilde f_r(x_1,y_1) - \tilde f_r(x_1,y_0) \leq \fprod{\nabla_y \tilde f(x_1,y_0)-g_0,~y_1-y_0}\leq \frac{1}{2\tau_2}\norm{y_1-y_0}^2+\tau_2(l^2_y(y_0)+l_h^2),
        \label{eq:y0-step}
    \end{align}
    where $g_0\in\partial h(y_0)$, $\bar l_y(y_0):=\max_{x \in X} \|\nabla_y f(A(x),y_0) \|<\infty$, and $l_h$ is the Lipschitz constant of $h$. The inequality in \eqref{eq:x0-step} follows from the similar arguments we used in~\eqref{eq:p-descent-1} and \eqref{eq:p-descent-2}; furthermore, \eqref{eq:y0-step} follows from concavity of $f_r(x_1,\cdot)$ and using Young's inequality. Therefore, summing \eqref{eq:x0-step} and \eqref{eq:y0-step}, we get
    $$\tilde f_r(x_1,y_1)-\tilde f_r(x_0,y_0)\leq -\frac{\beta_0+l}{2} \norm{x_1-x_0}^2+\frac{1}{2\tau_2}\norm{y_1-y_0}^2+\tau_2(l^2_y(y_0)+l_h^2);$$
    hence, using the definition of $V_1$ given in \eqref{eq:V-function} together with $V_0:=\tilde f_r(x_0,y_0)=\tilde f_r(x_0,y_0)$ for $x_0\in\cM$, it follows that
    $$ 
    V_1-V_0\leq -\frac{\beta_0+l}{2} \norm{x_1-x_0}^2+\Big(\frac{1}{2\tau_2}+\frac{8}{\tau_2^2\theta_1}-\frac{15}{2\tau_2}\Big)\norm{y_1-y_0}^2+\left( \frac{8}{\tau_2}\left(1-\frac{\theta_0}{\theta_{1}}\right)- \frac{\theta_0}{2}\right) \left\|y_{1}\right\|^2+\tau_2(l^2_y(y_0)+l_h^2).
    $$
    Note that from \eqref{eq:alpha-identity}, we have $\frac{\beta_0+l}{2}\geq \alpha_0$; moreover, since $\frac{8}{\tau_2^2\theta_1}\leq \frac{13}{10\tau_2}$, we also have $\frac{9}{10\tau_2}+\frac{1}{2\tau_2}+\frac{8}{\tau_2^2\theta_1}-\frac{15}{2\tau_2}\leq 0$, and 
    finally we also observe that $1-\frac{\theta_0}{\theta_1}<0$; therefore, after using these bounds in the above inequality and dropping the non-positive terms from the right hand side, we get
    \begin{equation}
    \label{eq:potential-recursion-2}
    \alpha_0 \|x_{1} - x_0\|^2 + \frac{9}{10\tau_2} \|y_{1} - y_0\|^2 \leq V_0 - V_{1} +B_0,\quad\mbox{where}\quad B_0:=\tau_2(l^2_y(y_0)+l_h^2).  
    \end{equation}}%

	\rv{For $t\geq 0$, we have $\rv{\norm{\cG_{t+1}^x}^2 + \norm{\cG_{t+1}^y}^2} \leq A_t (V_t-V_{t+1}+B_t) + 2 \theta_t^2 \|y_t\|^2$ for $A_t:= \max\{ 4d_1 \alpha_t, \frac{10}{9} (\frac{4}{\tau_2} + 6l^2 \tau_2)\}$. Therefore, we can conclude that
    \begin{equation}
    \label{eq:G-bound}
        \sum_{t=0}^{T-1}\frac{1}{A_t}\Big(\norm{\cG_{t+1}^x}^2 + \norm{\cG_{t+1}^y}^2\Big)\leq \sum_{t=0}^{T-1} \Big(V_t-V_{t+1}+B_t + \frac{2\theta_t^2 }{A_t} \|y_t\|^2 \Big),\qquad \forall~T\geq 1.
    \end{equation}}% 
    Note that from \eqref{eq:V-function}, we can lower bound the potential value at time $t+1$ as follows:
    \begin{equation}
        V_{t+1}\geq \tilde f_r(x_{t+1},y_{t+1}) -\frac{15}{2\tau_2}D_Y^2+\frac{8}{\tau_2}(1-\frac{\theta_0}{\theta_1})R_Y^2-\frac{\theta_0}{2}R_Y^2,\qquad \forall~t\geq 0,
    \end{equation}
    where $D_Y:=\max_{y,y' \in Y} \|y-y'\|$ and $R_Y:=\max_{y \in Y} \|y\|$; therefore, for $\underline{f} := \inf_{x \in X, y \in Y} f_r(A(x),y)$, we have $V_{t+1}\geq \underline{V}:=\underline{f} -\frac{15}{2\tau_2}D_Y^2+\frac{8}{\tau_2}(1-\frac{\theta_0}{\theta_1})R_Y^2-\frac{\theta_0}{2}R_Y^2$. Thus, for any $T\geq 1$, we have
    \begin{equation}
        \begin{aligned}
            \sum_{t=0}^{T-1}(V_t-V_{t+1})=V_0-V_{T}\leq V_0-\underline{V}; 
        \end{aligned}
    \end{equation}
    moreover, using the definition of $\{B_t\}_{t\geq 0}$ given in \eqref{eq:potential-recursion-1} and \eqref{eq:potential-recursion-2}, it follows that for any $T\geq 1$, we also have
    \begin{equation}
    \label{eq:B-sum}
        \begin{aligned}
        \sum_{t=0}^{T-1}B_t
        &=\frac{8}{\tau_2}\Big(\frac{\theta_0}{\theta_1}-\frac{\theta_{T-1}}{\theta_{T}}\Big)R_Y^2+\frac{\theta_0-\theta_{T-1}}{2}R_Y^2+\tau_2(l^2_y(y_0)+l_h^2),\\
        &\leq \Big(\frac{8}{\tau_2}\frac{\theta_0}{\theta_1}+\frac{\theta_0}{2}\Big)R_Y^2+\tau_2(l^2_y(y_0)+l_h^2).
        \end{aligned}
    \end{equation}
    Therefore, we can conclude that for all $T\geq 1$,
    \begin{equation}
    \label{eq:V-sum}
        \begin{aligned}
        \sum_{t=0}^{T-1}(V_t-V_{t+1}+B_t) 
        &\leq f_r(x_0,y_0)-\underline{f}+\frac{15}{2\tau_2}D_Y^2 +\Big(\frac{16}{\tau_2}\frac{\theta_0}{\theta_1}-\frac{8}{\tau_2}+\theta_0\Big)R_Y^2+\tau_2(l^2_y(y_0)+l_h^2)\\
        &\leq f_r(x_0,y_0)-\underline{f} +\frac{15}{2\tau_2}D_Y^2
        +\left(8 \cdot \frac{3}{2} + 1 \right) \frac{R_Y^2}{\tau_2}+\tau_2(l^2_y(y_0)+l_h^2):=d_2,
    \end{aligned}
    \end{equation}
    where in the last inequality we used $2\frac{\theta_0}{\theta_1}-1=2\cdot 2^{1/4}-1\leq \frac{3}{2}$. Note that $d_2$ is independent of the penalty parameter $\rho$.
    %Note from \cite[Equations (3.54) and (3.56)]{xu2023unified} that $\sum_{t=1}^{T_\epsilon -1} B_t \leq V_1 - \underline{V} + \left(8 \cdot 2^{1/4} + \frac{19}{40} \right) \frac{R_Y^2}{\tau_2} =: d_2$, where $\underline{V} = \inf_{x \in X, y \in Y} \tilde f_r(x,y)$ and $R_Y:=\max_{y \in Y} \|y\|$. 
	
	\rv{Let $\bar{A} := \max\{ 4d_1,~ \frac{10}{9} (\frac{4}{\tau_2} + 6l^2 \tau_2) \frac{1}{\alpha_0}\}$. Then, ${A_t} \leq {\bar{A}} \cdot{\alpha_t}$ for all $t\geq 0$; hence, \eqref{eq:G-bound} implies that for all $T\geq 1$,
	\[ \sum_{t=0}^{T -1} \frac{\norm{\cG_{t+1}^x}^2 + \norm{\cG_{t+1}^y}^2}{\alpha_t} \leq \bar{A} d_2 + 2\bar{A}R_Y^2\cdot \sum_{t=0}^{T -1} \frac{\theta_t^2}{A_t},\quad\mbox{and}\quad \sum_{t=0}^{T -1} \frac{\theta_t^2}{A_t}\leq \frac{1}{4d_1}\sum_{t=0}^{T -1}\frac{\theta_t^2}{\alpha_t}=\frac{\tau_2}{32d_1(\tilde b-2)l^2}\sum_{t=0}^{T -1}\theta_t^4. \] 
    Thus, $\sum_{t=0}^{T -1} \frac{\theta_t^2}{A_t}=\cO(1)\sum_{t=0}^{T -1}\frac{1}{t+1}=\cO(\log(T)+1)$.} Moreover, noting that $\sum_{t=0}^{T -1}  \frac{1}{\alpha_t} = \Omega(\sqrt{T})$, we have
    \begin{equation}
        \min_{t=0,\ldots,T-1}\{\norm{\cG_{t+1}^x}^2 + \norm{\cG_{t+1}^y}^2\}=\cO\Big(\log(T)/\sqrt{T}\Big),\qquad\forall~T\geq 1.
    \end{equation}
    Moreover, for any given $\epsilon>0$, define $T_\epsilon:= \inf\{t \geq 0: \rv{D_{t+1}} \leq \epsilon\}$, for which it holds that
	\begin{equation}
    \label{eq:frac-At-sum}
	    \epsilon^2\ \Omega(\sqrt{T_\epsilon}) = \sum_{t=0}^{T_\epsilon -1} \frac{\epsilon^2}{\alpha_t} \leq \sum_{t=0}^{T_\epsilon -1} \frac{1}{\alpha_t} D_{t+1}^2 \leq \sum_{t=0}^{T_\epsilon -1} \frac{\norm{\cG_{t+1}^x}^2 + \norm{\cG_{t+1}^y}^2}{\alpha_t} = \mathcal{O}(1 + \log T_\epsilon);
	\end{equation}
	therefore, within $T_\epsilon = \mathcal{O} (\frac{1}{\epsilon^4}\log^2(\frac{1}{\epsilon}))$ iterations, \sm{} will return a pair $(x_\epsilon, y_\epsilon)$ satisfying 
    \be \label{eq:eps-station-mere-concave} {\rm dist}^2\Big(0, \nabla_x \tilde f(x_\epsilon, y_\epsilon) + \partial \delta_X(x_\epsilon)\Big) + {\rm dist}^2\Big(0, -\nabla_y \tilde f(x_\epsilon, y_\epsilon) + \partial h(y_\epsilon)\Big) \leq \epsilon^2. \ee 
    %within $\tilde \cO(\frac{1}{\epsilon^4})$ iterations.  

    To establish the stationarity of $(x_\epsilon,y_\epsilon)$ for the original problem \eqref{prob}, we need to argue that the norm constraint $x\in X$ is never active. 
    % First, from \cite[Equation (3.43)]{xu2023unified}, we have
    % $V_1 - V_0 \leq -4 \alpha_0 C^2 + \frac{19}{20 \tau_2} R_Y^2 - \frac{36}{10\tau_2} R_Y^2 + \frac{12}{\tau_2} R_Y^2 =: W_0$. Then, $V_1 \leq V_0  + W_0$ and $W_0$ is independent of $\rho$. 
    % Moreover, it can be easily checked from \citep[Equation \rv{(3.48)}]{xu2023unified} that $V_{t+1} \leq V_1 + W_1$, where $W_1>0$ is a finite constant that depends on $C$ and $\rv{R_Y}$; but, \rv{$W_2$} is independent of $\rho$.  Then, $V_0 \leq f_r(x_0, y_0) + \frac{8}{\tau_2} \rv{R_Y}=:W_2$, and $W_2$ is also independent of $\rho$.
    Indeed, from \eqref{eq:potential-recursion-1}, \eqref{eq:potential-recursion-2} and \eqref{eq:B-sum}, it follows that $V_{T}\leq V_0+\sum_{t=0}^{T-1}B_t\leq f_r(x_0,y_0)+\Big(\frac{8}{\tau_2}\frac{\theta_0}{\theta_1}+\frac{\theta_0}{2}\Big)R_Y^2+\tau_2(l^2_y(y_0)+l_h^2)$.
    Thus, for $t\geq 0$, \rv{since $\{\theta_t/\theta_{t+1}\}_{t\geq 0}$ is decreasing
    %\nsa{Note $\theta_t/\theta_{t+1}\leq 2$} 
    and $\theta_t\leq \frac{1}{\tau_2}$, \eqref{eq:V-function} implies that}
\begin{equation}
    \begin{aligned}
    \rv{\tilde{f}_r(x_{t+1}, y_{t+1})} 
    &\leq \rv{V_{t+1}} + \frac{15}{2 \tau_2} \left\|y_{t+1}-y_t\right\|^2 + \Big(\rv{\frac{8}{\tau_2}\frac{\theta_t}{\theta_{t+1}}+}\frac{\theta_t}{2}\Big)\left\|y_{t+1}\right\|^2 \\
    &\leq f_r(x_0,y_0) + \frac{15}{2\tau_2}D_Y^2+\Big(\frac{16}{\tau_2}\frac{\theta_0}{\theta_1}+{\theta_0}\Big)R_Y^2+\tau_2(l^2_y(y_0)+l_h^2):=\rv{\bar f}.
    \end{aligned}
\end{equation} 

\rv{Using a similar argument that we adopted for deriving \eqref{eq:low-boun-lip}, for any $y^*_t\in\argmax_{y\in\cY}\rv{\tilde{f}_r(x_t,y)}$,} it holds that
\[ \Phi(x_t):= \argmin_{y \in \rv{\cY}} \rv{\tilde{f}_r(x_t,y)} \leq \rv{\tilde{f}_r(x_t, y_t)} + (\bar l_y + l_h)\|y_t - \rv{y^*_t}\| \leq \bar f + (\bar l_y+l_h) D_Y:=\rv{\bar\Phi}, 
\]
where $\bar l_y:=\max_{x \in X, y \in Y} \|\nabla_y f(A(x),y) \|<\infty$; therefore,
since $\Phi(x_t) = \rv{F}(A(x_t)) + \frac{\rho}{4}\|c(x)\|^2$ and $\min_{x \in %\cX
\rv{X}} \rv{F}(A(x)) %\geq \min_{x \in \cX} \Phi(x) = \Phi^*
\rv{=\bar F>-\infty}$, we have
\[ \|c(x_t)\|^2 \leq \frac{\rv{4}}{\rho}\Big(\bar\Phi-\rv{\bar F}\Big).  \]
Then, for $\rho > \rv{16}(\bar\Phi
-\rv{\bar F})=16(\bar f + (\bar l_y+l_h) D_Y-\bar F)$, we have $\|c(x_t)\| \leq \frac{1}{2}$, which further implies $\|x_t - \Pcal_{\Mcal}(x_t)\| \leq \frac{1}{2}$. \rv{Thus, whenever $C>0$ is sufficiently large,
%the norm constraint $X$ is inactive for all $x_t$ by choosing a proper large $C$. 
$x_t\in {\rm int}(X)$ for all $t\geq 0$.} Then, \eqref{eq:eps-station-mere-concave} reads
\[ \| \nabla_x \tilde f(x_\epsilon, y_\epsilon)\|^2 + {\rm dist}^2\Big(0, -\nabla_y \tilde f(x_\epsilon, y_\epsilon) + \partial h(y_\epsilon)\Big)  \leq \epsilon^2.  \]
Thus, we can conclude that $(\xe,\ye)$ is indeed $\cO(\epsilon)$-stationary point of the NCMC minimax problem in~\eqref{prob} by invoking Lemma \ref{lem:equiv-stationary} for $\rho>0$ sufficiently large, i.e., $\rho \geq \max\{ 16(\rv{\bar\Phi}-\rv{\bar F}), \rv{36} \max_{y \in Y} L_x(y) \}$, where $\rv{L_x(y)}:= \max\{\|\nabla_x f(x,y)\|:\ \|x\|_2 \leq 
    \rv{1}\}$. Moreover, define $\{(\bar x_t,\bar y_t)\}_{t\geq 0}$ such that $(\bar x_t,\bar y_t)=(x_{T(t)},y_{T(t)})$ where $T(t):=\argmin\{\norm{\cG_{k+1}^x}^2 + \norm{\cG_{k+1}^y}^2:\ k=0,\ldots, t-1\}$ defined for all $t\geq 1$. Since $\{(\bar x_t,\bar y_t)\}_{t\geq 0}$ is a bounded sequence, it has at least one limit point, and any of its limit points is a stationary point of stationary point of the NCMC minimax problem in~\eqref{prob}.

\section{Addition experiments} \label{append:num}
Here we provide additional experiments on %Roubst DNN training and  
superquantile-based learning. 
%are shown in Section \ref{A-1} and \ref{A-2} respectively.
%\subsection{Roubst DNN training} \label{A-1}
%\subsection{Superquantile-based learning} \label{A-2}
%In the subsection, 
Indeed, we focus on distributionally robust optimization (DRO) over Riemannian manifold. Given a set of data samples $\{\xi_i\}_{i=1}^n$, the DRO over Riemannian manifold $\mathcal{M}$ can be written as the following minimax problem:
\begin{equation}
\min_{x \in \mathcal{M}} \max_{w \in \mathcal{S}} \left\{ \sum_{i=1}^n w_i \ell(x; \xi_i) -  \alpha \|w - \frac{\bm{1}}{n}\|^2 \right\}, \label{prob:6}
\end{equation}
where $\alpha > 0$ denotes the coefficient, $w = (w_1, \cdots, w_n)$, $\mathcal{S} = \left\{ w \in \mathbb{R}^n : \sum_{i=1}^n w_i = 1, \, w_i \geq 0 \right\}$. Here $\ell(x; \xi_i)$ denotes the loss function over the Riemannian manifold $\mathcal{M}$, which applies to many machine learning problems such as ICA \citep{ijcai2021p345}, dictionary learning \citep{sun2016complete}, neural network training \citep{huang2023gradient}, structured low-rank matrix learning \citep{deng2023decentralized}, among others. For example, the task of PCA can be cast on a Grassmann manifold. 

In the experiment, we use Stiefel manifold $\mathcal{M} = \mathrm{St}(r,d) = \left\{ X \in \mathbb{R}^{d \times r} \; : \; X^\top X = I_r \right\}$ on parameters $x$ of DNNs (convolution layers and linear layers), see Table \ref{tab:dnn_architecture} for details. Different algorithms are tested on CIFAR-10, CIFAR-100, STL-10, Fashion MNIST, and MNIST datasets.  We set $\tau_1 = \tau_2 = 10^{-3}$, $\beta =0.9$, $p = 1$, $\rho = 10, C = 1000$ for \sm{} with the same $\tau_1$ and $\tau_2$ for MGDA and RSGDA. 
% We use the same hyperparameter setting for RSGDA, MGDA from this task. The learning rate of $x$ and $w$ are both set as $10^{-3}$.
The batch size is 512, and the model is trained for 200 epochs. The results are listed in Figure \ref{fig:datasets}, where the primal loss denotes $\rv{F}(x)$.  It is shown that \sm{} not only converges the fastest but also has the highest test accuracy compared with other tested algorithms. Furthermore, the gradient norm and primal loss are also the lowest compared with the tested algorithms. The final test accuracy of the compared algorithms are list in Table \ref{tab:test_accuracy}. It is shown that \sm{} has the highest test accuracy compared with the other two algorithms.

\begin{table}[ht]
\centering
\small
\setlength{\tabcolsep}{8pt}
\caption{Test accuracy (\%) of different algorithms on five datasets after 200 epochs.}
\begin{tabular}{lccccc}
\hline
\textbf{Algorithm} & \textbf{MNIST} & \textbf{FashionMNIST} & \textbf{CIFAR-10} & \textbf{CIFAR-100} & \textbf{STL-10} \\
\hline
MGDA     & 99.28 & 91.85 & 74.42& 37.82 & 64.01 \\
RSGDA   & 99.26 & 90.95 & 74.95 & 38.12 & 63.81 \\
\sm{}  & \textbf{99.42} & \textbf{94.12} & \textbf{76.95} & \textbf{41.12} & \textbf{65.81} \\
\hline
\end{tabular}
\label{tab:test_accuracy}
\end{table}

\begin{figure}
    \centering
    \includegraphics[width=1.0\linewidth]{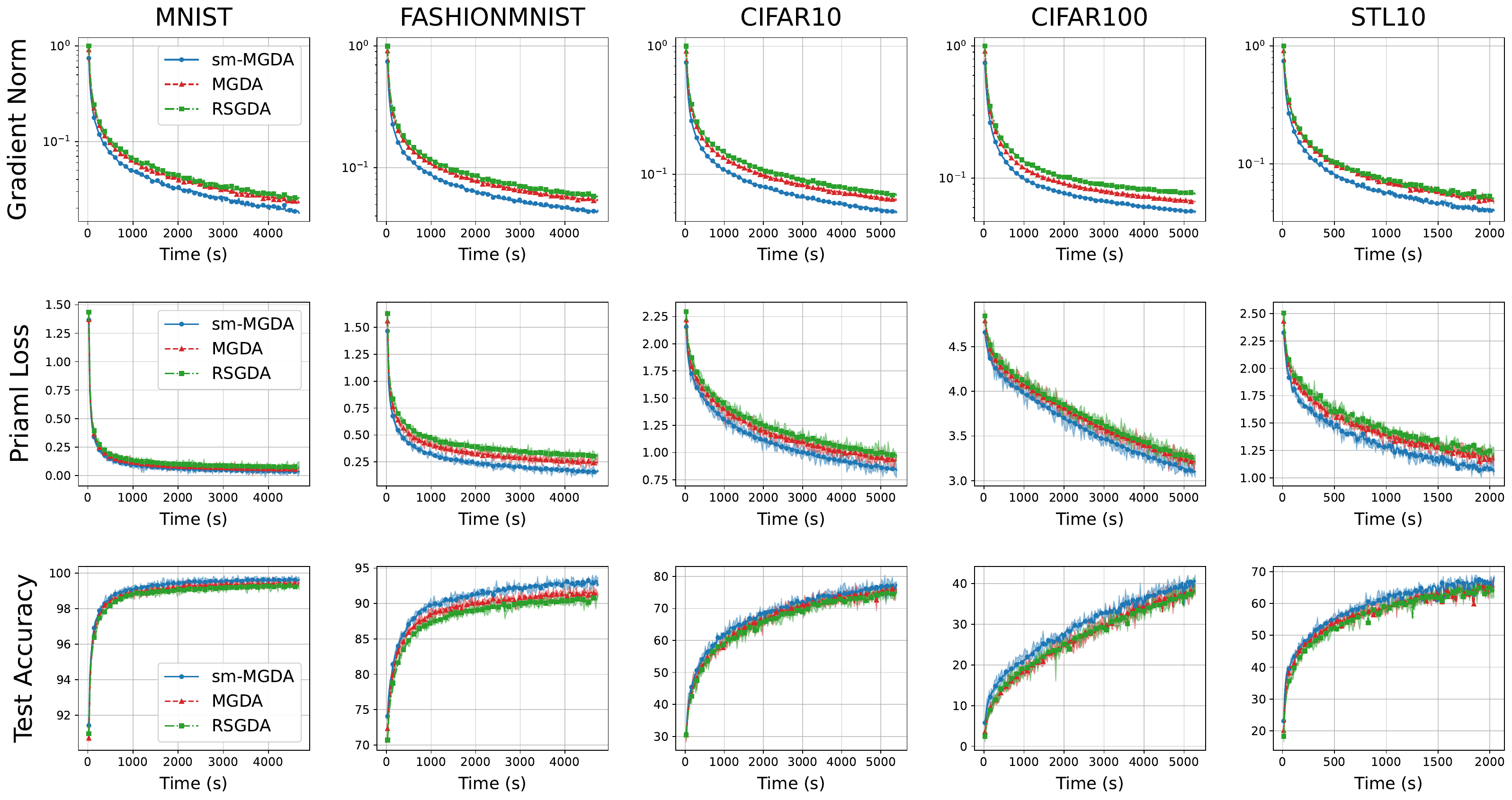}
    \caption{Primal loss, gradient norm, and test accuracy of tested algorithms over 3 runs. }
    % \hj{Use log scale in y axis, Multiple run with different seeds from the same initial point, 10/20 runs? Plot the mean and the range around it. Focus on FashionMNIST first}}
    \label{fig:datasets}
\end{figure}

Since Algorithm \ref{alg:sm-agda-stiefel} is a retration-free algorithm, the heatmaps of parameters $W^{\top}W$ across different layers of the model training by Algorithm \ref{alg:sm-agda-stiefel} after 200 epochs are shown in Figure \ref{fig:stiefel_constraints}, which demonstrates that the parameters of the models are indeed lies in the Stiefel manifold. The figures of manifold error with epoch for superquantile-based learning and robust DNN training task are shown in Figures \ref{fig:stiefel_constraints2} and \ref{fig:stiefel_constraints3}. respectively. It is shown that the manifold error decreases with the epochs, which validates our theoretical result.
\begin{figure}[h]
    \centering
    \includegraphics[width=1.0\linewidth]{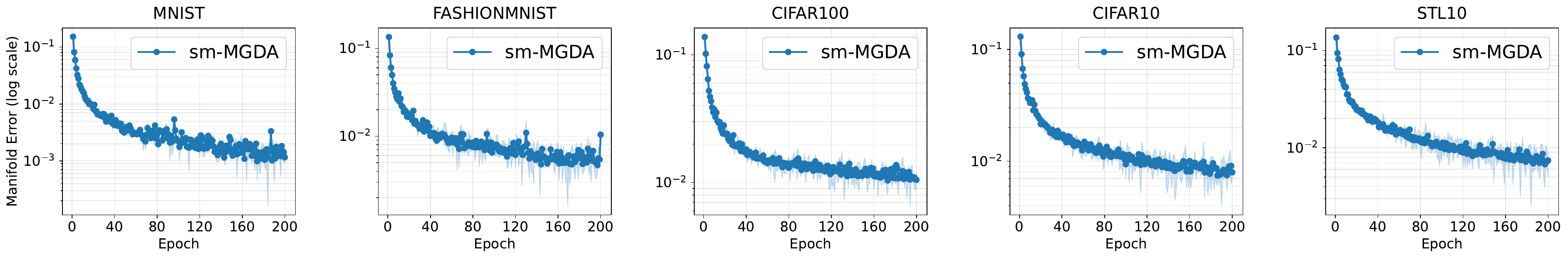}
    \caption{Manifold error of the model with epoch on superquantile-based learning problem.}
    \label{fig:stiefel_constraints2}
\end{figure}

\begin{figure}[h]
    \centering
    \includegraphics[width=1.0\linewidth]{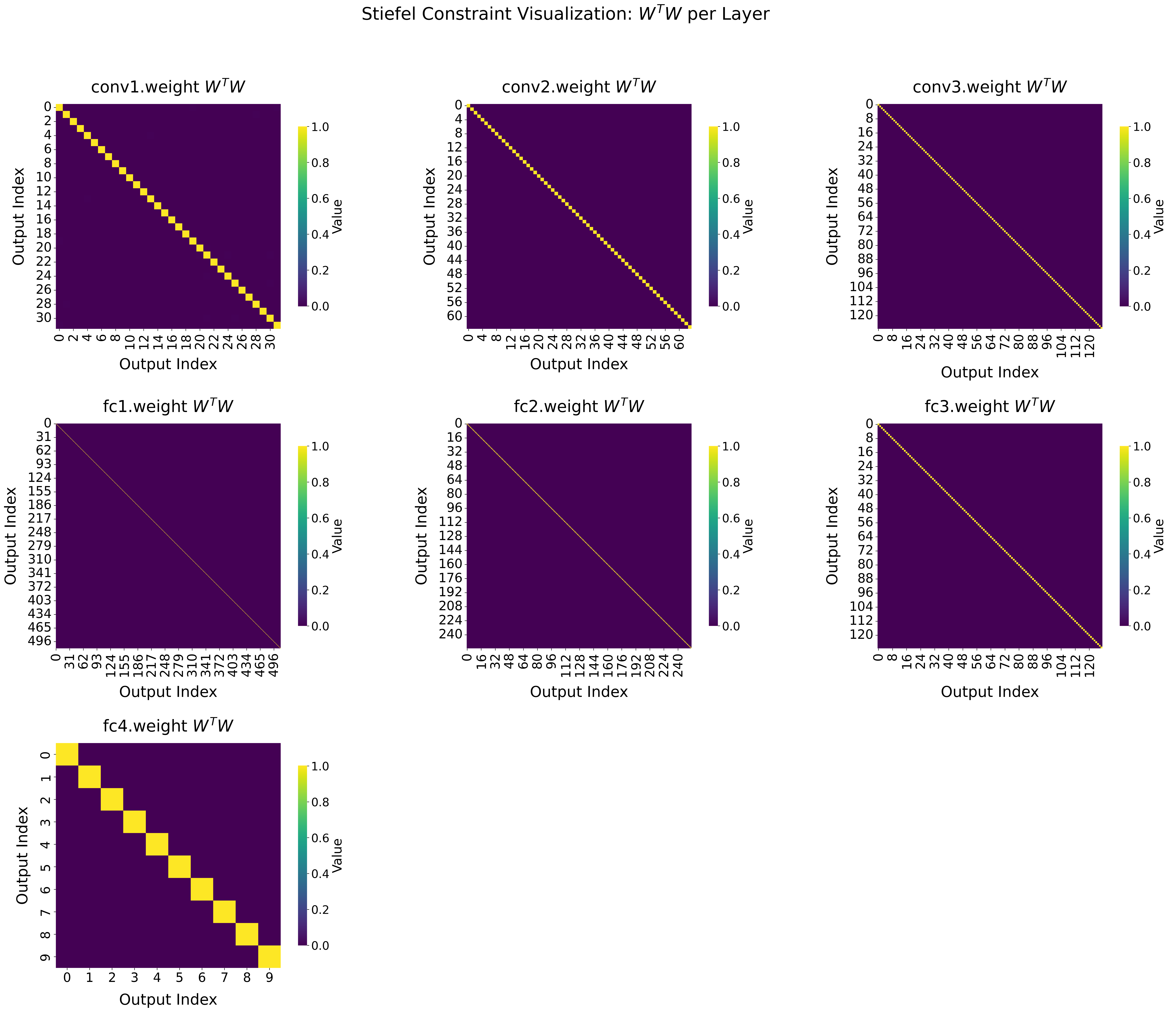}
    \caption{Heatmaps of $W^{\top}W$ across different layers (from left to right: Layer 1 to Layer 7) after training 200 epochs. The diagonal dominance in each block demonstrates the approximate satisfaction of Stiefel manifold constraints ($W^{\top}W \approx I$).}
    \label{fig:stiefel_constraints}
\end{figure}
\end{document}

%% file: macro.tex
\usepackage{array}
\usepackage{url}
\usepackage{amsmath,amssymb,amsbsy,amsthm}
\usepackage{paralist}
\usepackage{xcolor}
\usepackage{color}
\usepackage{graphicx}
\graphicspath{{./figs/}}
\usepackage{algorithm}
\usepackage{algorithmic}
\usepackage{comment}
\usepackage{wrapfig}
\usepackage{caption}
\usepackage{subcaption}
\usepackage{bm}
\usepackage{framed}
\usepackage{fancyhdr}
\usepackage{cleveref}
\usepackage{mathrsfs}
\usepackage{multirow}\usepackage{makecell}
\usepackage{pifont}
\newcommand{\xmark}{\ding{55}}  % ✗
\newcommand{\cmark}{\ding{51}} 

\newtheorem{thm}{Theorem}%[section] %(If you want theorem numbered
\newtheorem{defi}{Definition}%[section]

\newtheorem{assum}{Assumption}

\newtheorem{lemma}{Lemma} 

\newtheorem{remark}{Remark}

\newcommand{\R}{\mathbb{R}}

\def \reals   { \mathbb{R} }

\newcommand{\be}{\begin{equation}}
\newcommand{\ee}{\end{equation}}
\newcommand{\en}{\begin{equation*}}
\newcommand{\een}{\end{equation*}}
\newcommand{\eqn}{\begin{eqnarray}}
\newcommand{\eeqn}{\end{eqnarray}}
\newcommand{\bmat}{\begin{bmatrix}}
\newcommand{\emat}{\end{bmatrix}}
\newcommand{\btab}{\begin{tabular}}
\newcommand{\etab}{\end{tabular}}

\newcommand{\Mcal}{\mathcal{M}}
\newcommand{\Pcal}{\mathcal{P}}

\DeclareMathOperator*{\argmin}{\rm arg\,min}
\DeclareMathOperator*{\argmax}{\rm arg\,max}
\newcommand{\prox}{\operatorname{prox}}
\newcommand{\dom}{\operatorname{dom}}

\newcommand{\grad}{\operatorname{grad}}

\newlength{\imgwidth}
\newcommand{\revise}[1]{\textcolor{black}{#1}}

\newcommand{\hjn}[1]{\textcolor{black}{#1}}

\newboolean{twoColVersion}
\setboolean{twoColVersion}{false}
\newcommand{\twoCol}[2]{\ifthenelse{\boolean{twoColVersion}} {#1} {#2} }

\usepackage[framemethod=tikz]{mdframed}

\def\norm#1{\|#1\|}
\def\fprod#1{\langle #1\rangle}
\def\reals{\mathbb{R}}
\def\cM{\mathcal{M}}
\def\cO{\mathcal{O}}
\def\cP{\mathcal{P}}
\def\cX{\mathcal{X}}
\def\cY{\mathcal{Y}}
\def\sm{\texttt{sm-MGDA}}